\documentclass[aap]{imsart}

\RequirePackage{amsthm,amsmath,amsfonts,amssymb}
\RequirePackage[numbers]{natbib}
\RequirePackage[colorlinks,citecolor=blue,urlcolor=blue]{hyperref}
\RequirePackage{graphicx}

\startlocaldefs
\theoremstyle{plain}

\newtheorem{theorem}{Theorem}[section]
\newtheorem{lemma}[theorem]{Lemma}
\theoremstyle{remark}

\newtheorem{example}{Example}

\newtheorem{remark}[theorem]{Remark}

 \newcommand{\Ll}{\langle}
 \newcommand{\Rr}{\rangle}
 \newcommand{\Z}{\mathbb{Z}}

\newcommand{\Be}{\begin{equation}}
\newcommand{\Ees}{\end{equation*}}
\newcommand{\Bes}{\begin{equation*}}
\newcommand{\Ee}{\end{equation}}

\newcommand{\R}{\mathbb{R}}
\newcommand{\E}{\mathbb{E}}

\newcommand{\mcl}{\mathcal}

\newcommand{\dif}{\mathrm{d}}

\endlocaldefs

\begin{document}

\begin{frontmatter}
\title{A probability approximation framework: Markov process approach}
\runtitle{An approximation method}

\begin{aug}
\author[P. Chen]{\fnms{Peng}~\snm{Chen}\ead[label=e1]{chenpengmath@nuaa.edu.cn}},
\author[Q. M. Shao]{\fnms{Qi-Man}~\snm{Shao}\ead[label=e2]{shaoqm@sustech.edu.cn}}
\and
\author[L. Xu]{\fnms{Lihu}~\snm{Xu}\ead[label=e3]{lihuxu@um.edu.mo}}
\address[P. Chen]{College of Mathematics, Nanjing University of Aeronautics and Astronautics, Nanjing 211106, China\printead[presep={,\ }]{e1}}

\address[Q. M. Shao]{Department of Statistics and Data Science, SICM, NCAMS, Southern University of Science and Technology\printead[presep={,\ }]{e2}}

\address[L. Xu]{Department of Mathematics,
Faculty of Science and Technology,
University of Macau,
Av. Padre Tom\'{a}s Pereira, Taipa
Macau, China and UM Zhuhai Research Institute, Zhuhai, China.\printead[presep={,\ }]{e3}}
\end{aug}

\begin{abstract}
We view the classical Lindeberg principle in a Markov process setting to establish a probability
approximation framework by the associated It\^{o}'s formula and Markov operator.
As applications, we study the error bounds of the following three approximations: approximating a family of online stochastic gradient descents (SGDs) by a stochastic
differential equation (SDE) driven by multiplicative Brownian motion, Euler-Maruyama (EM) discretization for multi-dimensional Ornstein-Uhlenbeck stable process, and multivariate normal approximation. All these error bounds are in Wasserstein-1 distance.
\end{abstract}



\end{frontmatter}

\section{Introduction}

Lindeberg principle provides an elegant proof for the classical central limit theorem of the sum of independent
random variables \cite{Lin22}, it has been extensively applied to many research problems, see
\cite{Cha06,KM11,TV11,CCK14,BMP15,MP16,CCK17,CSZ17,KY17,WAP17,BCP18,CX19} and the references therein. In this paper, we
shall view the classical Lindeberg principle in a Markov process setting, and use the well developed tools in
stochastic analysis, such as It\^{o}'s formula and infinitesimal generator, to establish a probability approximation framework.

In order to interpret our method, we first briefly recall the classical Lindeberg principle by the following
example.  Let $(\xi_{i})_{i \ge 1}$ be a sequence of independent and identically distributed (i.i.d.) $\mathbb{R}$-valued
random variables with $\mathbb{E} \xi_{i}=0$, $\mathbb{E} \xi^{2}_{i}=1$ and $\mathbb{E}|\xi_{i}|^{3}<\infty$. Let $(\zeta_{i})_{i \ge 1}$
be a sequence of independent standard normal distributed random variables and it is well known that
$\frac{\zeta_{1}+\cdots+\zeta_{n}}{\sqrt{n}}$ is a standard normal distributed random variable for any $n \in \mathbb{N}$.
Write $\xi_{n,i}=\frac{\xi_{i}}{\sqrt n}$, $\zeta_{n,i}=\frac{\zeta_{i}}{\sqrt n}$, and denote
$$X_{n}=\zeta_{n,1}+...+\zeta_{n,n}, \ \ \ \ \ \ \ Y_{n}=\xi_{n,1}+...+\xi_{n,n}.$$  Further denote $Z_{0}=X_{n}$
and
$Z_{i}=Z_{i-1}-\zeta_{n,i}+\xi_{n,i}$ for $i \ge 1$, we easily see that $Z_{i}$ is obtained by swapping
$\zeta_{n,i}$ in $Z_{i-1}$ with $\xi_{n,i}$.
For any bounded 3rd order differentiable function $h$, we have
\begin{eqnarray} \label{e:Lindeberg0}
|\mathbb{E}[h(X_{n})]-\mathbb{E}[h(Y_{n})]| & \le & \sum_{i=1}^{n} \left|\mathbb{E}[h(Z_{i-1})]-\mathbb{E}[h(Z_{i})]\right| \\
&\le & C n^{-3/2}\|h'''\|\sum_{i=1}^{n} [\mathbb{E} |\xi_{i}|^{3}+\mathbb{E} |\zeta_{i}|^{3}] \le C n^{-1/2} \|h'''\|, \nonumber
\end{eqnarray}
where $\|.\|$ is the uniform norm of continuous function and the second inequality is obtained by a 3rd order
Taylor expansion.

Let us now explain the Lindeberg's proof from a perspective of Markov process and view the above swap trick as a
comparison of two Markov processes. Denote $X_{0}=0, X_{i}=\zeta_{n,1}+...+\zeta_{n,i}$ and $Y_{0}=0,
Y_{i}=\xi_{n,1}+...+\xi_{n,i}$ for $i \ge 1$, it is clear that $(X_{i})_{0 \le i \le n}$ and $(Y_{i})_{0 \le i \le
n}$ are both Markov processes. Formally, let $j \ge i \ge 1$, denote by $X_{j}(i,x)$ the random variable $X_{j}$ given
$X_{i}=x \in\mathbb{R}$, i.e.,
$X_{j}(i,x)=x+\zeta_{n,i+1}+...+\zeta_{n,j}$, it is obvious $X_{j}(i,X_{i})=X_{j}$ for $j \ge i$. Similarly, we
define $Y_{j}(i,y)$ for $j \ge i$.
It is easy to see that $Z_{j}=X_{n}(j, Y_{j})=X_{n}(j, Y_{j}(j-1,Y_{j-1}))$ and
$Z_{j-1}=X_{n}(j-1,Y_{j-1})=X_{n}(j,X_{j}(j-1,Y_{j-1}))$ for each $1 \le j \le n$, thus
\begin{equation*}
\begin{split}
&|\mathbb{E}[h(X_{n})]-\mathbb{E}[h(Y_{n})]|\\
& \le \sum_{j=1}^{n} \left|\mathbb{E}[h(Z_{j-1})]-\mathbb{E}[h(Z_{j})]\right| \\
&=\sum_{j=1}^{n} \left|\mathbb{E}[h(X_{n}(j,X_{j}(j-1,Y_{j-1}))]-\mathbb{E}[h(X_{n}(j,Y_{j}(j-1,Y_{j-1})))]\right|.
\end{split}
\end{equation*}
A rigorous proof will be given in Section \ref{mainchapter} below with the help of the Chapman-Kolmogorov equation and the time homogeneity. Notice that
$\mathbb{E}[h(X_{n}(j,X_{j}(j-1,Y_{j-1}))]$ and $\mathbb{E}[h(X_{n}(j,Y_{j}(j-1,Y_{j-1})))]$ are the functions of
$X_{j}(j-1,Y_{j-1})$ and $Y_{j}(j-1,Y_{j-1})$, respectively, and compare these two new functions
rather than directly compute $|\mathbb{E}[h(Z_{i-1})]-\mathbb{E}[h(Z_{i})]|$ in Lindeberg principle. Because $(X_{i})_{0 \le i \le
n}$ can be embedded into a Brownian motion $(B_t)_{0 \le t \le 1}$ which has a smoothening effect, we expect that
It\^{o}'s formula and the semigroup theory of Brownian motion will make $\mathbb{E}[h(X_{n}(j,X_{j}(j-1,Y_{j-1}))]$ and
$\mathbb{E}[h(X_{n}(j,Y_{j}(j-1,Y_{j-1})))]$ have better regularity than $h$, see more details in {Subsection \ref{proofCLT}}. Since
the above procedure only depends on Markov property, this perspective of viewing Lindeberg principle can be
extended to other Markov processes.

The novelty of this paper is the following two aspects. (1) We view the procedure of the classical Lindeberg
principle as a special Markov process and extend this point of view to general \emph{Markov process} setting, using
\emph{It\^{o}'s formula} of Markov process and \emph{Markov operator} (see, e.g.,
\cite{Gar85,BK02,EK09,KP13,Oks13}), we establish a probability approximation framework. Chatterjee \cite{Cha06} extended Lindeberg principle to a family of dependent random variables, and established a general approximation error bound from which he identified the limiting spectral
distribution of Wigner matrices with exchangeable entries. It is obvious to see from the SGD approximation below that our approximation framework also works for dependent random variables. (2) We apply our
framework to three applications in the classical Wasserstein-1 distance: approximating online
stochastic gradient descent (SGD) in machine learning by a stochastic differential equation (SDE), bounding the error between a SDE with $\alpha$-stable noise and its Euler-Maruyama (EM) discretization, and normal approximation.

For the first application, there have been many results on approximating SGD by a SDE, see for instance \cite{TTV16,LTW17,AN19,FGLLL19,HLLL19,LTW19,BS20,FDBD20} and the references therein. To the best of our knowledge, most of the known approximation results are about the error bounds over a family of test functions with bounded high order derivatives,
from which it is not easy to obtain an approximation error bound in a probability metric.
By restricting an SGD in a neighborhood of a local minimum and solving a Kolmogorov backward equation,
Feng et al. \cite{FGLLL19} studied locally approximating the SGD before it jumps to another minimum. When their test functions have bounded $k$-th order
derivatives with $k=0,...,4$, the error bound is of order $\eta$ ($\eta$ is the learning rate), whereas the bound is improved to be of order $\eta^2$ as the test functions additionally have bounded $5$-th and $6$-th order derivatives. In \cite{LTW19}, Li et al. proposed stochastic modified equation, a SDE with multiplicative noise, to approximate SGD, their approximation error is defined through a family of test functions which has high order derivatives and a certain growth condition. If the test function is Lipschitz function family, it seems to us that their result cannot provide a convergence rate, {see \cite[Definition 1 and Theorem 3]{LTW19}}.
In contrast, we will use our framework to get an explicit error bound between SGD and the associated SDE in
the classical Wasserstein-1 distance, where our test function family is Lipschitz.

In the second application, we consider the EM discretization of $\alpha$-stable Ornstein-Uhlenbeck (OU) process with a constant step size $\eta$, which leads to a heavy tailed AR(1) time series without second moment, {\cite{Cli83,DKL92,Res97,CLW12}.} Using our framework, we establish an error bound of EM discretization and obtain a rate $\eta^{\frac{2-\alpha}{\alpha}}$ for $\alpha \in (1,2)$. It seems that there are not many results about EM discretization for SDE with $\alpha$-stable noise, see \cite{JMW96,WY07,TA18,Liu19}, most of them are about bounding
strong approximations in a finite time interval. As the time tends to $\infty$, these bounds blow up. The bound that we obtained is uniform with respect to the time, this means that our bound still holds true even the time tends to $\infty$. Note that the discrete AR(1) time series are not independent.

The third application is normal approximations, which have recently been intensively
studied by Stein's method, see for instance, \cite{CM08,RR09,CGS10,VV10,FSX19,Son20} and the references therein. In particular, Chatterjee and Meckes used an exchangeable pair method to obtain a bound for multivariate normal approximation in \cite{CM08}, because the test functions in \cite{CM08} have bounded first and second derivatives, their bound cannot derive a convergence in Wasserstein-1 distance. By a direct calculation under our framework, we get a $\frac{1}{\sqrt n}$ convergence rate up to $\log n$
correction for multivariate normal approximation, which was established in \cite{VV10,FSX19}. For Stein's method, we refer the reader to
\cite{Xu19,CNX20,CNXYZ19,CNXY19,BU20} for stable approximation,\cite{Che75,BH84,Bar88,AGG90,BHJ92} for Poisson
approximation and
\cite{Bar90,BB92,BCL92,Loh92,NP09,Gur14,GM15,BD17,BDF17,Kas17,AH19} for other approximations.

Besides the applications to online SGD, EM discretization and normal approximation addressed in this paper, we hope that our new method
can also be applied to many other probability approximations, e.g., diffusion approximation with constant step size
and so on. We will study these research problems in the future paper.

In this paper, we focus on the approximation problems in Wasserstein-1 distance. However, it is clear to see
from Theorem \ref{mainthm} that our approximation method also works for {\bf other} metrics. For instance, if we
consider bound measurable function $h$, then the approximation will be in total variation metric.

The organization of this paper is as follows. We shall introduce our probability approximation framework and main theorem in next section. In Section \ref{application}, we will give
the results about the three applications to SGD, EM discretization and normal approximation, where proofs
are given in {Subsections} \ref{proofsgd}-\ref{proofCLT}, respectively. Appendixes \ref{s:A1} and \ref{Mal} are devoted to
proving some  auxiliary lemmas about the first application, while Appendixes \ref{stable application} and
\ref{normal spplication} provide the proofs of auxiliary lemmas about the second and third applications,
respectively.
\\

{\bf Notations.} We end this section by introducing some notations, which will be frequently used in sequel. The
inner product of $x,y\in\mathbb{R}^{d}$ is denoted by $\langle x,y\rangle$ and the Euclidean metric is denoted by
$|x|$.

Let $\mu$ and $\nu$ be two probability distributions on $\mathbb{R}^{d}$, their \emph{Wasserstein-1 distance} is defined as
$$d_W(\mu,\nu)=\sup_{h\in {\rm Lip(1)}} \big|\mathbb{E}[h(X)]-\mathbb{E}[h(Y)]\big|,$$
where ${\rm Lip(1)}:=\{h:\mathbb{R}^{d}\rightarrow\mathbb{R}: |h(x)-h(y)| \le {|x-y|}\}$, $X$ and $Y$ are two random
variables with distributions $\mu$ and $\nu$, respectively.

For a random variable $X$, we denote by $\mathcal{L}(X)$ its probability law. In addition, for any d-dimensional random
vectors $\xi_{1},\xi_{2},$ we call $\xi_{1}\stackrel{\rm d}{=}\xi_{2}$ if for any $A\in\mathcal{B}(E),$ the Borel
set of $E$, we have
\begin{align*}
\mathbb{P}(\xi_{1}\in A)=\mathbb{P}(\xi_{2}\in A).
\end{align*}

Let $\mathcal{C}(\mathbb{R}^{d},\mathbb{R})$ denote the collection of all continuous functions
$f:\mathbb{R}^{d}\rightarrow\mathbb{R}$ and $\mathcal{C}^{k}(\mathbb{R}^{d},\mathbb{R})$, $k\geq1$,  denote the
collection of all $k$-th order continuously differentiable functions. For $f \in \mathcal{C}^3(\mathbb{R}^d,\mathbb{R})$ and $v, v_1, v_2,v_{3},x \in\mathbb{R}^d$, the directional derivative $\nabla_v f(x)$, $\nabla_{v_2}\nabla_{v_1} f(x)$ and $\nabla_{v_{3}}\nabla_{v_2}\nabla_{v_1} f(x)$ are  defined by
\ \ \
\begin{equation*}
\begin{split}
& \nabla_v f(x)\ = \ \lim_{\varepsilon\rightarrow 0} \frac{f(x+\varepsilon v)-f(x)}{\varepsilon}, \\
\end{split}
\end{equation*}
\begin{equation*}
\begin{split}
&\nabla_{v_2} \nabla_{v_1} f(x)\ = \ \lim_{\varepsilon\rightarrow0} \frac{\nabla_{v_1}f(x+\varepsilon v_2)-\nabla_{v_1}f(x)}{\varepsilon},
\end{split}
\end{equation*}
\begin{equation*}
\begin{split}
&\nabla_{v_3}\nabla_{v_2} \nabla_{v_1} f(x)\ = \ \lim_{\varepsilon\rightarrow 0} \frac{\nabla_{v_2}\nabla_{v_1}f(x+\varepsilon v_3)-\nabla_{v_2}\nabla_{v_1}f(x)}{\varepsilon},
\end{split}
\end{equation*}
respectively. Let $\nabla f(x)\in\mathbb{R}^d$ and $\nabla^2 f(x)\in\mathbb{R}^{d \times d}$ denote the gradient and the Hessian of $f$, respectively.
It is known that
$\nabla_v f(x) =\langle \nabla f(x), v\rangle$ and
$\nabla_{v_2} \nabla_{v_1} f(x)=\langle \nabla^2 f(x), v_1 v^{\rm T}_2\rangle_{{\rm HS}}$, where ${\rm T}$ is the transpose operator and $\langle A, B\rangle_{{\rm HS}}:=\sum_{i,j=1}^d A_{ij} B_{ij}$ for $A, B \in\mathbb{R}^{d \times d}$. We define the operator norm of $\nabla^{2} f(x)$ by
\begin{align*}
\|\nabla^{2} f(x)\|_{\rm op}=& \sup_{|v_{1}|,|v_{2}|=1} |\nabla_{v_{2}} \nabla_{v_{1}} f(x)|, \qquad \|\nabla^{2} f\|_{{\rm op}, \infty}=& \sup_{x \in \mathbb{R}^{d}} \|\nabla^{2} f(x)\|_{\rm op}.
\end{align*}
We often drop the subscript "{\rm op}" in the definitions above and simply write $\|\nabla^{2} f(x)\|=\|\nabla^{2} f(x)\|_{\rm op}$ and  $\|\nabla^{2} f(x)\|_{\infty}=\|\nabla^{2} f(x)\|_{{\rm op}, \infty}$ if there is no confusion. Similarly we define
\begin{align*}
\|\nabla^{3} f(x)\|_{\rm op}=& \sup_{|v_{i}|=1, i=1,2,3} |\nabla_{v_{3}} \nabla_{v_{2}} \nabla_{v_{1}} f(x)|
\end{align*}
and $\|\nabla^{3} f\|_{{\rm op}, \infty}$ and the short-hand notations $\|\nabla^{3} f(x)\|$ and $\|\nabla^{3} f\|_{\infty}$.

Given a matrix $A \in\mathbb{R}^{d \times d}$, its Hilbert-Schmidt norm is
$\|A\|_{{\rm HS}} \ = \ \sqrt{\sum_{i,j=1}^{d} A^{2}_{ij}} \ = \ \sqrt{{\rm Tr} (A^{\rm T} A)}$ and its operator norm is
$\|A\|_{\rm op} \ = \ \sup_{|v|=1}|A v|$. We have the following relations:
\begin{equation}\label{op}
\|A\|_{\rm op} \ = \ \sup_{|v_{1}|,|v_{2}|=1} |\langle A, v_{1} v^{\rm T}_{2}\rangle_{\text{HS}}|, \ \ \ \ \ \ \ \ \ \|A\|_{\rm op} \ \le \ \|A\|_{{\rm HS}} \ \le \ \sqrt d \|A\|_{\rm op}.
\end{equation}

Moreover, $\mathcal{C}_{b}(\mathbb{R}^{d_{1}},\mathbb{R}^{d_{2}})$ with $d_{1},d_{2}\in\mathbb{N}$ denotes the set of all bounded continuous
functions from $\mathbb{R}^{d_{1}}$ to $\mathbb{R}^{d_{2}}$ with the supremum norm defined by
\begin{align*}
\|f\|=\sup_{x\in\mathbb{R}^{d_{1}}}|f(x)|.
\end{align*}
Denote by $C_{p_{1},\cdots,p_{k}}$ some positive number depending on $k$ parameters, $p_{1},\cdots,p_{k},$ whose
exact values can vary from line to line.

\section{The framework and main theorem}\label{mainchapter}
Let $E$ be a Polish space.
Let $(X_t)_{t \ge 0}$ be a continuous time
homogeneous $E$-valued {Markov} process, and let $(Y_k)_{k \in\mathbb{Z}_+}$ be a discrete time homogeneous $E$-valued
{Markov} process (note $\mathbb{Z}_+=\{0\}\cup \mathbb{N}$).
If $X_0=x \in \mathbb{R}^{d}$, we denote the Markov process $(X_{t})_{t \ge 0}$ by $(X^x_t)_{t \ge 0}$ to stress it starts from $x \in E$.
Similarly for the notation $(Y^{y}_k)_{k \in\mathbb{Z}_+}$ for $y \in E$.

Notice that the process $(X_t)_{t \ge 0}$ is a time homogeneous $E$-valued {Markov} process, for any $0\leq s\leq t<\infty$, $x\in E$ and $B\in\mathcal{B}(E)$, the Borel sets of $E$, denote the transition probability function of $(X_t)_{t \ge 0}$ by $P_{s,t}(x,B)$, that is,
\begin{align}\label{Xtran}
P_{s,t}(x,B)=\mathbb{P}\left(X_{t}\in B|X_{s}=x\right).
\end{align}
Then $P_{s,t}(x,B)$ satisfies the following properties (see, e.g., \cite[Section 10 of Chapter 2]{Sat99}):

(1) it is a probability measure as a mapping of $B$ for any fixed $x$;

(2) it is measurable in $x$ for any fixed $B$;

(3) $P_{s,s}(x,B)=\delta_{x}(B)$ for $s\geq0$, that is, when $x\in B$, $\delta_{x}(B)=1$, otherwise $\delta_{x}(B)=0$;

(4) (Chapman-Kolmogorov equation) it satisfies
\begin{align}\label{XC-K}
\int_{E}P_{s,u}(x,dy)P_{u,t}(y,B)=P_{s,t}(x,B) \quad {\rm for}\quad 0\leq s\leq u\leq t;
\end{align}

(5) (Time homogeneity)$P_{s+h,t+h}(x,B)$ does not depend on $h$.

According to the time homogeneity (5), we have
\begin{align*}
P_{t}(x,B)=P_{s,s+t}(x,B) \quad {\rm for}\quad s\geq0,
\end{align*}
then the Chapman-Kolmogorov equation (4) can be written as
\begin{align}\label{YC-K}
\int_{E}P_{s}(x,dy)P_{t}(y,B)=P_{s+t}(x,B), \quad {\rm for}\quad s\geq0, t\geq0.
\end{align}
For the process $(Y_k)_{k \in\mathbb{Z}_+}$, for any $y\in E$, $B\in\mathcal{B}(E)$, $i,j\in\mathbb{Z}^{+}$ and $i\leq j$, we denote the corresponding transition probability function by $Q_{i,j}(y,B)$, that is,
\begin{align}\label{Ytran}
Q_{i,j}(y,B)=\mathbb{P}\left(Y_{j}\in B|Y_{i}=x\right).
\end{align}
According to the time homogeneity, we denote
\begin{align*}
Q_{j}(y,B)=Q_{i,i+j}(y,B)\quad {\rm for}\quad i\in\mathbb{Z}^{+}.
\end{align*}

For $(X_{t})_{t\geq 0}$, its infinitesimal generator is defined as
\begin{align*}
\mathcal{A}^{X}f(x):=\lim_{t\rightarrow 0}\frac{\mathbb{E}f(X_{t}^{ x})-f(x)}{t}, \ \ \ \ f \in\mathcal{D}(\mathcal{A}^{X}),
\end{align*}
where $\mathcal{D}(\mathcal{A}^{X})$ is the domain of the operator $\mathcal{A}^{X}$, whose exact form varies according to the concrete applications at hand.  For $(Y_k)_{k \in\mathbb{Z}_+}$, its infinitesimal generator is defined as
\begin{align*}
\mathcal{A}^{Y}f(y):=\mathbb{E}[f(Y_{1}^{y})-f(y)], \ \ \ \ f \in\mathcal{D}(\mathcal{A}^{Y}),
\end{align*}
where $\mathcal{D}(\mathcal{A}^{Y})$ is the domain of the operator $\mathcal{A}^{X}$, whose exact form varies according to the concrete applications at hand.

In order to avoid entering the semigroup theory in which we need to figure out function spaces and operator domains to justify our method, we use the concept of full generator family, which is usually easy to be verified by It\^{o}'s formula in practice.  More precisely, for a function $f$, we call $(f, \mathcal{A}^X f)$ belongs to the full generator family of $X_t$ if
\begin{align}\label{Ito}
\mathcal{M}_t:=f(X^{x}_{t})-f(x)-&\int_0^t\mathcal{A}^X f(X^{x}_{s}) \dif s, \quad \quad \quad \ \ t>0,
\end{align}
is a martingale with $\mathbb{E}[\mathcal{M}_{t}]=0$ for all $t\geq0$, see \cite[Chapter 4]{EK09} for more details. In practice, it is easy for us to verify that a function belongs to a full generator family by It\^{o}'s formula. For a more thorough discussion on the infinitesimal generators and
It$\hat{o}$'s formula, we refer the reader to \cite[Chapter 1 and Chapter 4]{EK09}, \cite[Chapter
IX]{Yos88}, \cite[Chapter 4]{Oks13} and the references therein.

Our first main result is a framework of comparing the distributions of $(X_t)_{t \ge 0}$ and $(Y_k)_{k \in\mathbb{Z}_+}$,
which can be fitted into many probability approximations arising in concrete applications. The key ingredients of
the proof are Markov semigroup, (\ref{Ito}) and infinitesimal generator in stochastic analysis.
\begin{theorem}\label{mainthm}
	Let $N\geq2$ be a natural number and let $h: E \rightarrow \R$ be a measurable function such that:
	(1). $\E |h(X^x_t)|<\infty$ and $\E |h(Y^y_k)|<\infty$ for all $x \in E, y \in E$, $t \le N$ and $k \le N$;
	(2). the function $u_k(x):=\E h(X^x_k)$ for $k \ge 1$ satisfies
	$\E |\mcl A^X u_k(Y_j)|<\infty$ and $\E |\mcl A^X u_k(X_t^{Y_j})|<\infty$ for all $1 \le j, k \le N$ and $0 \le
t \le 1$; (3). $(u_k, \mcl A^X u_k)$ belongs to the full generator family of $X_t$ for $1 \le k \le N$.
Then, for any $x\in E$,
\begin{eqnarray}\label{orign}
 \mathbb{E}h(X_{N}^{x})-\mathbb{E}h(Y_{N}^{x}) &=&\sum_{j=1}^{N}\left[\E u_{N-j}(X^{Y_{j-1}}_{1})-\E u_{N-j}(Y_{1}^{Y_{j-1}})\right],
 \end{eqnarray}
where, for simplicity, we write $Y_{j-1}=Y^x_{j-1}$ for all j.
In practice, we often rewrite (\ref{orign}) in the following form and then estimate each term on its right hand:
\begin{eqnarray}  \label{e:mainnew}
\mathbb{E}h(X^x_{N})-\mathbb{E}h(Y^x_{N}) &=&\mcl{I}_h+\mcl{II}_h,
 \end{eqnarray}
 where
 $$\mcl I_h=\sum_{j=1}^{N-1}
 \int_{0}^{1}\big[\mathbb{E}\mathcal{A}^{X}u_{N-j}(X_{s}^{Y_{j-1}})-\mathbb{E}\mathcal{A}^{Y}u_{N-j}(Y_{j-1})\big]\dif s,$$
 $$\mcl{II}_h=\mathbb{E}\big[h\big(X_{1}^{Y_{N-1}}\big)-h(Y_{N-1})\big]+\mathbb{E}\big[h(Y_{N})-h(Y_{N-1})\big].$$
In particular,
\begin{eqnarray}  \label{e:mainnew-1}
 d_W(\mcl L(X_{N}),\mcl L(Y_{N})) &\le &\sup_{h\in{\rm Lip(1)}} (|\mcl{I}_h|+|\mcl{II}_h|).
 \end{eqnarray}
\end{theorem}

\begin{remark}
In the above theorem, for $1\leq j\leq N$, the expectation $\mathbb{E}\left[u_{N-j}\left(X_{1}^{Y_{j-1}}\right)\right]$ actually means the following:
\begin{align*}
\mathbb{E}\left[u_{N-j}\left(X_{1}^{Y_{j-1}}\right)\right]=\int_{E}Q_{j-1}(x,\dif y)\int_{E}P_{1}(y,\dif z)\int_{E}h(u)P_{N-j}(z,\dif u).
\end{align*}
\end{remark}

\begin{proof}
In the proof, we will often use Chapman-Kolmogorov equation and the following relation: for all $x \in E$ and $j \ge i$,
\begin{equation} \label{e:Uj-i}
u_{j-i}(x)=\int_E h(y) P_{i,j}(x, \dif y),
\end{equation}
where we have used the definition of $u_k(.)$ and the time homoegeneous property.

For $N\geq 2$, by (\ref{Xtran}) and (\ref{XC-K}), one can write
\begin{align*}
\mathbb{E}h(X_{N}^{x})=&\int_{E}h(y)P_{0,N}(x,\dif y)\\
=&\int_{E}P_{0,1}(x,\dif z_{1})\int_{E}h(y)P_{1,N}(z_{1},\dif y) \\
=&\int_{E}u_{N-1}(z_{1})P_{0,1}(x,\dif z_{1}),
\end{align*}
where the last equality is by \eqref{e:Uj-i},
therefore,
\begin{align} \label{e:UNExp1}
\mathbb{E}h(X_{N}^{x})=&\int_{E}u_{N-1}(z_{1})P_{0,1}(x,\dif z_{1})-\int_{E}u_{N-1}(z_{1})Q_{0,1}(x,\dif z_{1})\nonumber\\
&+\int_{E}u_{N-1}(z_{1})Q_{0,1}(x,\dif z_{1}) \nonumber \\
=&\mathbb{E}\left[u_{N-1}(X_{1}^{x})\right]-\mathbb{E}\left[u_{N-1}(Y_{1}^{x})\right]+\int_{E}u_{N-1}(z_{1})Q_{0,1}(x,\dif z_{1}).
\end{align}
By \eqref{e:Uj-i} and Chapman-Kolmogorov equation, we further have
\begin{align*}
\int_{E}u_{N-1}(z_{1})Q_{0,1}(x,\dif z_{1})
=&\int_{E}Q_{0,1}(x,\dif z_{1})\int_{E}h(y)P_{1,N}(z_{1},\dif y)\\
=&\int_{E}Q_{0,1}(x,\dif z_{1})\int_{E}P_{1,2}(z_{1},\dif z_{2})\int_{E}h(y)P_{2,N}(z_{2},\dif y)\\
=&\int_{E}Q_{0,1}(x,\dif z_{1})\int_{E}u_{N-2}(z_{2})P_{1,2}(z_{1},\dif z_{2}).
\end{align*}
By a similar argument with (\ref{XC-K}), the time homogeneity and (\ref{Xtran}), we have
\begin{align*}
&\int_{E}Q_{0,1}(x,\dif z_{1})\int_{E}u_{N-2}(z_{2})P_{1,2}(z_{1},\dif z_{2}) \\
=&\int_{E}Q_{0,1}(x,\dif z_{1})\int_{E}u_{N-2}(z_{2})P_{1,2}(z_{1},\dif z_{2})-\int_{E}Q_{0,1}(x,\dif z_{1})\int_{E}u_{N-2}(z_{2})Q_{1,2}(z_{1},\dif z_{2})\\
&+\int_{E}Q_{0,1}(x,\dif z_{1})\int_{E}u_{N-2}(z_{2})Q_{1,2}(z_{1},\dif z_{2}) \\
=&\mathbb{E}\left[u_{N-2}\left(X_{1}^{Y_{1}}\right)\right]-\mathbb{E}\left[u_{N-2}\left(Y_{1}^{Y_{1}}\right)\right]
+\int_{E}u_{N-2}(z_{2})Q_{0,2}(x,\dif z_{2}),
\end{align*}
where the last equality is by Chapman-Kolmogrov equation and
the following observations:
\begin{align*}
\int_{E}Q_{0,1}(x,\dif z_{1})\int_{E}u_{N-2}(z_{2})P_{1,2}(z_{1},\dif z_{2})
&=\int_{E}\mathbb{E}\left[u_{N-2}\left(X_{1}^{z_{1}}\right)\right]Q_{0,1}(x,\dif z_{1})\\
&=\mathbb{E}\left[u_{N-2}\left(X_{1}^{Y_{1}}\right)\right],
\end{align*}
\begin{align*}
\int_{E}Q_{0,1}(x,\dif z_{1})\int_{E}u_{N-2}(z_{2})Q_{1,2}(z_{1},\dif z_{2})&=\int_{E}\mathbb{E}\left[u_{N-2}\left(Y_{1}^{z_{1}}\right)\right]Q_{0,1}(x,\dif z_{1})\\
&=\mathbb{E}\left[u_{N-2}\left(Y_{1}^{Y_{1}}\right)\right].
\end{align*}
Hence, we have
\begin{align*}
\int_{E}u_{N-1}(z_{1})Q_{0,1}(x,\dif z_{1})=&\mathbb{E}\left[u_{N-2}\left(X_{1}^{Y_{1}}\right)\right]-\mathbb{E}\left[u_{N-2}\left(Y_{1}^{Y_{1}}\right)\right]\\
&+\int_{E}u_{N-2}(z_{2})Q_{0,2}(x,\dif z_{2}),
\end{align*}
where $Y_1=Y_1^x$.

By the same argument, we can show that for all $i=1,2,...N-1$,
\begin{align}
&\int_{E}u_{N-i}(z_{i})Q_{0,i}(x,\dif z_{i})\nonumber\\
=&\mathbb{E}\left[u_{N-i-1}\left(X_{1}^{Y_{i}}\right)\right]-\mathbb{E}\left[u_{N-i-1}\left(Y_{1}^{Y_{i}}\right)\right]
+\int_{E}u_{N-i-1}(z_{i+1})Q_{0,i+1}(x,\dif z_{i+1}),
\end{align}
where $Y_i=Y_i^x$.  Combining these relations with \eqref{e:UNExp1} and noticing $Y_0^x=x$, we obtain
\begin{align} \label{e:UNExpN}
\mathbb{E}h(X_{N}^{x})=&\sum_{i=1}^N\left(\mathbb{E}\left[u_{N-i}(X_{1}^{Y_{i-1}})\right]-\mathbb{E}\left[u_{N-i}(Y_{1}^{Y_{i-1}})\right]\right)+\int_{E}u_{0}(z_N)Q_{0,N}(x,\dif z_{N}).
\end{align}
Noticing $u_0=h$, this immediately implies \eqref{orign}.

Let us now calculate each term in the sum on the right hand side. For $1\leq j\leq N$, we have
\begin{align*}
&\E u_{N-j}(X^{Y_{j-1}}_{1})-\E u_{N-j}(Y_{1}^{Y_{j-1}})\\
=&\E \left[u_{N-j}(X^{Y_{j-1}}_{1})-u_{N-j}(Y_{j-1})
  	\right]-\E\left[u_{N-j}(Y_{1}^{Y_{j-1}})-u_{N-j}(Y_{j-1})\right].
\end{align*}

When $1\leq j\leq N-1$, by the condition (3), we have
\begin{eqnarray*}\label{1}
	u_{N-j}(X^{Y_{j-1}}_{1})-u_{N-j}(Y_{j-1})&=&\int_0^1 \mcl A^X u_{N-j}(X^{Y_{j-1}}_{s}) \dif s+\mcl M_{1},
\end{eqnarray*}
where $(\mcl M_{t})_{0\leq t\leq1}$ is a martingale with mean 0, and thus
\begin{eqnarray*}\label{2}
\E\left[u_{N-j}(X^{Y_{j-1}}_{1})-u_{N-j}(Y_{j-1})\right]=\E\left[\int_0^1 \mcl A^X u_{N-j}(X^{Y_{j-1}}_{s}) \dif
s\right].
\end{eqnarray*}
On the other hand, by conditional probability, we obtain
\begin{eqnarray*}
\E\left[u_{N-j}(Y_{1}^{Y_{j-1}})-u_{N-j}(Y_{j-1})\right]&=&\E\left[\mcl A^Y u_{N-j}(Y_{j-1}) \right].
\end{eqnarray*}
Hence, for $1 \le j \le N-1$,
\begin{align*}
\E u_{N-j}(X_{j}(j-1,Y_{j-1}))-\E u_{N-j}(Y_{j})
=\E \int_0^1 \left[\mcl A^X u_{N-j}(X^{Y_{j-1}}_{s})-\mcl A^Y u_{N-j}(Y_{j-1})\right] \dif s.
\end{align*}
Combining all the relations above, we immediately obtain the equalities (\ref{orign}) and (\ref{e:mainnew}) in the theorem, as desired. Moreover, \eqref{e:mainnew-1} is an immediate corollary from \eqref{e:mainnew} by the definition of Wasserstein-1 distance. The proof is complete.
\end{proof}

\section{Three Applications}\label{application}

We only  consider in this section three applications: SDE's
approximation to online SGD, EM discretization for SDE driven by $\alpha$-stable process with $\alpha\in(1,2)$, and normal approximation. As we mentioned early, we focus on Wasserstein-1 distance,
though Theorem \ref{mainthm} can be applied to approximation problems in
other metrics, for instance, if we replace the ${\rm Lip}(1)$ function family by bounded measurable function family, the
approximation turns to be in total variation metric. The other applications will be studied in the forthcoming paper.

\subsection{Application 1: Online SGD and SDEs (\cite{CLTZ16,LTW19})}\label{SMESGD}

For the first application, we concentrate on approximating a family of online SGDs by a SDE driven by multiplicative Brownian motion. Using
our framework, we will obtain an explicit error bound in
the classical Wasserstein-1 distance. We shall give two examples for Theorem \ref{SGDthm} below, which are considered in \cite[Section 5]{LTW19}. We refer to the reader to
\cite{PJ92,CLTZ16,LTW17,TTV16,FGLLL19,LTW19} for more details of SGDs and online SGDs.

The problem approximating SGD by SDE is well studied, see for instance
\cite{TTV16,LTW17,AN19,FGLLL19,HLLL19,LTW19,BS20,FDBD20} and the references therein.

Now, we first introduce the online SGD. Estimation of model parameters by minimizing an objective function is a fundamental idea in statistics. Let
$w^{*}\in\mathbb{R}^{d}$ be the true $d$-dimensional model parameters. In common models, $w^{*}$ is the minimizer
of a convex objective $P(w):\mathbb{R}^{d}\rightarrow\mathbb{R},$ i.e.,
\begin{align*}
w^{*}={\rm argmin}\Big(P(w):=\mathbb{E}_{\zeta\sim\Pi}\psi(w,\zeta)=\int\psi(w,\zeta)\dif\Pi(\zeta)\Big),
\end{align*}
where $\zeta$ denotes the random sample from a probability distribution $\Pi$ and $\psi(w,\zeta)$ is the loss
function. The online SGD is a widely used optimization method for minimizing $P(w)$.

The online SGD is an iterative algorithm, let $w_{0}=x$ and the $k$-th iterate $w_{k}$ takes the following form,
\begin{align}\label{SGD}
w_{k}=w_{k-1}-\eta\nabla\psi(w_{k-1},\zeta_{k}), \quad k\geq1,
\end{align}
where $\eta$ is a small positive step-size known as the learning rate, $\zeta_{k}$ is the $k$-th sample randomly drawn from the distribution $\Pi,$ and
$\nabla\psi(w_{k-1},\zeta_{k})$ denotes the gradient of $\psi(w_{k-1},\zeta_{k})$ with respect to $w$ at
$w=w_{k-1}.$

It is easily seen that online SGD (\ref{SGD}) can be rewritten as
\begin{align*}
w_{k}=w_{k-1}-\eta\nabla P(w_{k-1})+\sqrt{\eta}V_{\eta}(w_{k-1},\zeta_{k}),
\end{align*}
where $V_{\eta}(w_{k-1},\zeta_{k})=\sqrt{\eta}\left(\nabla P(w_{k-1})-\nabla\psi(w_{k-1},\zeta_{k})\right)$. It is straightforward to check that
\begin{align*}
\mathbb{E}\left[V_{\eta}(w_{k-1},\zeta_{k})|w_{k-1}\right]=0, \quad {\rm Cov}\left[V_{\eta}(w_{k-1},\zeta_{k}),V_{\eta}(w_{k-1},\zeta_{k})|w_{k-1}\right]=\eta\Sigma(w_{k-1}),
\end{align*}
where
\begin{align*}
\Sigma(w_{k-1})=&\mathbb{E}\left[\left(\nabla P(w_{k-1})-\nabla\psi(w_{k-1},\zeta_{k})\right)\left(\nabla P(w_{k-1})-\nabla\psi(w_{k-1},\zeta_{k})\right)^{T}|w_{k-1}\right]\\
=&\mathbb{E}\left[\nabla\psi(w_{k-1},\zeta_{k})\nabla\psi(w_{k-1},\zeta_{k})^{T}|w_{k-1}\right]-\nabla P(w_{k-1})\nabla P(w_{k-1})^{T}.
\end{align*}

Now, we can consider the stochastic differential equation (SDE) as follows to approximate the above online SGD:
\begin{align}\label{SME}
d\hat{X}_{t}=-\nabla P(\hat{X}_{t})dt+\left(\eta\Sigma(\hat{X}_{t})\right)^{\frac{1}{2}}d B_{t},\quad \hat{X}_{0}=x,
\end{align}
where $B_{t}$ is a $d$-dimensional Brownian motion. For the research of above SDE with the noise term depending on a small parameter (the learning rate), we refer the reader to \cite{Xi01,CF14,LTW17} and the references therein.

For further use, we shall assume:

{\bf Assumption A1} (i) There
exist $\theta_{0}>0$ and $\theta_{1},\theta_{2},\theta_{3},\theta_{4},\theta_{5}\geq0$ such that for any $v_{1},v_{2},v_{3},x\in\mathbb{R}^{d}$, $\nabla P(x)$ satisfies
\begin{align}\label{diss}
\langle v_{1},\nabla_{v_{1}}\nabla P(x)\rangle\geq\theta_{0}|v_{1}|^{2}, \quad |\nabla_{v_{1}}\nabla_{v_{2}}\nabla P(x)|\leq\theta_{1}|v_{1}||v_{2}|,
\end{align}
\begin{align}\label{onthird}
|\nabla_{v_{1}}\nabla_{v_{2}}\nabla_{v_{3}}\nabla P(x)|\leq\theta_{2}|v_{1}||v_{2}|;
\end{align}
and that any $x,y\in\mathbb{R}^{d}$, $\Sigma(x)^{\frac{1}{2}}$ satisfies
\begin{align}\label{expofir}
\|\nabla_{v_{1}}\Sigma(x)^{\frac{1}{2}}\|_{{\rm HS}}\leq\theta_{3}|v_{1}|, \quad \|\nabla_{v_{2}}\nabla_{v_{1}}\Sigma(x)^{\frac{1}{2}}\|_{{\rm HS}}\leq\theta_{4}|v_{1}||v_{2}|,
\end{align}
\begin{align}
\|\nabla_{v_{3}}\nabla_{v_{2}}\nabla_{v_{1}}\Sigma(x)^{\frac{1}{2}}\|_{{\rm HS}}\leq\theta_{5}|v_{1}||v_{2}||v_{3}|.
\end{align}

(ii) There exists $\delta>0$ such that for any $x\in\mathbb{R}^{d}$ and non-zero vector $\xi\in\mathbb{R}^{d}$, $\Sigma(x)^{\frac{1}{2}}$ satisfies
\begin{align}
\xi^{T}\Sigma(x)^{\frac{1}{2}}\xi\geq\delta|\xi|^{2}.
\end{align}

\begin{remark}\label{rem}
By integration, (\ref{diss}) implies
\begin{align}\label{integration}
\langle x-y,\nabla P(x)-\nabla P(y)\rangle\geq\theta_{0}|x-y|^{2}, \quad \forall x,y\in\mathbb{R}^{d}.
\end{align}

In addition, from now on, we simply write a number $C_{\theta_{0},\ldots,\theta_{5}}$, depending on $\theta_{0},\ldots,\theta_{5}$, by $C_{\theta}$ in shorthand.
\end{remark}

Moreover, in order to ensure $\mathbb{E}|w_{k}|^{4}<\infty$, we further assume:

{\bf Assumption A2} There exists a constant $\kappa>0$ such that
\begin{align}\label{stolin}
\mathbb{E}|\nabla\psi(x,\zeta)-\nabla\psi(y,\zeta)|^{4}\leq\kappa^{4}|x-y|^{4}, \quad \forall x,y\in\mathbb{R}^{d}.
\end{align}

\begin{remark}
By (\ref{stolin}) and convexity inequality, for $j=1,2,3,4,$ it is easily seen that
\begin{align}
\E|\nabla \psi(x,\zeta)|^{j}\le& 2^{j-1} \E|\nabla \psi(x,\zeta)-\nabla \psi(0,\zeta)|^{j}+2^{j-1} \E|\nabla \psi(0,\zeta)|^{j}\nonumber\\
\le& 2^{j-1}\kappa^{j} |x|^{j}+2^{j-1} \ell^{j}_{0}\label{stofo}
\end{align}
with $\ell^{j}_{0}:= \E|\nabla \psi(0,\zeta)|^{j}.$

In addition, since $\Sigma(x)=\mathbb{E}\big[\big(\nabla P(x)-\nabla\psi(x,\zeta)\big)\big(\nabla P(x)-\nabla\psi(x,\zeta)\big)^{T}\big]$, by (\ref{stofo}), (\ref{stolin}) and the Cauchy-Schwarz inequality, we have
\begin{align}
\|\Sigma(x)^{\frac{1}{2}}\|_{{\rm {HS}}}^{2}=&{\rm {Tr}}\big(\Sigma(x)\big)=
\E|\nabla P(x)-\nabla\psi(x,\zeta)|^{2} \nonumber \\
=&\E|\nabla\psi(x,\zeta)|^{2}-|\E\nabla \psi(x,\zeta)|^{2}\le\E|\nabla\psi(x,\zeta)|^{2}\le2\ell^2_{0}+2\kappa^{2} |x|^{2}\label{stodiffu},
\end{align}
\begin{align}\label{stodiffu2}
\|\Sigma(x)\|_{{\rm {HS}}}\leq\mathbb{E}\left|\nabla P(x)-\nabla\psi(x,\zeta)\right|^{2}\leq2\ell^2_{0}+2\kappa^{2} |x|^{2}.
\end{align}
Since $\nabla P(x)=\mathbb{E}\nabla\psi(x,\zeta)$, by (\ref{stofo}), (\ref{stolin}) and the Cauchy-Schwarz inequality, we have
\begin{align}\label{Lip}
\left|\nabla P(x)-\nabla P(y)\right|\leq\mathbb{E}|\psi(x,\zeta)-\psi(y,\zeta)|\leq\kappa|x-y|, \quad \forall x,y\in\mathbb{R}^{d},
\end{align}
which further implies
\begin{align}\label{shiftlin}
\left|\nabla P(x)\right|\leq\left|\nabla P(x)-\nabla P(0)\right|+|\nabla P(0)|\leq\kappa|x|+|\nabla P(0)|.
\end{align}
\end{remark}

To illustrate the online SGD recursion in (\ref{SGD}) and the {\bf Assumption A1} and {\bf Assumption A2}, we consider the
following two motivating examples. In appendix \ref{s:A1}, we will verify that the following two examples satisfy {\bf Assumption A1} and {\bf Assumption
A2}.

\begin{example}\label{ex1}(Model in \cite[Section 5]{LTW19}) Let $H\in\mathbb{R}^{d}$ be a symmetric, positive definite matrix. Define the sample objective
\begin{align*}
\psi(x,\zeta)=\frac{1}{2}(x-\zeta)^{T}H(x-\zeta)-\frac{1}{2}{\rm Tr}(H),
\end{align*}
where $\zeta\sim N(0,I_{d})$. Then, the online SGD iterates in (\ref{SGD}) become,
\begin{align*}
w_{k}=&w_{k-1}-\eta H(w_{k-1}-\zeta_{k}),
\end{align*}
which implies
\begin{align*}
\nabla\psi(x,\zeta)=H(x-\zeta), \qquad \nabla P(x)=Hx, \qquad \Sigma(x)=H^{2}.
\end{align*}
\end{example}

\begin{example}\label{ex2}
(Variation of alternate Model in \cite[Section 5.1]{LTW19}). Let $H\in\mathbb{R}^{d\times d}$ be a symmetric, positive definite matrix, we diagonalize it in the form $H=QDQ^{T},$ where $Q$ is an orthogonal matrix and $D$ is a diagonal matrix of eigenvalues. Let $\alpha\sim N(0,I_{d})$, $\beta\sim N(0,I_{d})$ and $\alpha$ is independent of $\beta$. Denote $\zeta=(\alpha,\beta)$ and we define the loss function
\begin{align*}
\psi(x,\zeta):=\frac{1}{2}(Q^{T}x)^{T}[D+{\rm diag}(\alpha)](Q^{T}x)+\frac{\gamma}{2}(x-\beta)^{T}(x-\beta),
\end{align*}
where ${\rm diag}(\alpha)$ is a diagonal matrix, whose diagonal elements are each component of the vector $\alpha$ and $\gamma>0$ is a tuning parameter. Therefore, the online SGD iterates in (\ref{SGD}) become,
\begin{align*}
w_{k}=w_{k-1}-&\eta\left[Q[D+{\rm diag}(\alpha_{k})]Q^{T}w_{k-1}+\gamma(w_{k-1}-\beta_{k})\right],
\end{align*}
which implies
\begin{align*}
\nabla\psi(x,\zeta)=Q[D+{\rm diag}(\alpha)](Q^{T}x)+\gamma(x-\beta), \qquad \nabla P(x)=Hx+\gamma x
\end{align*}
and
\begin{align*}
\Sigma(x)=&\mathbb{E}\left[\left(Q[D+{\rm diag}(\alpha)](Q^{T}x)-Hx\right)\left(Q[D+{\rm diag}(\alpha)](Q^{T}x)-Hx\right)^{T}\right]\\
&+\mathbb{E}\left[\left(\gamma(x-\beta)-\gamma x\right)\left(\gamma(x-\beta)-\gamma x\right)^{T}\right]\\
=&Q{\rm diag}(Qx)^{2}Q^{T}+\gamma^{2}I_{d}=Q\left[{\rm diag}(Qx)^{2}+\gamma^{2}I_{d}\right]Q^{T}.
\end{align*}
\end{example}

Now, we are at the position to state our theorem of the first application.

\begin{theorem}[Online SGD v.s. SDE]\label{SGDthm}
Keep the same notations as above. Let $N\geq2$ be a natural number. Suppose that {\bf Assumption A1} and {\bf Assumption A2} hold. Then, as $0<\eta\leq\min\{1,\frac{\theta_{0}}{2(10+7\kappa^{4}+7\ell^{4}_{0})}\}$, we have
\begin{align*}
d_W(\mathcal{L}(\hat{X}_{\eta N}),\mathcal{L}(w_{N}))\leq C_{\theta,\kappa,\ell^4_{0}}(1+|x|^{3})(1+\frac{d}{\delta^{2}})(1+|\ln \eta|)\eta,
\end{align*}
where $(\hat{X}_{t})_{t\geq0}$ is the diffusion process, which is defined by SDE (\ref{SME}), $(w_{k})_{k\in\Z_+}$
is the online SGD iteration process, which is defined by (\ref{SGD}).
\end{theorem}

\subsection{Application 2: EM discretization for SDEs driven by $\alpha$-stable process with $\alpha\in(1,2)$ (\cite{JMW96,TA18})}\label{stable}

In recent years, the EM discretization for SDEs driven by $\alpha$-stable process has been studied by \cite{JMW96,WY07,TA18,Liu19} in the finite time interval. In this subsection, we will consider a particular Ornstein-Uhlenbeck process driven by $\alpha$-stable process with $\alpha\in(1,2)$ and obtain a uniform convergence rate with respect to the time.

Let $(Z_{t})_{t\geq0}$ be the $d$-dimensional rotationally symmetric $\alpha$-stable process, i.e., $\mathbb{E}[e^{i\langle
Z_{t},\lambda\rangle}]=e^{-t|\lambda|^{\alpha}},$ then we have
$Z_{t}\stackrel{\rm d}{=}t^{\frac{1}{\alpha}}Z_{1}$ (see, e.g., \cite[Theorem 14.3]{Sat99}) and the corresponding
generator is
\begin{align*}
\Delta^{\frac{\alpha}{2}}f(x)=d_{\alpha}\int_{\mathbb{R}^{d}}\frac{f(x+y)-f(x)}{|y|^{\alpha+d}}dy,\quad
f\in\mathcal{C}_{b}^{2}(\mathbb{R}^{d}),
\end{align*}
where
\begin{align*}
d_{\alpha}=\big(\int_{0}^{\infty}\frac{1-\cos y}{y^{\alpha+1}}dy\int_{\mathbb{S}^{d-1}}|\langle
e,\theta\rangle|^{\alpha}d\theta\big)^{-1}
\end{align*}
and $e$ is an unit vector. Moreover, it is well known that
$d_{\alpha}=\frac{\alpha2^{\alpha}\Gamma((d+\alpha)/2)}{\Gamma(d/2)\Gamma((2-\alpha)/2)}$ (see, e.g.,
\cite{BG68}).

Now, we consider the following SDEs:
\begin{align}\label{SDE}
d\tilde{X}_{t}=-\frac{1}{\alpha}\tilde{X}_{t}dt+dZ_{t},\quad X_{0}=x\in\mathbb{R}^{d}.
\end{align}

Let $\tilde{Z}_{1},\tilde{Z}_{2},\cdots$ be a sequence of random vectors independently drawn from the Pareto random variable $\tilde{Z},$ which has the probability density function
\begin{align}\label{density}
p(z)=\frac{\alpha}{V(\mathbb{S}^{d-1})|z|^{\alpha+d}}{\bf 1}_{(1,\infty)}(|z|),
\end{align}
where $V(\mathbb{S}^{d-1})$ is the surface area of $\mathbb{S}^{d-1}$ and it is well known that $V(\mathbb{S}^{d-1})=\frac{2\pi^{\frac{d}{2}}}{\Gamma(\frac{d}{2})}$.
The reason we choose Pareto random variable is that it has an explicit density function and thus can be easily sampled on computer.

We consider the following discrete Markov process with step size $\eta\in(0,1)$ to approximate the above SDEs: let initial value $\tilde{Y}_{0}=x$ and
\begin{align}\label{E-M}
\tilde{Y}_{k+1}=\tilde{Y}_{k}-\frac{\eta}{\alpha}\tilde{Y}_{k}+\frac{\eta^{\frac{1}{\alpha}}}{\sigma}\tilde{Z}_{k+1}, \quad k\geq0,
\end{align}
where $\sigma=\big(\frac{\alpha}{V(\mathbb{S}^{d-1})d_{\alpha}}\big)^{\frac{1}{\alpha}}$.

Then, we have the following theorem:

\begin{theorem} [EM discretization for SDEs driven by $\alpha$-stable process]\label{stablethm}
Keep the same notations as above. Let $N\geq2$ be a natural number. Then, as $\eta\in(0,1]$, we have
\begin{align*}
d_W(\mathcal{L}(\tilde{X}_{\eta N}),\mathcal{L}(\tilde{Y}_{N}))\leq C_{\alpha,d}(1+|x|)\eta^{\frac{2-\alpha}{\alpha}},
\end{align*}
where $(\tilde{X}_{t})_{t\geq0}$ and $(\tilde{Y}_{k})_{k\in\mathbb{Z}^{+}}$ are defined by (\ref{SDE}) and (\ref{E-M}), respectively.
\end{theorem}

\subsection{Application 3: Multivariate Normal CLT (\cite{Ste72,CGS10})}\label{norm}

Finally, we apply Theorem \ref{mainthm} to the multivariate normal approximation, and recover the results in
\cite{VV10,FSX19}. 

In this application, we denote the $d$-dimensional Brownian motion by $(B_{t})_{t\geq0}$ and denote the $d$-dimensional standard normal distribution by $N(0,I_{d}),$ that is, if $B\sim N(0,I_{d}),$ then
$\mathbb{E}[e^{i\langle
B,\lambda\rangle}]=e^{-\frac{|\lambda|^{2}}{2}}$ for any $\lambda\in\mathbb{R}^{d}.$ Moreover, it is well known that $B\stackrel{\rm d}{=}B_{1}.$

\begin{theorem} [Multivariate normal CLT]\label{normalCLT}
Let $B\sim N(0,I_{d})$ and $S_{n}=\sum_{i=1}^{n}\frac{\xi_{i}}{\sqrt{n}}$
with i.i.d. random vectors $(\xi_{i})_{i\in\mathbb{N}}$ satisfying $\mathbb{E}\xi_{i}=\mathbf{0},$
$\mathbb{E}\xi_{i}\xi_{i}^{T}=I_{d}$ and $\sup_{i}\mathbb{E}|\xi_{i}|^{3}<\infty$. Then, we have
\begin{align*}
d_W(\mathcal{L}(B),\mathcal{L}(S_{n}))\leq\big[(\frac{2}{3}d+1)\mathbb{E}|B|+\frac{1}{3}\mathbb{E}|\xi_{1}|^{3}+\mathbb{E}|\xi_{1}|\big]n^{-\frac{1}{2}}(1+\ln
n).
\end{align*}
\end{theorem}

\section{Proofs of Theorems \ref{SGDthm}, \ref{stablethm} and \ref{normalCLT}}

In this section, with the help of Theorem \ref{mainthm}, we focus on proving the Theorem \ref{SGDthm}, Theorem \ref{stablethm} and Theorem \ref{normalCLT}.

\subsection{Proof of Theorem \ref{SGDthm}}\label{proofsgd}
We first give the following upper bounds of the processes $w_{k}$ and $\hat{X}_{t}$, which will be proved in
Appendix \ref{s:A1}.

\begin{lemma}\label{second moment}
Let $w_{k}$ be defined by (\ref{SGD}) with $w_{0}=x\in\mathbb{R}^{d}$. Then, as $\eta\leq\min\{1,\frac{\theta_{0}}{2(10+7\kappa^{4}+7\ell^{4}_{0})}\}$, for any $k\geq1$, we have
\begin{align}\label{SGDM}
\mathbb{E}|w_{k}|^{4}\leq|w_{0}|^{4}+C_{\theta,\kappa,\ell_{0}^{4}}.
\end{align}
\end{lemma}

\begin{lemma}\label{fourth}
Let $\hat{X}_{t}$ be the solution to the equation (\ref{SME}). Then, as $\eta\leq\min\{1,\frac{\theta_{0}}{4\kappa^{2}}\}$, for any $t>0,$ we have
\begin{align}\label{moment}
\mathbb{E}|\hat{X}_{t}|^{2}<|x|^{2}+\frac{2|\nabla P(0)|^{2}+2\theta_{0}\ell_{0}^{2}}{\theta_{0}^{2}},
\end{align}
\begin{align}\label{moment-1}
\mathbb{E}|\hat{X}_{t}-x|^{2} \le C_{\theta,\kappa,\ell_{0}^{2}}(1+|x|^{2})(t+\eta)t.
\end{align}
\end{lemma}

With the help of Malliavin calculus and Bismut's formula, we can obtain the following estimates, which will be
proved in Appendix \ref{Mal}.

\begin{lemma}\label{mainlem1}
Let $\hat{X}_{t}$ be the solution to the equation (\ref{SME}) and denote
$P_{t}h(x)=\mathbb{E}[h(\hat{X}_{t}^{x})]$ for $h\in {\rm Lip(1)}.$ Then, as $\eta\leq\min\{1,\frac{\theta_{0}}{4\theta_{3}^{2}}\}$, for any
$x,v,v_{1},v_{2},v_{3}\in\mathbb{R}^{d}$ and $t>0$, we have
\begin{align}\label{gradient2}
|\nabla_{v}(P_{t}h)(x)|\leq e^{-\frac{\theta_{0}}{8} t}|v|,
\end{align}
\begin{align}\label{sec}
|\nabla_{v_{2}}\nabla_{v_{1}}P_{t}h(x)|\leq
C_{\theta}(1+\frac{\sqrt{d}}{\delta})(1+\frac{1}{\sqrt{\eta t}})e^{-\frac{{\theta_{0}}}{8}t}|v_{1}||v_{2}|,
\end{align}
\begin{align}\label{third}
|\nabla_{v_{3}}\nabla_{v_{2}}\nabla_{v_{1}}P_{t}h(x)|\leq
C_{\theta}(1+\frac{d}{\delta^{2}})\big(1+\frac{1}{\eta t}+\frac{1}{t^{\frac{5}{4}}}\big)e^{-\frac{\theta_{0}}{8} t}|v_{1}||v_{2}||v_{3}|.
\end{align}
\end{lemma}

In order to use Theorem \ref{mainthm} to solve the problem, we also need the following lemma, which will be proved in Appendix \ref{s:A1}.

\begin{lemma}\label{SGDtaylor}
Let $X_{t}=\hat{X}_{\eta t}$ and $Y_{k}=w_{k}$. Denote $u_k(x)=\E h(X^x_{k})$ for $1\leq k\leq N.$ Then, as $\eta\leq\min\{1,\frac{\theta_{0}}{2(10+7\kappa^{4}+7\ell^{4}_{0})}\}$, we have
\begin{align*}
&\left|\mathbb{E}\int_{0}^{1}\big[\mathcal{A}^{X}u_{k}(X_{s}^{x})-\mathcal{A}^{Y} u_{k}(x)\big]\dif s\right|\\
\leq&C_{\theta,\kappa,\ell^3_{0}}(1+|x|^{3})(1+\frac{d}{\delta^{2}})\big(1+\frac{1}{\eta k}+\frac{\eta}{(\eta k)^{\frac{5}{4}}}\big)\eta^{2}e^{-\frac{\theta_{0}}{8} \eta k}.
\end{align*}
\end{lemma}

With the above results, we can give the proof of Theorem \ref{SGDthm}.\\
{\bf {\it Proof of Theorem \ref{SGDthm}.}}
In order to apply Theorem \ref{mainthm},
we need to identify the $X_{t}$ and $Y_{k}$ therein in our setting and compute the corresponding $\mcl A^{X}$ and
$\mcl A^{Y}$. Let $X_{t}=\hat{X}_{\eta t}$, $Y_{k}=w_{k}$, $X_{0}=Y_{0}=x \in \R^{d}$ and $N\geq2$. {Then, $u_{k}(x)=\mathbb{E}h(X_{k}^{x})=\mathbb{E}h(\hat{X}_{\eta k}^{x})$ for $k\geq1$. Notice that $h\in{\rm Lip}(1)$, the Cauchy-Schwarz inequality and (\ref{moment}) imply $\mathbb{E}|h(X_{t}^{x})|\leq |h(0)|+\sqrt{\mathbb{E}|\hat{X}_{\eta t}|^{2}}<\infty$; Similarly, we can derive $\mathbb{E}|h(Y_{k}^{y})|<\infty$, that is, the condition (1) holds. In addition, Lemma \ref{mainlem1} implies that the function $u_{k}$ has bounded 1st, 2nd and 3rd order derivatives. Hence, by the It$\hat{o}$'s formula for SDE (see, e.g., \cite[Theorem 4.2.1]{Oks13}), we can see that the conditions (2) and (3) in Theorem \ref{mainthm} is satisfied.

Now we apply Theorem \ref{mainthm} to prove the theorem,} it suffices to bound the two terms $\mathcal{I}_{h}$ and $\mathcal{II}_{h}$ in (\ref{e:mainnew}). For the term $\mathcal{I}_{h}$, by Lemma \ref{SGDtaylor} and (\ref{SGDM}), we have
\begin{align*}
|\mathcal{I}_{h}|\leq&C_{\theta,\kappa,\ell^3_{0}}\sum_{j=1}^{N-1}(1+\mathbb{E}|Y_{j-1}|^{3})(1+\frac{d}{\delta^{2}})\big(1+\frac{1}{\eta(N-j)}+\frac{\eta}{[\eta(N-j)]^{\frac{5}{4}}}\big)e^{-\frac{\theta_{0}}{8} \eta(N-j)}\eta^{2}\\
\leq&C_{\theta,\kappa,\ell^4_{0}}(1+|x|^{3})(1+\frac{d}{\delta^{2}})\eta^{2}\sum_{j=1}^{N-1}\big(1+\frac{1}{\eta j}+\frac{\eta}{[\eta j]^{\frac{5}{4}}}\big)e^{-\frac{\theta_{0}}{8} \eta j}\\
\leq&C_{\theta,\kappa,\ell^4_{0}}(1+|x|^{3})(1+\frac{d}{\delta^{2}})\eta^{2}\Big[1+\frac{1}{\eta}+\frac{1}{\eta^{\frac{1}{4}}}+\int_{1}^{N}\big(1+\frac{1}{\eta v}+\frac{\eta}{[\eta v]^{\frac{5}{4}}}\big)e^{-\frac{\theta_{0}}{8} \eta v}\dif v\Big].
\end{align*}
Then, since $\eta\leq1$, it is straightforward to calculate
\begin{align*}
|\mathcal{I}_{h}|\leq&C_{\theta,\kappa,\ell^4_{0}}(1+|x|^{3})(1+\frac{d}{\delta^{2}})\eta\Big[\eta+1+\eta^{\frac{3}{4}}+\eta\int_{1}^{N}\big(1+\frac{1}{\eta v}+\frac{\eta}{[\eta v]^{\frac{5}{4}}}\big)e^{-\frac{\theta_{0}}{8} \eta v}\dif v\Big]\\
\leq&C_{\theta,\kappa,\ell^4_{0}}(1+|x|^{3})(1+\frac{d}{\delta^{2}})\eta\Big[1+\int_{\eta}^{\eta N}\big(1+\frac{1}{v}+\frac{\eta}{v^{\frac{5}{4}}}\big)e^{-\frac{\theta_{0}}{8}v}\dif v\Big]\\
\leq&C_{\theta,\kappa,\ell^4_{0}}(1+|x|^{3})(1+\frac{d}{\delta^{2}})(1+|\ln \eta|)\eta.
\end{align*}

For the term $\mathcal{II}_{h}$, by the Cauchy-Schwarz inequality, (\ref{moment-1}) and (\ref{SGDM}), we have
\begin{align*}
\mathbb{E}|h\big(X_{1}^{Y_{N-1}}\big)-h(Y_{N-1})|\leq&\mathbb{E}|\hat{X}_{\eta}^{w_{N-1}}-w_{N-1}|\leq C_{\theta,\kappa,\ell_{0}^{4}}(1+|x|)\eta.
\end{align*}
Recall (\ref{SGD}), by Cauchy-Schwarz inequality, (\ref{stofo}) and (\ref{SGDM}), we have
\begin{align*}
\mathbb{E}|h(Y_{N})-h(Y_{N-1})|\leq&\mathbb{E}|w_{N}-w_{N-1}|\leq\eta\mathbb{E}\left|\nabla\psi(w_{N-1},\zeta_{N})\right|\leq C_{\theta,\kappa,\ell_{0}^{4}}(1+|x|)\eta.
\end{align*}
These imply
\begin{align*}
|\mathcal{II}_{h}|\leq C_{\theta,\kappa,\ell_{0}^{4}}(1+|x|)\eta.
\end{align*}

Combining all of above, we have
\begin{align*}
d_W(\mathcal{L}(\hat{X}_{\eta N}),\mathcal{L}(w_{N}))\leq C_{\theta,\kappa,\ell^4_{0}}(1+|x|^{3})(1+\frac{d}{\delta^{2}})(1+|\ln \eta|)\eta.
\end{align*}

\qed

\subsection{Proof of Theorem \ref{stablethm}}\label{EMPROOF}
We first give the following upper bounds of the processes $\tilde{Y}_{k}$ and $\tilde{X}_{t}$, which will be proved in Appendix \ref{stable application}.

\begin{lemma}\label{EMmoment}
Let $\tilde{Y}_{k}$ be defined by (\ref{E-M}) with $\tilde{Y}_{0}=x\in\mathbb{R}^{d}$. Then, as $\eta\in(0,1]$, we have
\begin{align*}
\mathbb{E}|\tilde{Y}_{k}|\leq C_{\alpha,d}(1+|x|).
\end{align*}
\end{lemma}

\begin{lemma}\label{SDEM}
Let $\tilde{X}_{t}$ be the solution to the equation (\ref{SDE}). Then, for any $x\in\mathbb{R}^{d}$ and $t>0$, we have
\begin{align}\label{SDEmoment1}
\mathbb{E}|\tilde{X}_{t}|\leq C_{\alpha,d}(1+|x|),
\end{align}
\begin{align}\label{SDEM2}
\mathbb{E}|\tilde{X}_{t}^{x}-x|\leq C_{\alpha,d}(1+|x|)(t+t^{\frac{1}{\alpha}}).
\end{align}
\end{lemma}

Moreover, by It$\hat{o}$'s formula, we have the following lemma, which will be proved in Appendix \ref{stable application}.

\begin{lemma}\label{compare}
Let $(\tilde{X}_{t})_{t\geq0}$ and $(\tilde{Y}_{k})_{k\geq0}$ be defined by (\ref{SDE}) and (\ref{E-M}), respectively. Then, for any $x\in\mathbb{R}^{d}$, $\eta\in(0,1]$, $f:\mathbb{R}^{d}\rightarrow\mathbb{R}$ satisfying $\|\nabla f\|<\infty$ and $\|\nabla^{2}f\|_{{\rm HS}}<\infty$, we have
\begin{align*}
\big|\mathbb{E}[f(\tilde{X}_{\eta}^{x})-f(\tilde{Y}_{1})]\big|\leq C_{\alpha,d}(1+|x|)(\|\nabla f\|+\|\nabla^{2}f\|_{{\rm HS}})\eta^{\frac{2}{\alpha}}.
\end{align*}
\end{lemma}

With the help of the heat kernel estimates of the $\alpha$-stable process, we can obtain the following estimates, which will be proved in Appendix \ref{stable application}.

\begin{lemma}\label{Qregular}
Let $\tilde{X}_{t}$ be the solution to the equation (\ref{SDE}), and denote $Q_{t}h(x)=\mathbb{E}[h(\tilde{X}_{t}^{x})]$ for $h\in{\rm Lip}(1)$. Then, for any $x\in\mathbb{R}^{d}$ and $t>0$, we have
\begin{align*}
|\nabla(Q_{t}h)(x)|\leq\|\nabla h\|e^{-\frac{t}{\alpha}}, \qquad \|\nabla^{2}(Q_{t}h)(x)\|_{{\rm HS}}\leq C_{\alpha,d}\|\nabla h\|t^{-\frac{1}{\alpha}}e^{-\frac{t}{\alpha}}.
\end{align*}
\end{lemma}

With the above results, we can give the proof of Theorem \ref{stablethm}.\\
{\bf {\it Proof of Theorem \ref{stablethm}.}}
In order to apply Theorem \ref{mainthm}, we need to identify the $X_{t}$ and $Y_{k}$. Let $X_{t}=\tilde{X}_{\eta t}$, $Y_{k}=\tilde{Y}_{k}$, $X_{0}=Y_{0}=x\in\mathbb{R}^{d}$ and $N\geq2$. Now we apply (\ref{orign}) with $u_{k}(x)=\mathbb{E}h(X_{k}^{x})=\mathbb{E}h(\tilde{X}_{\eta k}^{x})$ for $k\geq1$ to prove the theorem. Since $h\in{\rm Lip}(1)$, the (\ref{SDEmoment1}) and Lemma \ref{EMmoment} imply $\mathbb{E}|h(X_{t}^{x})|\leq |h(0)|+\mathbb{E}|\tilde{X}_{\eta t}|<\infty$ and $\mathbb{E}|h(Y_{k}^{x})|\leq |h(0)|+\mathbb{E}|\tilde{Y}_{k}|<\infty$, respectively, the condition (1) is proved. By Lemma \ref{Qregular}, we can further imply that the function $u_{k}$ has bounded 1st and 2nd derivatives. Hence, by the It$\hat{o}$'s formula for SDE, we can see that the conditions (2) and (3) in Theorem \ref{mainthm} is satisfied. When $1\leq j\leq N-1$, by Lemmas \ref{compare}, \ref{Qregular} and \ref{EMmoment}, we have
\begin{align*}
\left|\E \left[u_{N-j}(X^{Y_{j-1}}_{1})-u_{N-j}(Y_{1}^{Y_{j-1}})\right]\right|\leq& C_{\alpha,d}(1+\mathbb{E}|Y_{j-1}|)\big(1+[\eta(N-j)]^{-\frac{1}{\alpha}}\big)e^{-\frac{\eta(N-j)}{\alpha}}\eta^{\frac{2}{\alpha}}\\
\leq& C_{\alpha,d}(1+|x|)\big(1+[\eta(N-j)]^{-\frac{1}{\alpha}}\big)e^{-\frac{\eta(N-j)}{\alpha}}\eta^{\frac{2}{\alpha}},
\end{align*}
which implies
\begin{align*}
&\left|\sum_{j=1}^{N-1} \E \left[u_{N-j}(X^{Y_{j-1}}_{1})-u_{N-j}(Y_{1}^{Y_{j-1}})\right]\right|\\
\leq&C_{\alpha,d}(1+|x|)\sum_{j=1}^{N-1}\big(1+[\eta(N-j)]^{-\frac{1}{\alpha}}\big)e^{-\frac{\eta(N-j)}{\alpha}}\eta^{\frac{2}{\alpha}}\\
\leq&C_{\alpha,d}(1+|x|)\eta^{\frac{2}{\alpha}}\int_{0}^{N}\big(1+(\eta y)^{-\frac{1}{\alpha}}\big)e^{-\frac{\eta y}{\alpha}}\dif y\leq C_{\alpha,d}(1+|x|)\eta^{\frac{2-\alpha}{\alpha}}.
\end{align*}
When $j=N$, by (\ref{SDEM2}) and Lemma \ref{EMmoment}, we have
\begin{align*}
\big|\mathbb{E}\big[h\big(\tilde{X}_{\eta}^{\tilde{Y}_{N-1}}\big)-h(\tilde{Y}_{N-1})\big]\big|\leq\mathbb{E}\big|\tilde{X}_{\eta}^{\tilde{Y}_{N-1}}-\tilde{Y}_{N-1}|\big|
\leq&C_{\alpha,d}(1+\mathbb{E}|\tilde{Y}_{N-1}|)\eta^{\frac{1}{\alpha}}\\
\leq&C_{\alpha,d}(1+|x|)\eta^{\frac{1}{\alpha}},
\end{align*}
and recall (\ref{E-M}), Lemma \ref{EMmoment} implies
\begin{align*}
\big|\mathbb{E}[h(\tilde{Y}_{N})-h(\tilde{Y}_{N-1})]\big|\leq\mathbb{E}\big|\tilde{Y}_{N}-\tilde{Y}_{N-1}\big|
\leq&\frac{\eta}{\alpha}\mathbb{E}|\tilde{Y}_{N-1}|+\frac{\eta^{\frac{1}{\alpha}}}{\sigma}\mathbb{E}|\tilde{Z}_{N}|\leq C_{\alpha,d}(1+|x|)\eta^{\frac{1}{\alpha}}.
\end{align*}
These imply
\begin{align*}
\big|\mathbb{E}\big[h\big(\tilde{X}_{\eta}^{\tilde{Y}_{N-1}}\big)-h(\tilde{Y}_{1}^{\tilde{Y}_{N-1}})\big]\big|\leq C_{\alpha,d}(1+|x|)\eta^{\frac{1}{\alpha}}.
\end{align*}

Combining all of above, we have
\begin{align*}
 \big|\mathbb{E}h(\tilde{X}_{\eta N})-\mathbb{E}h(\tilde{Y}_{N})\big|\leq C_{\alpha,d}(1+|x|)\eta^{\frac{2-\alpha}{\alpha}}.
 \end{align*}

\qed

\subsection{Proof of theorem \ref{normalCLT}}\label{proofCLT}

In order to use Theorem \ref{mainthm}, we need the following properties for the semigroup of Brownian motion, which
will be proved in Appendix \ref{normal spplication}.

\begin{lemma}\label{normal}
	Let $h\in {\rm Lip(1)}$ and denote $P_{t}h(x)=\mathbb{E}h(B_{t}^{x}),$
	then for any $x,v,v_{1},v_{2}\in\mathbb{R}^{d}$ and $t>0$, we have
	\begin{align}\label{Bp3}
	\big|\langle\nabla^{2}(P_{t}h)(x+v)-\nabla^{2}(P_{t}h)(x),v_{1}v_{2}^{T}\rangle_{\rm{HS}}\big|\leq
\frac{2|v_{1}||v_{2}||v|}{t},
	\end{align}
	\begin{align}\label{Bp4}
	\big|\Delta(P_{t}h)(x+v)-\Delta(P_{t}h)(x)\big|\leq\frac{2d}{t}|v|.
	\end{align}
\end{lemma}

With the above results, we can give the proof of Theorem \ref{normalCLT}.\\
{\bf {\it Proof of Theorem \ref{normalCLT}.}}
In order to apply Theorem \ref{mainthm}, we first need to identify the $X_{t}$ and $Y_{k}$ therein in our setting
and compute the corresponding $\mcl A^{X}$ and $\mcl A^{Y}$. Let $X_{t}=B_{\frac{t}{n}},$
$Y_{k}=S_{k}=\sum_{i=1}^{k}\frac{\xi_{i}}{\sqrt{n}},$ where $(\xi_{i})_{i\in\mathbb{Z}_{+}}$ is a sequence of
i.i.d. random vectors satisfying $\mathbb{E}\xi_{i}=\mathbf{0},$ $\mathbb{E}\xi_{i}\xi_{i}^{T}=I_{d}$ and
$\mathbb{E}|\xi_{i}|^{3}<\infty,$ $X_{0}=Y_{0}=0$ and $N=n.$ {Then, $u_k(x)=\E h({X}^x_{k})=\E h\big(B^x_{\frac{k}{n}}\big)$ for $k\geq1$, (\ref{condition1}) and (\ref{condition2}) below imply that $u_k$ has bounded first and second order derivatives, whereby it is easy to obtain that the conditions (2) and (3) in Theorem \ref{mainthm} is satisfied by It\^{o}'s formula of Brownian motion. In addition, the condition (1) can be derived easily from the fact $h\in{\rm Lip}(1)$, $\mathbb{E}|B_{t}|<\infty$ and $\mathbb{E}|\xi_{i}|<\infty$ for any $t\geq0$, $i=1,2,\cdots, n$. Then, it is straightforward to check that
\begin{align*}
\mathcal{A}^{X}u_{k}(x)=\lim_{t\rightarrow0}\frac{\mathbb{E}u_{k}(X_{t}^{x})-u_{k}(x)}{t}=\frac{1}{n}\lim_{t\rightarrow0}\frac{\mathbb{E}u_{k}\left(B_{\frac{t}{n}}^{x}\right)-u_{k}(x)}{\frac{t}{n}}
=\frac{1}{2n}\Delta
u_{k}(x)
\end{align*}
and
\begin{align*}
\mathcal{A}^{Y}u_{k}(x)=&\mathbb{E}[u_{k}(Y_{1}^{x})-u_{k}(x)]=\mathbb{E}\big[u_{k}(Y_{1}^{x})-u_{k}(x)-\langle\nabla
u_{k}(x),\frac{\xi_{1}}{\sqrt{n}}\rangle\big]\\
=&\frac{1}{n}\mathbb{E}\big[\int_{0}^{1}\int_{0}^{r}\langle\nabla^{2}u_{k}(x+s\frac{\xi_{1}}{\sqrt{n}}),\xi_{1}\xi_{1}^{T}\rangle_{{\rm
{HS}}}dsdr\big].
\end{align*}
Hence,
\begin{align}\label{difference2}
|\mathcal{A}^{X}u_{k}(x)-\mathcal{A}^{Y} u_{k}(x)|=&\big|\frac{1}{2n}\Delta
u_{k}(x)-\frac{1}{n}\mathbb{E}\big[\int_{0}^{1}\int_{0}^{r}\langle\nabla^{2}u_{k}(x+s\frac{\xi_{1}}{\sqrt{n}}),\xi_{1}\xi_{1}^{T}\rangle_{{\rm
{HS}}}dsdr\big]\big|.
\end{align}

Now we apply Theorem \ref{mainthm} to
prove the theorem,} it suffices to bound the two terms $\mcl I_{h}, \mcl{II}_{h}$ in \eqref{e:mainnew}. For the term $\mcl I_{h}$, we rewrite it as
\begin{align*}
\mcl I_{h}=&\sum_{j=1}^{N-1}
 \mathbb{E}\big[\mathcal{A}^{X}u_{N-j}(Y_{j-1})-\mathcal{A}^{Y}u_{N-j}(Y_{j-1})\big]\\
 &+\sum_{j=1}^{N-1}
 \mathbb{E}\int_{0}^{1}\big[\mathcal{A}^{X}u_{N-j}(X_{s}^{Y_{j-1}})-\mathcal{A}^{X}u_{N-j}(Y_{j-1})\big]\dif s:=\mcl I_{h,1}+\mcl I_{h,2}.
\end{align*}
For the first term, noticing that $\mathbb{E}\xi_{1}\xi_{1}^{T}=I_{d}$ and
$\langle\nabla^{2}f(x),I_{d}\rangle_{{\rm HS}}=\Delta f(x),$ by (\ref{difference2}), we have
\begin{align*}
|\mathcal{A}^{X}f(x)-\mathcal{A}^{Y} f(x)|=&\big|\frac{1}{2n}\Delta
f(x)-\frac{1}{n}\mathbb{E}\big[\int_{0}^{1}\int_{0}^{r}\langle\nabla^{2}f(x+s\frac{\xi_{1}}{\sqrt{n}}),\xi_{1}\xi_{1}^{T}\rangle_{{\rm
{HS}}}dsdr\big]\big|\\
\leq&\frac{1}{n}\int_{0}^{1}\int_{0}^{r}\mathbb{E}\big|\langle\nabla^{2}f(x+s\frac{\xi_{1}}{\sqrt{n}})-\nabla^{2}f(x),\xi_{1}\xi_{1}^{T}\rangle_{{\rm
{HS}}}\big|dsdr.
\end{align*}
Then, by (\ref{Bp3}), we have
\begin{align*}
|\mathcal{I}_{h,1}|\leq&\frac{1}{n}\sum_{j=1}^{n-1}\int_{0}^{1}\int_{0}^{r}\mathbb{E}\big|\langle\nabla^{2}u_{n-j}(x+s\frac{\xi_{1}}{\sqrt{n}})-\nabla^{2}u_{n-j}(x),\xi_{1}\xi_{1}^{T}\rangle_{{\rm
{HS}}}\big|dsdr\\
\leq&\frac{2}{\sqrt{n}}\mathbb{E}|\xi_{1}|^{3}\int_{0}^{1}\int_{0}^{r}sdsdr\sum_{j=1}^{n-1}\frac{1}{n-j}\\
=&\frac{1}{3\sqrt{n}}\mathbb{E}|\xi_{1}|^{3}\sum_{j=1}^{n-1}\frac{1}{j}\leq\frac{1}{3\sqrt{n}}\mathbb{E}|\xi_{1}|^{3}(1+\int_{1}^{n}\frac{1}{y}\dif
y)=\frac{1}{3\sqrt{n}}\mathbb{E}|\xi_{1}|^{3}(1+\ln n).
\end{align*}
For the second term, by (\ref{Bp4}) and the scaling property of $B_{t},$ i.e., $B_{t}\stackrel{\rm
d}{=}t^{\frac{1}{2}}B_{1},$ we have
\begin{align*}
|\mathcal{I}_{h,2}|\leq&\frac{1}{2n}\sum_{j=1}^{n-1}\int_{0}^{1}\mathbb{E}\big|\Delta u_{n-j}(X_{s}^{Y_{j-1}})-\Delta
u_{n-j}(Y_{j-1})\big|ds\\
\leq&d\int_{0}^{1}\mathbb{E}|B_{\frac{s}{n}}|ds\sum_{j=1}^{n-1}\frac{1}{n-j}=\frac{d}{\sqrt{n}}\mathbb{E}|B_{1}|\int_{0}^{1}s^{\frac{1}{2}}ds\sum_{j=1}^{n-1}\frac{1}{j}
\leq\frac{2d}{3\sqrt{n}}\mathbb{E}|B_{1}|(1+\ln n).
\end{align*}

It remains to estimate $\mathcal{II}_h$. By the scaling property of $B_{t},$ it is easily seen that
\begin{align*}
\mathbb{E}|h\big(X_{1}^{Y_{n-1}}\big)-h(Y_{n-1})|\leq&\mathbb{E}|X_{1}^{Y_{n-1}}-Y_{n-1}|=\mathbb{E}|B_{\frac{1}{n}}|=\frac{1}{\sqrt{n}}\mathbb{E}|B_{1}|,
\end{align*}
\begin{align*}
\mathbb{E}|h(Y_{n})-h(Y_{n-1})|\leq\mathbb{E}|Y_{n}-Y_{n-1}|=E\frac{|\xi_{n}|}{\sqrt{n}}.
\end{align*}
These imply
\begin{align*}
|\mathcal{II}_h|\leq(\mathbb{E}|B_{1}|+\mathbb{E}|\xi_{n}|)\frac{1}{\sqrt{n}}.
\end{align*}
Collecting the estimates of $\mathcal{I}_h$ and $\mathcal{II}_h$, which hold true for all $h\in {\rm {\rm
Lip(1)}},$ we immediately obtain
\begin{align*}
d_W(\mathcal{L}(B_{1}),\mathcal{L}(S_{n}))\leq\big[(\frac{2}{3}d+1)\mathbb{E}|B_{1}|+\frac{1}{3}E|\xi_{1}|^{3}+E|\xi_{1}|\big]n^{-\frac{1}{2}}(1+\ln
n).
\end{align*}
Then, the desired result follows from the fact that $B\stackrel{\rm d}{=}B_{1}.$
\qed

\begin{appendix}
\section{Proofs of Lemmas in {Subsections} \ref{SMESGD} and \ref{proofsgd}}\label{s:A1}

\subsection{Verifying assumptions for two examples.}

In this subsection, we verify {\bf Assumption A1} and {\bf Assumption A2} for Examples \ref{ex1} and \ref{ex2}.

\begin{lemma}
In Example \ref{ex1}, denote the smallest eigenvalue of the matrix $H$ by $\lambda_{\min}(H)$. Then, {\bf Assumption A1} and  {\bf Assumption A2} hold for
$\theta_{0}=\delta=\lambda_{\min}(H),$ $\kappa=\|H\|_{{\rm HS}}$ and $\theta_{1}=\theta_{2}=\theta_{3}=\theta_{4}=\theta_{5}=0$.
\end{lemma}
\begin{proof}
Recall
\begin{align}\label{exab1}
\nabla P(x)=Hx, \qquad \Sigma(x)^{\frac{1}{2}}=H, \qquad \nabla\psi(x,\zeta)=H(x-\zeta).
\end{align}
Then, for any $v,x,y\in\mathbb{R}^{d}$ and non-zero vector $\xi\in\mathbb{R}^{d}$, it is easy to see that
\begin{align*}
\langle v,\nabla_{v}\nabla P(x)\rangle=\langle v,Hv\rangle\geq\lambda_{\min}(H)|v|^{2},
\end{align*}
\begin{align*}
\mathbb{E}\left|H(x-\zeta)-H(y-\zeta)\right|^{4}\leq\|H\|_{{\rm HS}}^{4}|x-y|^{4}
\end{align*}
and
\begin{align*}
\xi^{T}H\xi\geq\lambda_{\min}(H)|\xi|^{2}.
\end{align*}
Moreover, other results are clearly available from (\ref{exab1}).
\end{proof}

\begin{lemma}
In Example \ref{ex2}, denote the smallest eigenvalue of the matrix $H$ by $\lambda_{\min}(H)$. Then, {\bf Assumption A1} and  {\bf Assumption A2} hold for
$\theta_{0}=\lambda_{\min}(H)+\gamma,$ $\kappa^{4}=27(\|H\|_{{\rm HS}}^{4}+3d^{6}+\gamma^{4}),$ $\delta=\gamma$, $\theta_{1}=\theta_{2}=0$, $\theta_{3}=\sqrt{d}\|Q\|_{{\rm HS}}^{3}$, $\theta_{4}=\sqrt{d}\|Q\|_{{\rm HS}}^{4}(1+\gamma^{-1}+\gamma^{-3})$ and $\theta_{5}=3\sqrt{d}\|Q\|_{{\rm HS}}^{5}(2+\gamma^{-3}+\gamma^{-5})$.
\end{lemma}
\begin{proof}
Recall
\begin{align}\label{last2}
\nabla\psi(x,\zeta)=Q[D+{\rm diag}(\alpha)](Q^{T}x)+\gamma(x-\beta), \qquad \nabla P(x)=Hx+\gamma x,
\end{align}
\begin{align*}
\Sigma(x)^{\frac{1}{2}}=Q\left[{\rm diag}(Qx)^{2}+\gamma^{2}I_{d}\right]^{\frac{1}{2}}Q^{T}.
\end{align*}
Then, for any $v,x,y\in\mathbb{R}^{d}$ and non-zero vector $\xi\in\mathbb{R}^{d}$, it is easy to see that
\begin{align*}
\langle v,\nabla_{v}\nabla P(x)\rangle=\langle v,(H+\gamma I_{d})v\rangle\geq(\lambda_{\min}(H)+\gamma)|v|^{2},
\end{align*}
\begin{align*}
&\mathbb{E}\left|Q[D+{\rm diag}(\alpha)](Q^{T}x)+\gamma(x-\beta)-Q[D+{\rm diag}(\alpha)](Q^{T}y)-\gamma(y-\beta)\right|^{4}\\
\leq&27(\|H\|_{{\rm HS}}^{4}+\|Q\|_{{\rm HS}}^{8}\mathbb{E}|\alpha|^{4}+\gamma^{4})|x-y|^{4}\leq27(\|H\|_{{\rm HS}}^{4}+3d^{6}+\gamma^{4})|x-y|^{4},
\end{align*}
\begin{align*}
\xi^{T}\Sigma(x)^{\frac{1}{2}}\xi\geq\gamma|\xi|^{2}.
\end{align*}
Moreover, by (\ref{last2}), it is easily seen that $\theta_{1}=\theta_{2}=0$. For any $v_{1},v_{2},v_{3},x\in\mathbb{R}^{d}$, notice that $\Sigma(x)^{\frac{1}{2}}$ is a diagonal matrix, by the chain rule and product rule, it is straightforward to calculate
\begin{align*}
\|\nabla_{v_{1}}\Sigma(x)^{\frac{1}{2}}\|_{{\rm HS}}=\frac{1}{2}\left\|\left(\Sigma(x)^{\frac{1}{2}}\right)^{-1}\nabla_{v_{1}}\Sigma(x)\right\|_{{\rm HS}}\leq\sqrt{d}\|Q\|_{{\rm HS}}^{3}|v_{1}|,
\end{align*}
\begin{align*}
\|\nabla_{v_{2}}\nabla_{v_{1}}\Sigma(x)^{\frac{1}{2}}\|_{{\rm HS}}\leq\sqrt{d}\|Q\|_{{\rm HS}}^{4}(1+\gamma^{-1}+\gamma^{-3})|v_{1}||v_{2}|,
\end{align*}
\begin{align*}
\|\nabla_{v_{3}}\nabla_{v_{2}}\nabla_{v_{1}}\Sigma(x)^{\frac{1}{2}}\|_{{\rm HS}}\leq3\sqrt{d}\|Q\|_{{\rm HS}}^{5}(2+\gamma^{-3}+\gamma^{-5})|v_{1}||v_{2}||v_{3}|.
\end{align*}
\end{proof}

\subsection{Proof of Lemma \ref{second moment}}
Recall (\ref{SGD}), it is easily seen that
\begin{align*}
\mathbb{E}|w_{k}|^{4}
=&\mathbb{E}|w_{k-1}|^{4}-4\eta\mathbb{E}\big[|w_{k-1}|^{2}\langle\nabla\psi(w_{k-1},\zeta_{k}),w_{k-1}\rangle\big]+2\eta^{2}\mathbb{E}\big[|w_{k-1}|^{2}|\nabla\psi(w_{k-1},\zeta_{k})|^{2}\big]\\
&-4\eta^{3}\mathbb{E}\big[|\nabla\psi(w_{k-1},\zeta_{k})|^{2}\langle\nabla\psi(w_{k-1},\zeta_{k})
,w_{k-1}\rangle\big]+4\eta^{2}\mathbb{E}\big[\langle\nabla\psi(w_{k-1},\zeta_{k}),w_{k-1}\rangle^{2}\big]\\
&+\eta^{4}\mathbb{E}|\nabla\psi(w_{k-1},\zeta_{k})|^{4}.
\end{align*}
Since $\zeta_{k}$ is independent of $w_{k-1}$ for any $k\geq1$, \eqref{integration} yields
\begin{align*}
&\mathbb{E}\big[|w_{k-1}|^{2}\langle\nabla\psi(w_{k-1},\zeta_{k}),w_{k-1}\rangle\big]
=\E\left[\E\left[|w_{k-1}|^{2}\langle\nabla(w_{k-1},\zeta_{k}),w_{k-1}\rangle|w_{k-1}\right]\right] \\
=& \E \big[|w_{k-1}|^{2}\Ll\nabla P(w_{k-1})-\nabla P(0),w_{k-1}\Rr \big]+ \E \big[|w_{k-1}|^{2}\Ll\nabla P(0),w_{k-1}\Rr \big]\\
\ge& \theta_{0}\mathbb{E}|w_{k-1}|^{4}+ \E \big[|w_{k-1}|^{2}\Ll\nabla P(0),w_{k-1}\Rr \big],
\end{align*}
which implies
\begin{align*}
-4\eta\mathbb{E}\big[|w_{k-1}|^{2}\langle w_{k-1},\nabla\psi(w_{k-1},\zeta_{k})\rangle\big]
\leq&-4\theta_{0}\eta\mathbb{E}|w_{k-1}|^{4}+4\eta\E \big[|w_{k-1}|^{3}|\nabla P(0)|\big]\\
\leq&-3\theta_{0}\eta\mathbb{E}|w_{k-1}|^{4}+\frac{27|\nabla P(0)|^{4}}{\theta_{0}^{3}}\eta,
\end{align*}
where the last inequality comes from Young's inequality. In addition, by the Cauchy-Schwarz inequality and (\ref{stofo}), we have
\begin{align*}
&\E \big[|w_{k-1}|^{2}|\nabla\psi(w_{k-1},\zeta_{k})|^{2}\big]=\E \big[|w_{k-1}|^{2}\E[|\nabla(w_{k-1},\zeta_{k})|^{2}|w_{k-1}]\big]\\
\leq&\mathbb{E}[2\kappa^{2}|w_{k-1}|^{4}+2\ell^{2}_{0}|w_{k-1}|^{2}\big]\leq2(\kappa^{2}+\ell^{2}_{0})\mathbb{E}|w_{k-1}|^{4}+2\ell^{2}_{0},
\end{align*}
\begin{align*}
\mathbb{E}\big[|\nabla\psi(w_{k-1},\zeta_{k})|^{2}\langle\nabla\psi(w_{k-1},\zeta_{k}),w_{k-1}\rangle\big]
\leq&\mathbb{E}\big[|\nabla\psi(w_{k-1},\zeta_{k})|^{3}|w_{k-1}|\big]\\
\leq&4(\kappa^{3}+\ell^{3}_{0})(\mathbb{E}|w_{k-1}|^{4}+1),
\end{align*}
\begin{align*}
\mathbb{E}\big[\langle\nabla\psi(w_{k-1},\zeta_{k}),w_{k-1}\rangle^{2}\big]\leq2(\kappa^{2}+\ell^{2}_{0})\mathbb{E}|w_{k-1}|^{4}+2\ell^{2}_{0},
\end{align*}
\begin{align*}
\E |\nabla\psi(w_{k-1},\zeta_{k})|^{4}=\E\left[\E[|\nabla\psi(w_{k-1},\zeta_{k})|^{4}|w_{k-1}]\right]\leq 8\kappa^{4}\mathbb{E}|w_{k-1}|^{4}+8 \ell^{4}_{0}.
\end{align*}
These imply
\begin{align*}
\mathbb{E}|w_{k}|^{4}\leq&\left[1-3\theta_{0}\eta+12(\kappa^{2}+\ell^{2}_{0})\eta^{2}+16(\kappa^{3}+\ell^{3}_{0})\eta^{3}+8\kappa^{4}\eta^{4}\right]
\mathbb{E}|w_{k-1}|^{4}\\
&+\frac{27|\nabla P(0)|^{4}}{\theta_{0}^{3}}\eta+12\ell^{2}_{0}\eta^{2}+16(\kappa^{3}+\ell^{3}_{0})\eta^{3}+8 \ell^{4}_{0}\eta^{4}\\
\leq&(1-\theta_{0}\eta)\mathbb{E}|w_{k-1}|^{4}+C_{\theta,\kappa,\ell_{0}^{4}}\eta,
\end{align*}
where the second inequality is by the fact $\eta\leq\min\{1,\frac{\theta_{0}}{2(10+7\kappa^{4}+7\ell^{4}_{0})}\}$. Therefore,
\begin{align*}
\mathbb{E}|w_{k}|^{4}
\le \big[1-\theta_{0}\eta\big]^{k}|w_{0}|^{4}+C_{\theta,\kappa,\ell_{0}^{4}}\eta\sum_{j=0}^{k-1}(1-\theta_{0}\eta)^{j}\leq|w_{0}|^{4}+C_{\theta,\kappa,\ell_{0}^{4}}.
\end{align*}
\qed

\subsection{Proof of Lemma \ref{fourth}}
Recall (\ref{SME}), by It\^{o}'s formula, (\ref{stodiffu}), (\ref{integration}) and the Young inequality, we have
\begin{align*}
\frac{d}{ds}\E |\hat X^{x}_{s}|^{2}=&2\E\left[\Ll \hat X_{s}^{x}, -\nabla P(\hat X_{s}^{x})\Rr\right]+\eta\mathbb{E}\left[\|\Sigma(\hat X_{s}^{x})^{\frac{1}{2}}\|^{2}_{\rm HS}\right]\\
\leq&-2\E\Ll \hat X_{s}^{x},\nabla P(\hat X_{s}^{x})-\nabla P(0)\Rr+2\E \left[|\hat X_{s}^{x}||\nabla P(0)|\right]+2\eta\mathbb{E}\left[\kappa^{2}|\hat X_{s}^{x}|^{2}+\ell^2_{0}\right]\\
\leq&-\frac{3}{2}\theta_{0}\mathbb{E}|\hat X_{s}^{x}|^{2}+\frac{2|\nabla P(0)|^{2}}{\theta_{0}}+2\kappa^{2}\eta\mathbb{E}\left[|\hat X_{s}^{x}|^{2}\right]+2\eta\ell^2_{0}\\
\leq&-\theta_{0}\mathbb{E}|\hat X_{s}^{x}|^{2}+\frac{2|\nabla P(0)|^{2}}{\theta_{0}}+2\ell^2_{0},
\end{align*}
where the last inequality is by the fact $\eta\leq\min\{1,\frac{\theta_{0}}{4\kappa^{2}}\}$. This inequality, together with $\hat X^{x}_{0}=x,$ implies
\begin{align*}
\E |\hat X^{x}_{t}|^{2}\leq& e^{-\theta_{0}t}|x|^{2}+\left(\frac{2|\nabla P(0)|^{2}}{\theta_{0}}+2\ell^2_{0}\right)\int_{0}^{t}e^{-\theta_{0}(t-s)}ds\leq|x|^{2}+\frac{2|\nabla P(0)|^{2}+2\theta_{0}\ell_{0}^{2}}{\theta_{0}^{2}},
\end{align*}
\eqref{moment} is proved.

By the Cauchy-Schwarz inequality, the It$\hat{o}$ isometry, (\ref{shiftlin}) and (\ref{stodiffu}), it is easy to verify
\begin{align*}
\mathbb{E}|\hat X_{t}^{x}-x|^{2}\leq&2\mathbb{E}\left|\int_{0}^{t}-\nabla
P(\hat{X}_{r}^x)dr\right|^{2}+2\mathbb{E}\left|\int_{0}^{t}\left(\eta\Sigma(\hat{X}_{r})\right)^{\frac{1}{2}}dB_{r}\right|^{2}\nonumber\\
\leq&2t\int_{0}^{t}\mathbb{E}|\nabla P(\hat{X}_{r}^{x})|^{2}dr+2\eta\int_{0}^{t}\mathbb{E}\left\|\Sigma(\hat{X}_{r})^{\frac{1}{2}}\right\|^{2}_{{\rm {HS}}}dr\\
\leq&C_{\kappa,\ell_{0}^{2}}(t+\eta)\int_{0}^{t}\left(\mathbb{E}|\hat{X}_{r}^{x}|^{2}+1\right)dr,
\end{align*}
which, together with \eqref{moment}, implies (\ref{moment-1}).
\qed

\subsection{Proof of Lemma \ref{SGDtaylor}}
Recall SDE (\ref{SME}), for any $f\in\mathcal{C}_{b}^{2}(\mathbb{R}^{d}),$ we have
\begin{align*}
\mathcal{A}^{X}f(x)=\lim_{t\rightarrow0}\frac{\mathbb{E}f(X_{t}^{x})-f(x)}{t}=-\eta \langle\nabla f(x),\nabla P(x)\rangle+\frac{1}{2}\eta^{2}\langle\nabla^{2}f(x),\Sigma(x)\rangle_{{\rm HS}},
\end{align*}
Then, for any $u_k(x)=\E h(X^x_{\eta k})$ with $k\geq1,$ we have
\begin{align*}
&\mathbb{E}\int_{0}^{1}\mathcal{A}^{X}u_{k}(X_{s}^{x})ds\\
=&-\eta\mathbb{E}\int_{0}^{1}\langle\nabla u_{k}(\hat{X}_{\eta s}^{x}),\nabla P(\hat{X}_{\eta s}^{x})\rangle\dif s+\frac{1}{2}\eta^{2}\mathbb{E}\int_{0}^{1}\langle\nabla^{2}u_{k}(\hat{X}_{\eta s}^{x}),\Sigma(\hat{X}_{\eta s}^{x})\rangle_{{\rm HS}} \dif s\\
=&-\mathbb{E}\int_{0}^{\eta}\langle\nabla u_{k}(\hat{X}_{s}^{x}),\nabla P(\hat{X}_{s}^{x})\rangle\dif s+\frac{1}{2}\eta\mathbb{E}\int_{0}^{\eta}\langle\nabla^{2}u_{k}(\hat{X}_{s}^{x}),\Sigma(\hat{X}_{s}^{x})\rangle_{{\rm HS}} \dif s.
\end{align*}
Recall (\ref{SGD}), by Taylor expansion, we have
\begin{align*}
&\mathcal{A}^{Y}u_{k}(x)=\mathbb{E}[u_{k}(Y_{1}^{x})-u_{k}(x)]=\mathbb{E}[u_{k}(w_{1}^{x})-u_{k}(x)]\\
=&\mathbb{E}[\langle\nabla u_{k}(x),-\eta\nabla\psi(x,\zeta)\rangle]+\frac{1}{2}\eta^{2}\mathbb{E}\langle\nabla^{2}u_{k}(x),\nabla\psi(x,\zeta)(\nabla\psi(x,\zeta))^{T}\rangle_{{\rm HS}}+E[\mathcal{R}^{u_{k}}(x)]\\
=&\langle\nabla u_{k}(x),-\eta\nabla P(x)\rangle+\frac{1}{2}\eta^{2}\langle\nabla^{2}u_{k}(x),\Sigma(x)+\nabla P(x)\big(\nabla P(x)\big)^{T}\rangle_{{\rm HS}}+E[\mathcal{R}^{u_{k}}(x)],
\end{align*}
where
\begin{align*}
\mathcal{R}^{u_{k}}(x)=\eta^{2}\int_{0}^{1}\int_{0}^{r}\langle\nabla^{2}u_{k}(x-s\eta\nabla\psi(x,\zeta))-\nabla^{2}u_{k}(x),(\nabla\psi(x,\zeta))(\nabla\psi(x,\zeta))^{T}\rangle dsdr.
\end{align*}
Therefore, we have
\begin{align*}
\big|\mathbb{E}\int_{0}^{1}\big[\mathcal{A}^{Y}u_{k}(X_{s}^{x})-\mathcal{A}^{Y} u_{k}(x)\big]\dif s\big|\leq\mathcal{J}_{1}+\mathcal{J}_{2}+\E|\mathcal{R}^{u_{k}}(x)|,
\end{align*}
where
\begin{align*}
\mathcal{J}_{1}:=\Big|\mathbb{E}\int_{0}^{\eta}\langle\nabla u_{k}(\hat{X}_{s}^{x}),\nabla P(\hat{X}_{s}^{x})\rangle\dif s&-\eta\langle\nabla u_{k}(x),\nabla P(x)\rangle\\
&+\frac{1}{2}\eta^{2}\langle\nabla^{2} u_{k}(x),\nabla P(x)\big(\nabla P(x)\big)^{T}\rangle_{{\rm HS}}\Big|,
\end{align*}
\begin{align*}
\mathcal{J}_{2}:=\frac{\eta}{2}\Big|\mathbb{E}\int_{0}^{\eta}\langle\nabla^{2}u_{k}(\hat{X}_{s}^{x}),\Sigma(\hat{X}_{s}^{x})\rangle_{{\rm HS}} \dif s-\eta\langle\nabla^{2}u_{k}(x),\Sigma(x)\rangle_{{\rm HS}}\Big|.
\end{align*}

For $\mathcal{J}_{1},$ we have
\begin{align*}
\mathcal{J}_{1}\leq&\big|\mathbb{E}\int_{0}^{\eta}\langle\nabla u_{k}(\hat{X}_{s}^{x}),\nabla P(\hat{X}_{s}^{x})-\nabla P(x)\rangle\dif s\big|\\
&+\big|\mathbb{E}\int_{0}^{\eta}\langle\nabla u_{k}(\hat{X}_{s}^{x})-\nabla u_{k}(x),\nabla P(x)\rangle\dif s+\frac{1}{2}\eta^{2}\langle\nabla^{2} u_{k}(x),\nabla P(x)\big(\nabla P(x)\big)^{T}\rangle_{{\rm HS}}\big|\\
:=&\mathcal{J}_{11}+\mathcal{J}_{12}.
\end{align*}
By (\ref{gradient2}), (\ref{Lip}), the Cauchy-Schwarz inequality and (\ref{moment-1}), one has
\begin{align*}
\mathcal{J}_{11}\leq&\kappa e^{-\frac{\theta_{0}}{8}\eta k}\int_{0}^{\eta}\mathbb{E}\big|\hat{X}_{s}^{x}-x\big|\dif s\\
\leq& C_{\theta,\kappa,\ell^2_{0}}(1+|x|^{2})e^{-\frac{\theta_{0}}{8}\eta k}\int_{0}^{\eta}(s^{2}+\eta s)^{\frac{1}{2}}\dif s\leq C_{\theta,\kappa,\ell^2_{0}}(1+|x|^{2})\eta^{2}e^{-\frac{\theta_{0}}{8}\eta k}.
\end{align*}
Notice that
\begin{align*}
&\mathbb{E}\langle\nabla u_{k}(\hat{X}_{s}^{x})-\nabla u_{k}(x),\nabla P(x)\rangle\\
=&-\int_{0}^{s}\mathbb{E}\langle\nabla^{2} u_{k}(x),\nabla P(\hat{X}_{v}^{x})\big(\nabla P(x)\big)^{T}\rangle_{{\rm HS}}\dif v\\
&+\int_{0}^{1}\mathbb{E}\langle\nabla^{2} u_{k}\big(x+r(\hat{X}_{s}^{x}-x)\big)-\nabla^{2} u_{k}(x),(\hat{X}_{s}^{x}-x)\big(\nabla P(x)\big)^{T}\rangle_{{\rm HS}}\dif r.
\end{align*}
By (\ref{sec}), (\ref{Lip}), (\ref{shiftlin}) and (\ref{third}), we have
\begin{align*}
\mathcal{J}_{12}\leq&\big|\mathbb{E}\int_{0}^{\eta}\int_{0}^{s}\mathbb{E}\langle\nabla^{2} u_{k}(x),\big(\nabla P(\hat{X}_{v}^{x})-\nabla P(x)\big)\big(\nabla P(x)\big)^{T}\rangle_{{\rm HS}}\dif v\dif s\big|\\
&+\big|\int_{0}^{\eta}\int_{0}^{1}\mathbb{E}\langle\nabla^{2} u_{k}\big(x+r(\hat{X}_{s}^{x}-x)\big)-\nabla^{2} u_{k}(x),(\hat{X}_{s}^{x}-x)\big(\nabla P(x)\big)^{T}\rangle_{{\rm HS}}\dif r\dif s\big|\\
\leq&C_{\theta,\kappa,\ell^2_{0}}(1+|x|)(1+\frac{\sqrt{d}}{\delta})(1+\frac{1}{\sqrt{\eta^{2}k}})e^{-\frac{{\theta_{0}}}{8}\eta k}\int_{0}^{\eta}\int_{0}^{s}\mathbb{E}\big|\hat{X}_{v}^{x}-x\big|\dif v\dif s\\
&+C_{\theta,\kappa,\ell^2_{0}}(1+|x|)(1+\frac{d}{\delta^{2}})\big(1+\frac{1}{\eta^{2}k}+\frac{1}{(\eta k)^{\frac{5}{4}}}\big)e^{-\frac{\theta_{0}}{8} \eta k}\int_{0}^{\eta}\int_{0}^{1}r\mathbb{E}|\hat{X}_{s}^{x}-x|^{2}\dif r\dif s.
\end{align*}
Then, by the Cauchy-Schwarz inequality and (\ref{moment-1}), we can obtain
\begin{align*}
\mathcal{J}_{12}\leq&C_{\theta,\kappa,\ell^2_{0}}(1+|x|^{2})(1+\frac{\sqrt{d}}{\delta})(1+\frac{1}{\sqrt{\eta^{2}k}})e^{-\frac{{\theta_{0}}}{8}\eta k}\int_{0}^{\eta}\int_{0}^{s}(v^{2}+\eta v)^{\frac{1}{2}}\dif v\dif s\\
&+C_{\theta,\kappa,\ell^2_{0}}(1+|x|^{3})(1+\frac{d}{\delta^{2}})\big(1+\frac{1}{\eta^{2}k}+\frac{1}{(\eta k)^{\frac{5}{4}}}\big)e^{-\frac{\theta_{0}}{8} \eta k}\int_{0}^{\eta}\int_{0}^{1}r(s^{2}+\eta s)\dif r\dif s\\
\leq&C_{\theta,\kappa,\ell^2_{0}}(1+|x|^{3})(1+\frac{d}{\delta^{2}})\big(1+\frac{1}{\eta^{2}k}+\frac{1}{(\eta k)^{\frac{5}{4}}}\big)\eta^{3}e^{-\frac{\theta_{0}}{8} \eta k}.
\end{align*}
Hence, the fact $\eta\leq1$ implies
\begin{align}
\mathcal{J}_{1}\leq&C_{\theta,\kappa,\ell^2_{0}}(1+|x|^{3})(1+\frac{d}{\delta^{2}})\big(1+\frac{1}{\eta k}+\frac{\eta}{(\eta k)^{\frac{5}{4}}}\big)\eta^{2}e^{-\frac{\theta_{0}}{8} \eta k}\label{computation}.
\end{align}

For $\mathcal{J}_{2},$ recall $\eta\Sigma(x)=\mathbb{E}\big[V_{\eta}(x,I)V_{\eta}(x,\zeta)^{T}\big]$ with $V_{\eta}(x,\zeta)=\sqrt{\eta}\big(\nabla P(x)-\nabla\psi(x,\zeta)\big),$ we have
\begin{align*}
\eta\Sigma(x)-\eta\Sigma(y)=&\mathbb{E}\big[V_{\eta}(x,\zeta)\big(V_{\eta}(x,\zeta)-V_{\eta}(y,\zeta)\big)^{T}\big]+\mathbb{E}\big[\big(V_{\eta}(x,\zeta)-V_{\eta}(y,\zeta)\big)V_{\eta}(y,\zeta)^{T}\big].
\end{align*}
By (\ref{stofo}), the Cauchy-Schwarz inequality, (\ref{stolin}) and (\ref{Lip}), we further have
\begin{align}\label{diffuqua}
\mathbb{E}|V_{\eta}(x,\zeta)|^{2}=\eta\left(\mathbb{E}|\nabla\psi(x,\zeta)|^{2}-|\nabla P(x)|^{2}\right)\leq\eta\mathbb{E}|\nabla\psi(x,\zeta)|^{2}\leq2(\kappa^{2} |x|^{2}+\ell^{2}_{0})\eta,
\end{align}
\begin{align}
\mathbb{E}\big[|V_{\eta}(x,\zeta)|\big|V_{\eta}(x,\zeta)-V_{\eta}(y,\zeta)\big|\big]\leq&C_{\kappa,\ell^{2}_{0}}(1+|x|)|x-y|\eta\label{diffulip}.
\end{align}
Then, (\ref{diffuqua}), (\ref{third}), (\ref{diffulip}) and (\ref{sec}) imply
\begin{align*}
\mathcal{J}_{2}\leq&\frac{1}{2}\mathbb{E}\big|\int_{0}^{\eta}\langle\nabla^{2}u_{k}(\hat{X}_{s}^{x})-\nabla^{2}u_{k}(x),\eta\Sigma(x)\rangle_{{\rm HS}} \dif s\big|\\
&+\frac{1}{2}\mathbb{E}\big|\int_{0}^{\eta}\langle\nabla^{2}u_{k}(X_{s}^{x}),\eta\Sigma(\hat{X}_{s}^{x})-\eta\Sigma(x)\rangle_{{\rm HS}} \dif s\big|\\
\leq&C_{\theta,\kappa,\ell^2_{0}}\left(1+|x|^{2}\right)(1+\frac{d}{\delta^{2}})\big(1+\frac{1}{\eta^{2}k}+\frac{1}{(\eta k)^{\frac{5}{4}}}\big)e^{-\frac{\theta_{0}}{8} \eta k}\eta\int_{0}^{\eta}\mathbb{E}|\hat{X}_{s}^{x}-x|\dif s\\
&+C_{\theta,\kappa,\ell^2_{0}}(1+\frac{\sqrt{d}}{\delta})(1+\frac{1}{\sqrt{\eta^{2}k}})e^{-\frac{{\theta_{0}}}{8}\eta k}\eta\int_{0}^{\eta}\mathbb{E}\left[\big[(1+|x|+|\hat{X}_{s}^{x}|)|\hat{X}_{s}^{x}-x|\big]\right]\dif s,
\end{align*}
Following the Cauchy-Schwarz inequality, (\ref{moment}) and (\ref{moment-1}), by the same argument as the proof of (\ref{computation}), one has
\begin{align*}
\mathcal{J}_{2}\leq&C_{\theta,\kappa,\ell^2_{0}}(1+|x|^{3})(1+\frac{d}{\delta^{2}})\big(1+\frac{1}{\eta k}+\frac{\eta}{(\eta k)^{\frac{5}{4}}}\big)\eta^{2}e^{-\frac{\theta_{0}}{8} \eta k}.
\end{align*}

In addition, by (\ref{third}), H\"{o}lder's inequality and (\ref{stofo}), we have
\begin{align*}
\mathbb{E}|\mathcal{R}^{u_{k}}(x)|\leq&C_{\theta,\kappa,\ell^3_{0}}(1+\frac{d}{\delta^{2}})\big(1+\frac{1}{\eta^{2}k}+\frac{1}{(\eta k)^{\frac{5}{4}}}\big)e^{-\frac{\theta_{0}}{8} \eta k}\eta^{3}\int_{0}^{1}\int_{0}^{r}(1+|x|^{3})\dif s\dif r\\
\leq&C_{\theta,\kappa,\ell^3_{0}}(1+|x|^{3})(1+\frac{d}{\delta^{2}})\big(1+\frac{1}{\eta k}+\frac{\eta}{(\eta k)^{\frac{5}{4}}}\big)\eta^{2}e^{-\frac{\theta_{0}}{8} \eta k}.
\end{align*}

Combining all of above, we have
\begin{align*}
&\big|\mathbb{E}\int_{0}^{1}\!\!\big[\mathcal{A}^{X}u_{k}(X_{s}^{x})-\mathcal{A}^{Y} u_{k}(x)\big]\dif s\big|\\
\leq&C_{\theta,\kappa,\ell^3_{0}}(1+|x|^{3})(1+\frac{d}{\delta^{2}})\big(1+\frac{1}{\eta k}+\frac{\eta}{(\eta k)^{\frac{5}{4}}}\big)\eta^{2}e^{-\frac{\theta_{0}}{8} \eta k}.
\end{align*}
\qed

\section{Proof of Lemma \ref{mainlem1}}\label{Mal}

Under {\bf Assumption A1}, we recall some preliminary of {Malliavin}
calculus and derive standard estimates related to Malliavin calculus and SDE, which will be applied to prove the
Lemma \ref{mainlem1} in {Subsection \ref{proofsgd}}.

\subsection{Malliavin calculus of SDE (\ref{SME})}
For simplicity, denote $B(x):=-\nabla P(x),$ $\sigma(x):=\left(\Sigma(x)\right)^{\frac{1}{2}}$. Then SDE
(\ref{SME}) can be written as the following form:
\begin{align}\label{ASDE}
d\hat{X}_{t}=B(\hat{X}_{t})dt+\sqrt{\eta}\sigma(\hat{X}_{t})dB_{t},\quad \hat{X}_{0}=x.
\end{align}

Moreover, {\bf Assumption A1} in {Subsection} \ref{SMESGD} can be rewritten as the following form:

{\bf Assumption A1} (i) There
exist $\theta_{0}>0$ and $\theta_{1},\theta_{2},\theta_{3},\theta_{4},\theta_{5}\geq0$ such that for any $v_{1},v_{2},v_{3},x\in\mathbb{R}^{d}$, $\nabla P(x)$ satisfies
\begin{align}\label{Adiss}
\langle v_{1},\nabla_{v_{1}}B(x)\rangle\leq-\theta_{0}|v_{1}|^{2}, \quad |\nabla_{v_{1}}\nabla_{v_{2}}B(x)|\leq\theta_{1}|v_{1}||v_{2}|,
\end{align}
\begin{align}\label{Aonthird}
|\nabla_{v_{1}}\nabla_{v_{2}}\nabla_{v_{3}}B(x)|\leq\theta_{2}|v_{1}||v_{2}|;
\end{align}
and that any $x,y\in\mathbb{R}^{d}$, $\sigma(x)$ satisfies
\begin{align}\label{diffusion1}
\|\nabla_{v_{1}}\sigma(x)\|_{{\rm HS}}\leq\theta_{3}|v_{1}|, \quad \|\nabla_{v_{2}}\nabla_{v_{1}}\sigma(x)\|_{{\rm HS}}\leq\theta_{4}|v_{1}||v_{2}|,
\end{align}
\begin{align}\label{diffusion2}
\|\nabla_{v_{3}}\nabla_{v_{2}}\nabla_{v_{1}}\sigma(x)\|_{{\rm HS}}\leq\theta_{5}|v_{1}||v_{2}||v_{3}|.
\end{align}

(ii) There exists $\delta>0$ such that for any $x\in\mathbb{R}^{d}$ and non-zero vector $\xi\in\mathbb{R}^{d}$, $\sigma(x)$ satisfies
\begin{align}\label{definite}
\xi^{T}\sigma(x)\xi\geq\delta|\xi|^{2}.
\end{align}

Under the {\bf Assumption A1}, there exists a unique solution to the SDE (\ref{ASDE}) and the SDE (\ref{ASDE}) has a
unique non-degenerate invariant measure (see, e.g., \cite{BR95,Cer96,DPG01,Eva14}).

Next, we briefly recall Bismut's approach to Malliavin calculus, which is crucial to prove Lemma \ref{mainlem1}. Let $v\in\mathbb{R}^{d}$
and $\nabla_{v}\hat{X}_{t}^{x}$ is defined by
\begin{align*}
\nabla_{v}\hat{X}_{t}^{x}=\lim_{\epsilon\rightarrow0}\frac{\hat{X}_{t}^{x+\epsilon
v}-\hat{X}_{t}^{x}}{\epsilon},\quad t\geq0.
\end{align*}
The above limit exists and satisfies
\begin{align}\label{gradient SDE}
d\nabla_{v}\hat{X}_{t}^{x}=\nabla B(\hat{X}_{t}^{x})\nabla_{v}\hat{X}_{t}^{x}dt+\sqrt{\eta}\nabla\sigma(\hat{X}_{t})\nabla_{v}\hat{X}_{t}^{x}dB_{t}, \quad
\nabla_{v}\hat{X}_{0}^{x}=v.
\end{align}
Then, we use the notations $J_{s,t}^{x}$ with $0\leq s\leq t<\infty$ for the stochastic flow between time $s$ and $t,$ that is,
\begin{align*}
\nabla_{v}\hat{X}_{t}^{x}=J_{0,t}^{x}v.
\end{align*}
Note that we have the important cocycle property $J_{0,s}^{x}J_{s,t}^{x}=J_{0,t}^{x}$ for all $0\leq s\leq t<\infty$. For a more thorough discussion on stochastic flow, we refer the reader to \cite{K84,K97,B99,HM06} and the references therein.

For $v_{1},v_{2}\in\mathbb{R}^{d},$ we can define $\nabla_{v_{2}}\nabla_{v_{1}}\hat{X}_{t}^{x},$ which satisfies
\begin{align}\label{secequation}
d\nabla_{v_{2}}\nabla_{v_{1}}\hat{X}_{t}^{x}=&\nabla B(\hat{X}_{t}^{x})\nabla_{v_{2}}\nabla_{v_{1}}\hat{X}_{t}^{x}dt+\nabla^{2}B(\hat{X}_{t}^{x})\nabla_{v_{2}}\hat{X}_{t}^{x}\nabla_{v_{1}}\hat{X}_{t}^{x}dt\nonumber\\
&+\sqrt{\eta}\nabla\sigma(\hat{X}_{t}^{x})\nabla_{v_{2}}\nabla_{v_{1}}\hat{X}_{t}^{x}dB_{t}+\sqrt{\eta}\nabla^{2}\sigma(\hat{X}_{t}^{x})\nabla_{v_{2}}\hat{X}_{t}^{x}\nabla_{v_{1}}\hat{X}_{t}^{x}dB_{t},
\end{align}
with $\nabla_{v_{2}}\nabla_{v_{1}}\hat{X}_{0}^{x}=0.$ Moreover, For $v_{1},v_{2},v_{3}\in\mathbb{R}^{d},$ we can similarly define $\nabla_{v_{3}}\nabla_{v_{2}}\nabla_{v_{1}}\hat{X}_{t}^{x}$ from above equation. Then, we have the following estimates:

\begin{lemma}
For all  $x,v,v_{1},v_{2},v_{3}\in\mathbb{R}^{d}$, as $\eta\leq\min\{1,\frac{\theta_{0}}{4\theta_{3}^{2}}\}$, we have
\begin{align}\label{grad3}
\mathbb{E}|\nabla_{v}\hat{X}_{t}^{x}|^{8}\leq e^{-\theta_{0} t}|v|^{8},
\end{align}
\begin{align}\label{gradsec}
\mathbb{E}|\nabla_{v_{2}}\nabla_{v_{1}}\hat{X}_{t}^{x}|^{4}
\leq C_{\theta}e^{-\frac{\theta_{0}}{2}t}|v_{1}|^{4}|v_{2}|^{4},
\end{align}
and
\begin{align}\label{gradthird}
\mathbb{E}|\nabla_{v_{3}}\nabla_{v_{2}}\nabla_{v_{1}}\hat{X}_{t}^{x}|^{2}
\leq C_{\theta}e^{-\frac{\theta_{0}}{4}t}|v_{1}|^{2}|v_{2}|^{2}|v_{3}|^{2}.
\end{align}
\end{lemma}
\begin{proof}
Recalling (\ref{gradient SDE}), by It\^{o}'s formula, (\ref{Adiss}) and (\ref{diffusion1}), we have
\begin{align*}
\frac{d}{ds}\mathbb{E}|\nabla_{v}\hat{X}_{s}^{x}|^{8}=&8\mathbb{E}[|\nabla_{v}\hat{X}_{s}^{x}|^{6}\langle\nabla B(\hat{X}_{s}^{x})\nabla_{v}\hat{X}_{s}^{x},\nabla_{v}\hat{X}_{s}^{x}\rangle]+4\eta\mathbb{E}[|\nabla_{v}\hat{X}_{s}^{x}|^{6}\|\nabla\sigma(\hat{X}_{s}^{x})\nabla_{v}\hat{X}_{s}^{x}\|_{{\rm HS}}^{2}]\\
&+24\eta\mathbb{E}[|\nabla_{v}\hat{X}_{s}^{x}|^{4}|\nabla\sigma(\hat{X}_{s}^{x})\nabla_{v}\hat{X}_{s}^{x}\nabla_{v}\hat{X}_{s}^{x}|^{2}]\\
\leq&-4(2\theta_{0}-7\theta_{3}^{2}\eta)\mathbb{E}[|\nabla_{v}\hat{X}_{s}^{x}|^{8}]\leq\theta_{0}\mathbb{E}[|\nabla_{v}\hat{X}_{s}^{x}|^{8}],
\end{align*}
where the last inequality is by the fact $\eta\leq\frac{\theta_{0}}{4\theta_{3}^{2}}$. This inequality, together with $\nabla_{v}X_{0}^{x}=v,$ implies
\begin{align*}
\mathbb{E}|\nabla_{v}\hat{X}_{t}^{x}|^{8}\leq e^{-\theta_{0} t}|v|^{8}.
\end{align*}

Using It\^{o}'s formula to $\varsigma(t)=\nabla_{v_{2}}\nabla_{v_{1}}\hat{X}_{t}^{x}$, by \eqref{secequation}, the Cauchy-Schwarz inequality, and {\bf Assumption A1}, we have
\begin{align*}
\frac{d}{ds}\mathbb{E}|\varsigma(s)|^{4}=&4\mathbb{E}\big[|\varsigma(s)|^{2}\langle\nabla B(\hat{X}_{s}^{x})\varsigma(s)+\nabla^{2}B(\hat{X}_{s}^{x})\nabla_{v_{2}}\hat{X}_{s}^{x}\nabla_{v_{1}}\hat{X}_{s}^{x},\varsigma(s)\rangle\big]\\
&+2\eta\mathbb{E}\big[|\varsigma(s)|^{2}\|\nabla\sigma(\hat{X}_{s}^{x})\varsigma(s)+\nabla^{2}\sigma(\hat{X}_{s}^{x})\nabla_{v_{2}}\hat{X}_{s}^{x}\nabla_{v_{1}}\hat{X}_{s}^{x}\|_{{\rm HS}}^{2}\big]\\
&+4\eta\mathbb{E}\big[|(\nabla\sigma(\hat{X}_{s}^{x})\varsigma(s)+\nabla^{2}\sigma(\hat{X}_{s}^{x})\nabla_{v_{2}}\hat{X}_{s}^{x}\nabla_{v_{1}}\hat{X}_{s}^{x})\varsigma(s)|^{2}\big]\\
\leq&-4\theta_{0}\mathbb{E}|\varsigma(s)|^{4}+4\theta_{1}\mathbb{E}[|\nabla_{v_{2}}\hat{X}_{s}^{x}||\nabla_{v_{1}}\hat{X}_{s}^{x}||\varsigma(s)|^{3}]\\
&+12\eta\mathbb{E}\big[|\varsigma(s)|^{2}\big(\|\nabla\sigma(\hat{X}_{s}^{x})\varsigma(s)\|_{{\rm HS}}^{2}+\|\nabla^{2}\sigma(\hat{X}_{s}^{x})\nabla_{v_{2}}\hat{X}_{s}^{x}\nabla_{v_{1}}\hat{X}_{s}^{x}\|_{{\rm HS}}^{2}\big)\big]\\
\leq&-4(\theta_{0}-3\theta_{3}^{2}\eta)\mathbb{E}|\varsigma(s)|^{4}+4\theta_{1}\mathbb{E}[|\nabla_{v_{2}}\hat{X}_{s}^{x}||\nabla_{v_{1}}\hat{X}_{s}^{x}||\varsigma(s)|^{3}]\\
&+12\theta_{4}^{2}\eta\mathbb{E}\big[|\varsigma(s)|^{2}|\nabla_{v_{2}}\hat{X}_{s}^{x}|^{2}|\nabla_{v_{1}}X_{s}^{x}|^{2}\big].
\end{align*}
By Young's inequality, Cauchy-Schwarz inequality and (\ref{grad3}), the fact $\eta\leq\min\{1,\frac{\theta_{0}}{4\theta_{3}^{2}}\}$ implies
\begin{align*}
\frac{d}{ds}\mathbb{E}|\varsigma(s)|^{4}\leq&-4(\frac{7}{8}\theta_{0}-3\theta_{3}^{2}\eta)\mathbb{E}|\varsigma(s)|^{4}
+\frac{12^{3}\theta_{1}^{4}}{\theta_{0}^{3}}\mathbb{E}[|\nabla_{v_{2}}\hat{X}_{s}^{x}|^{4}|\nabla_{v_{1}}\hat{X}_{s}^{x}|^{4}]\\
&+\frac{12^{2}\theta_{4}^{4}}{\theta_{0}}\mathbb{E}\big[|\nabla_{v_{2}}\hat{X}_{s}^{x}|^{4}|\nabla_{v_{1}}X_{s}^{x}|^{4}\big]\\
\leq&-\frac{\theta_{0}}{2}\mathbb{E}|\varsigma(s)|^{4}+C_{\theta}e^{-\theta_{0}t}|v_{1}|^{4}|v_{2}|^{4}.
\end{align*}
This inequality, together with $\varsigma(0)=0,$ implies
\begin{align*}
\mathbb{E}|\varsigma(t)|^{4}\leq C_{\theta}\int_{0}^{t}e^{-\theta_{0}s}|v_{1}|^{4}|v_{2}|^{4}e^{-\frac{\theta_{0}}{2}(t-s)}ds\leq C_{\theta}e^{-\frac{\theta_{0}}{2}t}|v_{1}|^{4}|v_{2}|^{4}.
\end{align*}

Furthermore, according to (\ref{Adiss})-(\ref{definite}), a similar calculation implies (\ref{gradthird}).
\end{proof}

Next, we use Bismut's approach to Malliavin calculus for SDE (\ref{SDE})(\cite{Nor86}). Let $u\in L_{loc}^{2}([0,\infty)\times(\Omega,\mathcal{F},\mathbb{P});\mathbb{R}^{d}),$ i.e., $\mathbb{E}\int_{0}^{t}|u(s)|^{2}\dif s<\infty$ for all $t>0.$ Further assume that $u$ is adapted to the filtration $(\mathcal{F}_{t})_{t\geq0}$ with $\mathcal{F}_{t}:=\sigma(B_{s}:0\leq s\leq t);$ i.e., $u(t)$ is $\mathcal{F}_{t}$ measurable for $t\geq0.$ Define
\begin{align}\label{bismutf}
U_{t}=\int_{0}^{t}u(s)\dif s,\quad t\geq0.
\end{align}
For a $t>0,$ let $F_{t}:\mathcal{C}([0,t],\mathbb{R}^{d})\rightarrow\mathbb{R}$ be a $\mathcal{F}_{t}$ measurable map. If the following limit exists
\begin{align*}
D_{U}F_{t}(B)=\lim_{\epsilon\rightarrow0}\frac{F_{t}(B+\epsilon U)-F_{t}(B)}{\epsilon}
\end{align*}
in $L^{2}((\Omega,\mathcal{F},\mathbb{P});\mathbb{R}),$ then $F_{t}(B)$ is said to be {\it Malliavin differentiable} and $D_{U}F_{t}(B)$ is called the Malliavin derivative of $F_{t}(B)$ in the direction $U$.

Let $F_{t}(B)$ and $G_{t}(B)$ both be Malliavin differentiable, then the following product rule holds:
\begin{align}\label{chain}
D_{U}(F_{t}(B)G_{t}(B))=F_{t}(B)D_{U}G_{t}(B)+G_{t}(B)D_{U}F_{t}(B).
\end{align}
When
\begin{align*}
F_{t}(B)=\int_{0}^{t}\langle a(s),\dif B(s)\rangle,
\end{align*}
where $a(s)=(a_{1}(s),\cdots,a_{d}(s))$ is a stochastic process adapted to the filtration $\mathcal{F}_{s}$ such that $\mathbb{E}\int_{0}^{t}|a(s)|^{2}\dif s<\infty$ for all $t>0$, it is easy to check that
\begin{align}\label{chain2}
D_{U}F_{t}(B)=\int_{0}^{t}\langle a(s),u(s)\rangle \dif s+\int_{0}^{t}\langle D_{U}a(s),\dif B_{s}\rangle.
\end{align}

Then, we consider the following integration by parts formula, which is called Bismut's formula. For Malliavin differentiable $F_{t}(B)$ such that $F_{t}(B),$ $D_{U}F_{t}(B)\in L^{2}((\Omega,\mathcal{F},\mathbb{P});\mathbb{R}),$ we have
\begin{align}\label{bismut}
\mathbb{E}[D_{U}F_{t}(B)]=\mathbb{E}\big[F_{t}(B)\int_{0}^{t}\langle u(s),\dif B_{s}\rangle\big].
\end{align}

Let $\phi\in {\rm Lip}(1)$ and let $F_{t}(B)=(F_{t}^{1}(B),\cdots,F_{t}^{d}(B))$ be a $d$-dimensional Malliavin differentiable functional. The following chain rule holds:
\begin{align*}
D_{U}\phi(F_{t}(B))=\langle\nabla\phi(F_{t}(B)),D_{U}F_{t}(B)\rangle=\sum_{i=1}^{d}\partial_{i}\phi(F_{t}(B))D_{U}F_{t}^{i}(B).
\end{align*}

Now, we come back to the SDE (\ref{SME}). Fixing $t\geq0$ and $x\in\mathbb{R}^{d},$ the solution $X_{t}^{x}$ is a $d$-dimensional functional of Brownian motion $(B_{s})_{0\leq s\leq t}.$

The following Malliavin derivative of $\hat{X}_{t}^{x}$ along the direction $U$ exists in $L^{2}((\Omega,\mathcal{F},\mathbb{P});\mathbb{R}^{d})$ and is defined by
\begin{align*}
D_{U}\hat{X}_{t}^{x}(B)=\lim_{\epsilon\rightarrow0}\frac{\hat{X}_{t}^{x}(B+\epsilon U)-\hat{X}_{t}^{x}(B)}{\epsilon}.
\end{align*}
We drop the $B$ in $D_{U}\hat{X}_{t}^{x}(B)$ and write $D_{U}\hat{X}_{t}^{x}=D_{U}\hat{X}_{t}^{x}(B)$ for simplicity. It satisfies the equation
\begin{align*}
\dif D_{U}\hat{X}_{t}^{x}=\nabla B(\hat{X}_{t}^{x})D_{U}\hat{X}_{t}^{x}\dif t+\eta^{\frac{1}{2}}\nabla\sigma(\hat{X}_{t}^{x})D_{U}\hat{X}_{t}^{x}\dif B_{t}+\eta^{\frac{1}{2}}\sigma(\hat{X}_{t}^{x})u(t)\dif t, \quad D_{U}\hat{X}_{0}^{x}=0,
\end{align*}
and the equation has a unique solution:
\begin{align*}
D_{U}\hat{X}_{t}^{x}=\int_{0}^{t}J_{r,t}^{x}\eta^{\frac{1}{2}}\sigma(\hat{X}_{r}^{x})u(r)\dif r.
\end{align*}
Noticing that $\nabla_{v}\hat{X}_{t}^{x}=J_{0,t}^{x}v,$ if we take
\begin{align}\label{malliavin1}
u(s)=\frac{1}{t}\eta^{-\frac{1}{2}}\sigma(\hat{X}_{s}^{x})^{-1}\nabla_{v}\hat{X}_{s}^{x},\quad 0\leq s\leq t,
\end{align}
then (\ref{definite}) and (\ref{grad3}) imply $u\in L_{loc}^{2}([0,\infty)\times(\Omega,\mathcal{F},\mathbb{P});\mathbb{R}^{d}).$ Since $\nabla_{v}\hat{X}_{r}^{x}=J_{0,r}^{x}v$ and $J_{0,r}^{x}J_{r,t}^{x}=J_{0,t}^{x},$ for all $0\leq r\leq t,$ we have
\begin{align}\label{malliavin2}
D_{U}\hat{X}_{t}^{x}=\nabla_{v}\hat{X}_{t}^{x}
\end{align}
and
\begin{align}\label{ma1}
D_{U}\hat{X}_{s}^{x}=\frac{s}{t}\nabla_{v}\hat{X}_{s}^{x},\quad 0\leq s\leq t.
\end{align}
Let $v_{1},v_{2}\in\mathbb{R}^{d},$ and define $u_{i}$ and $U_{i}$ as (\ref{malliavin1}) and (\ref{bismutf}), respectively, for $i=1,2.$ We can similarly define $D_{U_{2}}\nabla_{v_{1}}\hat{X}_{s}^{x},$ which satisfies the following equation: for $s\in[0,t],$
\begin{align}\label{malsec}
&\dif D_{U_{2}}\nabla_{v_{1}}\hat{X}_{s}^{x}\nonumber\\
=&\big[\nabla B(\hat{X}_{s}^{x})D_{U_{2}}\nabla_{v_{1}}\hat{X}_{s}^{x}+\nabla^{2}B(\hat{X}_{s}^{x})D_{U_{2}}\hat{X}_{s}^{x}\nabla_{v_{1}}\hat{X}_{s}^{x}
+\eta^{\frac{1}{2}}\nabla\sigma(\hat{X}_{s}^{x})\nabla_{v_{1}}\hat{X}_{s}^{x}u_{2}(s)\big]\dif s\nonumber\\
&+\eta^{\frac{1}{2}}\big[\nabla\sigma(\hat{X}_{s}^{x})D_{U_{2}}\nabla_{v_{1}}\hat{X}_{s}^{x}
+\nabla^{2}\sigma(\hat{X}_{s}^{x})D_{U_{2}}\hat{X}_{s}^{x}\nabla_{v_{1}}\hat{X}_{s}^{x}\big]\dif B_{s}\nonumber\\
=&\big[\nabla B(\hat{X}_{s}^{x})D_{U_{2}}\nabla_{v_{1}}\hat{X}_{s}^{x}+\frac{s}{t}\nabla^{2}B(\hat{X}_{s}^{x})\nabla_{v_{2}}\hat{X}_{s}^{x}\nabla_{v_{1}}\hat{X}_{s}^{x}\nonumber\\
&\qquad\qquad\qquad\qquad\qquad+\frac{1}{t}\nabla\sigma(\hat{X}_{s}^{x})\nabla_{v_{1}}\hat{X}_{s}^{x}\sigma(\hat{X}_{s}^{x})^{-1}\nabla_{v_{2}}\hat{X}_{s}^{x}\big]\dif s\nonumber\\
&+\eta^{\frac{1}{2}}\big[\nabla\sigma(\hat{X}_{s}^{x})D_{U_{2}}\nabla_{v_{1}}\hat{X}_{s}^{x}
+\frac{s}{t}\nabla^{2}\sigma(\hat{X}_{s}^{x})\nabla_{v_{2}}\hat{X}_{s}^{x}\nabla_{v_{1}}\hat{X}_{s}^{x}\big]\dif B_{s},
\end{align}
with $D_{U_{2}}\nabla_{v_{1}}\hat{X}_{0}^{x}=0,$ where the second equality is by (\ref{malliavin1}) and (\ref{ma1}).

For further use, we define
\begin{align*}
\mathcal{I}_{v_{1}}^{x}(t):=\frac{1}{t}\int_{0}^{t}\langle \eta^{-\frac{1}{2}}\sigma(\hat{X}_{s}^{x})^{-1}\nabla_{v_{1}}\hat{X}_{s}^{x},\dif B_{s}\rangle,
\end{align*}
\begin{align*}
\mathcal{R}_{v_{1},v_{2}}^{x}(t):=\nabla_{v_{2}}\nabla_{v_{1}}\hat{X}_{t}^{x}-D_{U_{2}}\nabla_{v_{1}}\hat{X}_{t}^{x}.
\end{align*}

Then, we have the following upper bounds on Malliavin derivatives.

\begin{lemma}
Let $v_{1},v_{2}\in\mathbb{R}^{d}$ and
\begin{align*}
U_{i,s}=\int_{0}^{s}u_{i}(r)\dif r, \quad 0\leq s\leq t,
\end{align*}
where $u_{i}(r)=\frac{1}{t}\eta^{-\frac{1}{2}}\sigma(\hat{X}_{r}^{x})^{-1}\nabla_{v_{i}}\hat{X}_{r}^{x}$ for $0\leq r\leq t$ and $i=1,2.$ Then, as $\eta\leq\min\{1,\frac{{\theta_{0}}}{48\theta_{3}^{2}}\},$ we have
\begin{align}\label{malgrad1}
\mathbb{E}|D_{U_{2}}\nabla_{v_{1}}\hat{X}_{s}^{x}|^{2}\leq C_{\theta}(1+\frac{d}{\delta^{2}t})e^{-\frac{{\theta_{0}}}{4}s}|v_{1}|^{2}|v_{2}|^{2}
\end{align}
and
\begin{align}\label{malgrad2}
\mathbb{E}|D_{U_{2}}\nabla_{v_{1}}\hat{X}_{s}^{x}|^{4}\leq C_{\theta}\big(1+\frac{d^{2}}{\delta^{4}t^{3}}\big)e^{-\frac{{\theta_{0}}}{2}s}|v_{1}|^{4}|v_{2}|^{4}.
\end{align}
\end{lemma}
\begin{proof}
We only give the proof of (\ref{malgrad1}) and the (\ref{malgrad2}) can be proved in the same way.

Writing $\zeta(s)=D_{U_{2}}\nabla_{v_{1}}\hat{X}_{s}^{x}$, by It\^{o}'s formula, (\ref{diss}), Cauchy-Schwarz inequality, (\ref{Adiss}) and (\ref{diffusion1}), we have
\begin{align*}
\frac{\dif}{\dif r}\mathbb{E}|\zeta(r)|^{2}=&2\mathbb{E}[\langle\nabla B(\hat{X}_{r}^{x})\zeta(r)+\frac{r}{t}\nabla^{2}B(\hat{X}_{r}^x)\nabla_{v_{2}}\hat{X}_{r}^{x}\nabla_{v_{1}}\hat{X}_{r}^{x},\zeta(r)\rangle]\\
&+2\mathbb{E}[\langle\frac{1}{t}\nabla\sigma(\hat{X}_{r}^{x})\nabla_{v_{1}}\hat{X}_{r}^{x}\sigma(\hat{X}_{r}^{x})^{-1}\nabla_{v_{2}}\hat{X}_{r}^{x},\zeta(r)\rangle]\\
&+\eta\mathbb{E}[\|\nabla\sigma(\hat{X}_{r}^{x})\zeta(r)+\frac{r}{t}\nabla^{2}\sigma(\hat{X}_{r}^{x})\nabla_{v_{2}}\hat{X}_{r}^{x}\nabla_{v_{1}}\hat{X}_{r}^{x}\|_{\rm{HS}}^{2}]\\
\leq&-2\theta_{0}\mathbb{E}|\zeta(r)|^{2}+2\theta_{1}\frac{r}{t}\mathbb{E}[|\zeta(r)||\nabla_{v_{1}}\hat{X}_{r}^{x}||\nabla_{v_{2}}\hat{X}_{r}^{x}|]\\
&+2\theta_{3}\frac{1}{t}\mathbb{E}[|\zeta(r)||\nabla_{v_{1}}\hat{X}_{r}^{x}||\nabla_{v_{2}}\hat{X}_{r}^{x}|\|\sigma(\hat{X}_{r}^{x})^{-1}\|_{\rm{HS}}]\\
&+2\eta\mathbb{E}\big[\theta_{3}^{2}|\zeta(r)|^{2}+\theta_{4}^{2}\frac{r^{2}}{t^{2}}|\nabla_{v_{1}}\hat{X}_{r}^{x}|^{2}|\nabla_{v_{2}}\hat{X}_{r}^{x}|^{2}\big].
\end{align*}
By Young's inequality, we further have
\begin{align*}
\frac{\dif}{\dif r}\mathbb{E}|\zeta(r)|^{2}\leq&-2\theta_{0}\mathbb{E}|\zeta(r)|^{2}+2\theta_{1}\mathbb{E}[\frac{3{\theta_{0}}}{32\theta_{1}}|\zeta(r)|^{2}+\frac{8\theta_{1}}{3\theta_{0}}\frac{r^{2}}{t^{2}}|\nabla_{v_{1}}\hat{X}_{r}^{x}|^{2}|\nabla_{v_{2}}\hat{X}_{r}^{x}|^{2}]\\
&+2\theta_{3}\mathbb{E}[\frac{3\theta_{0}}{32\theta_{3}}|\zeta(r)|^{2}+\frac{8}{3\theta_{0}}\frac{1}{t^{2}}|\nabla_{v_{1}}\hat{X}_{r}^{x}|^{2}|\nabla_{v_{2}}\hat{X}_{r}^{x}|^{2}\|\sigma(\hat{X}_{r}^{x})^{-1}\|_{\rm{HS}}^{2}]\\
&+2\eta\mathbb{E}\big[\theta_{3}^{2}|\zeta(r)|^{2}+\theta_{4}^{2}\frac{r^{2}}{t^{2}}|\nabla_{v_{1}}\hat{X}_{r}^{x}|^{2}|\nabla_{v_{2}}\hat{X}_{r}^{x}|^{2}\big].
\end{align*}
Noticing that $\eta\leq\frac{11{\theta_{0}}}{16\theta_{3}^{2}}$, by (\ref{definite}), Cauchy's inequality and (\ref{grad3}), we can get
\begin{align*}
\frac{\dif}{\dif r}\mathbb{E}|\zeta(r)|^{2}\leq&-(\frac{13}{8}{\theta_{0}}-2\theta_{3}^{2}\eta)\mathbb{E}|\zeta(r)|^{2}+\big(\frac{16\theta_{1}^{2}}{3{\theta_{0}}}+\frac{16\theta_{3}}{3{\theta_{0}}}\frac{d}{\delta^{2}}\frac{1}{t^{2}}
+2\theta_{4}^{2}\eta\big)e^{-\frac{\theta_{0}}{2} t}|v_{1}|^{2}|v_{2}|^{2}\\
\leq&-\frac{{\theta_{0}}}{4}\mathbb{E}|\zeta(r)|^{2}+C_{\theta}(1+\frac{d}{\delta^{2}t^{2}})e^{-\frac{\theta_{0}}{2} t}|v_{1}|^{2}|v_{2}|^{2}.
\end{align*}
This inequality, together with $\zeta(0)=0,$ implies
\begin{align*}
\mathbb{E}|\zeta(s)|^{2}
\leq&C_{\theta}(1+\frac{d}{\delta^{2}t})e^{-\frac{{\theta_{0}}}{4}s}|v_{1}|^{2}|v_{2}|^{2}.
\end{align*}
\end{proof}

Based on the results above, we have the following two lemmas:

\begin{lemma}
Let unit vectors $v_{1},v_{2}\in\mathbb{R}^{d}$ and $x\in\mathbb{R}^{d}.$ Then, for all $\eta\leq\min\{1,\frac{{\theta_{0}}}{48\theta_{3}^{2}}\},$ we have
\begin{align}\label{1}
\mathbb{E}|\mathcal{I}_{v_{1}}^{x}(t)|^{4}\leq\frac{10}{t^{2}}\frac{d^{2}}{\delta^{4}\eta^{2}}|v_{1}|^{4},
\end{align}
\begin{align}\label{2}
\mathbb{E}|\nabla_{v_{2}}\mathcal{I}_{v_{1}}^{x}(t)|^{2}\leq C_{\theta}\frac{1}{t}\big(1+\frac{d^{2}}{\delta^{4}}\big)\eta^{-1}|v_{1}|^{2}|v_{2}|^{2},
\end{align}
\begin{align}\label{3}
\mathbb{E}|D_{U_{2}}\mathcal{I}_{v_{1}}^{x}(t)|^{2}\leq C_{\theta}\frac{1}{t\eta}\big(1+\frac{d^{2}}{\delta^{4}}+\frac{d^{2}}{\delta^{4}}\frac{1}{t\eta}\big)|v_{1}|^{2}|v_{2}|^{2}.
\end{align}
\end{lemma}
\begin{proof}
By Burkholder's inequality \cite[Theorem 2]{Ren08}, (\ref{definite}) and (\ref{grad3}), we have
\begin{align*}
\mathbb{E}|\mathcal{I}_{v_{1}}^{x}(t)|^{4}=&\mathbb{E}|\frac{1}{t}\int_{0}^{t}\langle\eta^{-\frac{1}{2}}\sigma(\hat{X}_{s}^{x})^{-1}\nabla_{v_{1}}\hat{X}_{s}^{x},\dif B_{s}\rangle|^{4}\\
\leq&\frac{4\sqrt{2}}{t^{4}}\eta^{-2}\mathbb{E}\big(\int_{0}^{t}|\sigma(\hat{X}_{s}^{x})^{-1}\nabla_{v_{1}}\hat{X}_{s}^{x}|^{2}\dif s\big)^{2}\\
\leq&\frac{10}{t^{4}}\frac{d^{2}}{\delta^{4}\eta^{2}}\mathbb{E}\big(\int_{0}^{t}|\nabla_{v_{1}}\hat{X}_{s}^{x}|^{2}\dif s\big)^{2}\leq\frac{10}{t^{3}}\frac{d^{2}}{\delta^{4}\eta^{2}}\int_{0}^{t}\mathbb{E}|\nabla_{v_{1}}\hat{X}_{s}^{x}|^{4}\dif s
\leq\frac{10}{t^{2}}\frac{d^{2}}{\delta^{4}\eta^{2}}|v_{1}|^{4}.
\end{align*}

For \eqref{2}, the definition of $\mathcal{I}_{v_{1}}^{x}(t)$ yields
\begin{align*}
\nabla_{v_{2}}\mathcal{I}_{v_{1}}^{x}(t)=\frac{1}{t}\eta^{-\frac{1}{2}}\int_{0}^{t}\langle -\sigma(X_{s}^{x})^{-1}\nabla_{v_{2}}\sigma(X_{s}^{x})\sigma(X_{s}^{x})^{-1}\nabla_{v_{1}}X_{s}^{x}+\sigma(X_{s}^{x})^{-1}\nabla_{v_{2}}\nabla_{v_{1}}X_{s}^{x},\dif B_{s}\rangle.
\end{align*}
Then, by It\^{o} isometry, (\ref{diffusion1}) and (\ref{definite}), we have
\begin{align*}
&\mathbb{E}|\nabla_{v_{2}}\mathcal{I}_{v_{1}}^{x}(t)|^{2}\\
\leq&\frac{2}{t^{2}}\eta^{-1}\int_{0}^{t}\mathbb{E}[\theta_{3}^{2}\|\sigma(\hat{X}_{s}^{x})^{-1}\|_{{\rm HS}}^{4}|\nabla_{v_{1}}\hat{X}_{s}^{x}|^{2}|\nabla_{v_{1}}\hat{X}_{s}^{x}|^{2}
+\|\sigma(\hat{X}_{s}^{x})^{-1}\|_{{\rm HS}}^{2}|\nabla_{v_{2}}\nabla_{v_{1}}\hat{X}_{s}^{x}|^{2}]\dif s\\
\leq&\frac{2}{t^{2}}\eta^{-1}\int_{0}^{t}\mathbb{E}[\theta_{3}^{2}\frac{d^{2}}{\delta^{4}}|\nabla_{v_{1}}\hat{X}_{s}^{x}|^{2}|\nabla_{v_{1}}\hat{X}_{s}^{x}|^{2}
+\frac{d}{\delta^{2}}|\nabla_{v_{2}}\nabla_{v_{1}}\hat{X}_{s}^{x}|^{2}]\dif s.
\end{align*}
By Cauchy-Schwarz inequality, (\ref{grad3}) and (\ref{gradsec}), we further have
\begin{align*}
\mathbb{E}|\nabla_{v_{2}}\mathcal{I}_{v_{1}}^{x}(t)|^{2}\leq&
C_{\theta}\frac{1}{t^{2}}\eta^{-1}\int_{0}^{t}\big(1+\frac{d^{2}}{\delta^{4}}\big)\dif s|v_{1}|^{2}|v_{2}|^{2}\leq C_{\theta}\frac{1}{t}\big(1+\frac{d^{2}}{\delta^{4}}\big)\eta^{-1}|v_{1}|^{2}|v_{2}|^{2}.
\end{align*}

Recall (\ref{malliavin1}) and (\ref{ma1}), it is easy to see that $D_{U_{2}}\mathcal{I}_{v_{1}}^{x}(t)$ can be computed by (\ref{chain2}) as
\begin{align*}
&D_{U_{2}}\mathcal{I}_{v_{1}}^{x}(t)\\
=&\frac{\eta^{-1}}{t^{2}}\int_{0}^{t}\langle \sigma(\hat{X}_{s}^{x})^{-1}\nabla_{v_{1}}\hat{X}_{s}^{x},\sigma(\hat{X}_{s}^{x})^{-1}\nabla_{v_{2}}\hat{X}_{s}^{x}\rangle \dif s+\frac{\eta^{-\frac{1}{2}}}{t}\int_{0}^{t}\langle \sigma(\hat{X}_{s}^{x})^{-1}D_{U_{2}}\nabla_{v_{1}}\hat{X}_{s}^{x},\dif B_{s}\rangle\\
&-\frac{1}{t}\eta^{-\frac{1}{2}}\int_{0}^{t}\langle \sigma(\hat{X}_{s}^{x})^{-1}\nabla\sigma(\hat{X}_{s}^{x})\sigma(\hat{X}_{s}^{x})^{-1}\frac{s}{t}\nabla_{v_{2}}\hat{X}_{s}^{x}\nabla_{v_{1}}\hat{X}_{s}^{x},\dif B_{s}\rangle.
\end{align*}
Then, by the Cauchy-Schwarz inequality, It\^{o} isometry, (\ref{diffusion1}), (\ref{definite}), (\ref{grad3}) and (\ref{malgrad1}), we have
\begin{align*}
&\mathbb{E}|D_{U_{2}}\mathcal{I}_{v_{1}}^{x}(t)|^{2}\\
\leq&\frac{3}{t^{3}\eta^{2}}\int_{0}^{t}\mathbb{E}[\|\sigma(\hat{X}_{s}^{x})^{-1}\|_{{\rm HS}}^{4}|\nabla_{v_{1}}\hat{X}_{s}^{x}|^{2}|\nabla_{v_{2}}\hat{X}_{s}^{x}|^{2}] \dif s+\frac{3}{t^{2}\eta}\int_{0}^{t}\mathbb{E}[\|\sigma(\hat{X}_{s}^{x})^{-1}\|_{{\rm HS}}^{2}|D_{U_{2}}\nabla_{v_{1}}\hat{X}_{s}^{x}|^{2}]\dif s\\
&+\theta_{3}^{2}\frac{3}{t^{2}\eta}\int_{0}^{t}\mathbb{E}\big[\|\sigma(\hat{X}_{s}^{x})^{-1}\|_{{\rm HS}}^{4}|\nabla_{v_{2}}\hat{X}_{s}^{x}|^{2}|\nabla_{v_{1}}\hat{X}_{s}^{x}|^{2}\big]\dif s\\ \leq&\frac{3}{t^{2}\eta}(\frac{1}{t\eta}+\theta_{3}^{2})\int_{0}^{t}\frac{d^{2}}{\delta^{4}}\mathbb{E}[|\nabla_{v_{1}}\hat{X}_{s}^{x}|^{2}|\nabla_{v_{2}}\hat{X}_{s}^{x}|^{2}] \dif s+\frac{3}{t^{2}\eta}\int_{0}^{t}\frac{d}{\delta^{2}}\mathbb{E}|D_{U_{2}}\nabla_{v_{1}}\hat{X}_{s}^{x}|^{2}\dif s\\
\leq&C_{\theta}\frac{1}{t\eta}\big(1+\frac{d^{2}}{\delta^{4}}+\frac{d^{2}}{\delta^{4}}\frac{1}{t\eta}\big)|v_{1}|^{2}|v_{2}|^{2}.
\end{align*}
\end{proof}

Furthermore, let $v_{1},v_{2},v_{3}\in\mathbb{R}^{d},$ and define $u_{i}$ and $U_{i}$ as (\ref{malliavin1}) and (\ref{bismutf}), respectively, for $i=1,2,3.$ From (\ref{malsec}), we can similarly define $\nabla_{v_{3}}D_{U_{2}}\nabla_{v_{1}}X_{s}^{x},$ which satisfies the following equation: for $s\in[0,t],$
\begin{align*}
&\dif \nabla_{v_{3}}D_{U_{2}}\nabla_{v_{1}}\hat{X}_{s}^{x}\\
=&\big[\nabla^{2} B(\hat{X}_{s}^{x})\nabla_{v_{3}}\hat{X}_{s}^{x}D_{U_{2}}\nabla_{v_{1}}\hat{X}_{s}^{x}+\nabla B(\hat{X}_{s}^{x})\nabla_{v_{3}}D_{U_{2}}\nabla_{v_{1}}\hat{X}_{s}^{x}\\
&+\frac{s}{t}\nabla^{3}B(\hat{X}_{s}^{x})\nabla_{v_{3}}\hat{X}_{s}^{x}\nabla_{v_{2}}\hat{X}_{s}^{x}\nabla_{v_{1}}\hat{X}_{s}^{x}
+\frac{s}{t}\nabla^{2}B(\hat{X}_{s}^{x})\nabla_{v_{3}}\nabla_{v_{2}}\hat{X}_{s}^{x}\nabla_{v_{1}}\hat{X}_{s}^{x}\\
&+\frac{s}{t}\nabla^{2}B(\hat{X}_{s}^{x})\nabla_{v_{2}}\hat{X}_{s}^{x}\nabla_{v_{3}}\nabla_{v_{1}}\hat{X}_{s}^{x}
+\frac{1}{t}\nabla^{2}\sigma(\hat{X}_{s}^{x})\nabla_{v_{3}}\hat{X}_{s}^{x}\nabla_{v_{1}}\hat{X}_{s}^{x}\sigma(\hat{X}_{s}^{x})^{-1}\nabla_{v_{2}}\hat{X}_{s}^{x}\\
&+\frac{1}{t}\nabla\sigma(\hat{X}_{s}^{x})\big(\nabla_{v_{3}}\nabla_{v_{1}}\hat{X}_{s}^{x}\sigma(\hat{X}_{s}^{x})^{-1}\nabla_{v_{2}}\hat{X}_{s}^{x}
+\nabla_{v_{1}}\hat{X}_{s}^{x}\sigma(\hat{X}_{s}^{x})^{-1}\nabla_{v_{3}}\nabla_{v_{2}}\hat{X}_{s}^{x}\big)\\
&-\frac{1}{t}\nabla\sigma(\hat{X}_{s}^{x})\nabla_{v_{1}}\hat{X}_{s}^{x}\sigma(\hat{X}_{s}^{x})^{-1}\nabla\sigma(\hat{X}_{s}^{x})\nabla_{v_{3}}\hat{X}_{s}^{x}
\sigma(\hat{X}_{s}^{x})^{-1}\nabla_{v_{2}}\hat{X}_{s}^{x}\big]\dif s\\
&+\eta^{\frac{1}{2}}\big[\nabla^{2}\sigma(\hat{X}_{s}^{x})\nabla_{v_{3}}\hat{X}_{s}^{x}D_{U_{2}}\nabla_{v_{1}}\hat{X}_{s}^{x}
+\nabla\sigma(\hat{X}_{s}^{x})\nabla_{v_{3}}D_{U_{2}}\nabla_{v_{1}}\hat{X}_{s}^{x}\\
&\qquad\quad+\frac{s}{t}\nabla^{3}\sigma(\hat{X}_{s}^{x})\nabla_{v_{3}}\hat{X}_{s}^{x}\nabla_{v_{2}}\hat{X}_{s}^{x}\nabla_{v_{1}}\hat{X}_{s}^{x}
+\frac{s}{t}\nabla^{2}\sigma(\hat{X}_{s}^{x})\nabla_{v_{3}}\nabla_{v_{2}}\hat{X}_{s}^{x}\nabla_{v_{1}}\hat{X}_{s}^{x}\\
&\qquad\quad+\frac{s}{t}\nabla^{2}\sigma(\hat{X}_{s}^{x})\nabla_{v_{2}}\hat{X}_{s}^{x}\nabla_{v_{3}}\nabla_{v_{1}}\hat{X}_{s}^{x}\big]\dif B_{s}
\end{align*}
with $\nabla_{v_{3}}D_{U_{2}}\nabla_{v_{1}}\hat{X}_{0}^{x}=0$. Moreover, from (\ref{secequation}) and (\ref{malsec}), we can similarly define $D_{U_{3}}\nabla_{v_{2}}\nabla_{v_{1}}\hat{X}_{s}^{x}$ and $D_{U_{3}}D_{U_{2}}\nabla_{v_{1}}\hat{X}_{s}^{x}$, respectively.

Then, we have the following upper bounds on Malliavin derivatives.

\begin{lemma}
Let $v_{i}\in\mathbb{R}^{d}$ for $i=1,2,3,$ and let
\begin{align*}
U_{i,s}=\int_{0}^{s}u_{i}(r)\dif r, \quad 0\leq s\leq t,
\end{align*}
where $u_{i}(r)=\frac{1}{t}\eta^{-\frac{1}{2}}\sigma(\hat{X}_{r}^{x})^{-1}\nabla_{v_{i}}\hat{X}_{r}^{x}$ for $0\leq r\leq t.$ Then, for all $\eta\leq\min\{1,\frac{17{\theta_{0}}}{60\theta_{3}^{2}}\}$, we have
\begin{align}\label{malthird}
\mathbb{E}|\nabla_{v_{3}}D_{U_{2}}\nabla_{v_{1}}\hat{X}_{s}^{x}|^{2}\leq C_{\theta}\big(1+\frac{d^{2}}{\delta^{4}}\big)\big(1+\frac{1}{t}\big)e^{-\frac{\theta_{0}}{4}s}|v_{1}|^{2}|v_{2}|^{2}|v_{3}|^{2},
\end{align}
\begin{align}\label{malthird2}
\mathbb{E}|D_{U_{3}}\nabla_{v_{2}}\nabla_{v_{1}}\hat{X}_{s}^{x}|^{2}\leq C_{\theta}(1+\frac{d}{\delta^{2}t})e^{-\frac{\theta_{0}}{4}s}|v_{1}|^{2}|v_{2}|^{2}|v_{3}|^{2},
\end{align}
\begin{align}\label{malthird3}
\mathbb{E}|D_{U_{3}}D_{U_{2}}\nabla_{v_{1}}\hat{X}_{s}^{x}|^{2}\leq C_{\theta}\big(1+\frac{d^{2}}{\delta^{4}}\big)\big(1+\frac{1}{t^{\frac{5}{2}}}\big)e^{-\frac{\theta_{0}}{4}s}|v_{1}|^{2}|v_{2}|^{2}|v_{3}|^{2}.
\end{align}
\end{lemma}
\begin{proof}
Writing $\tau_{1}(s)=\nabla_{v_{3}}D_{U_{2}}\nabla_{v_{1}}\hat{X}_{s}^{x},$ by It\^{o}'s formula, we have
\begin{align*}
&\frac{\dif}{\dif r}\mathbb{E}|\tau_{1}(r)|^{2}\\
=&2\mathbb{E}\langle\nabla^{2} B(\hat{X}_{r}^{x})\nabla_{v_{3}}\hat{X}_{r}^{x}D_{U_{2}}\nabla_{v_{1}}\hat{X}_{r}^{x}+\nabla B(\hat{X}_{r}^{x})\tau_{1}(r)\\
&\qquad\quad+\frac{r}{t}\nabla^{3}B(\hat{X}_{r})\nabla_{v_{3}}\hat{X}_{r}^{x}\nabla_{v_{2}}\hat{X}_{r}^{x}\nabla_{v_{1}}\hat{X}_{r}^{x}
+\frac{r}{t}\nabla^{2}B(\hat{X}_{r})\nabla_{v_{3}}\nabla_{v_{2}}\hat{X}_{r}^{x}\nabla_{v_{1}}\hat{X}_{r}^{x}\\
&\qquad\quad+\frac{r}{t}\nabla^{2}B(\hat{X}_{r})\nabla_{v_{2}}\hat{X}_{r}^{x}\nabla_{v_{3}}\nabla_{v_{1}}\hat{X}_{r}^{x}
+\frac{1}{t}\nabla^{2}\sigma(\hat{X}_{r}^{x})\nabla_{v_{3}}\hat{X}_{r}^{x}\nabla_{v_{1}}\hat{X}_{r}^{x}\sigma(\hat{X}_{r}^{x})^{-1}\nabla_{v_{2}}\hat{X}_{r}^{x}\\
&\qquad\quad+\frac{1}{t}\nabla\sigma(\hat{X}_{r}^{x})\big(\nabla_{v_{3}}\nabla_{v_{1}}\hat{X}_{r}^{x}\sigma(\hat{X}_{r}^{x})^{-1}\nabla_{v_{2}}\hat{X}_{r}^{x}
+\nabla_{v_{1}}\hat{X}_{r}^{x}\sigma(\hat{X}_{r}^{x})^{-1}\nabla_{v_{3}}\nabla_{v_{2}}\hat{X}_{r}^{x}\big)\\
&\qquad\quad-\frac{1}{t}\nabla\sigma(\hat{X}_{r}^{x})\nabla_{v_{1}}\hat{X}_{r}^{x}\sigma(\hat{X}_{r}^{x})^{-1}\nabla\sigma(\hat{X}_{r}^{x})\nabla_{v_{3}}\hat{X}_{r}^{x}
\sigma(\hat{X}_{r}^{x})^{-1}\nabla_{v_{2}}\hat{X}_{r}^{x},\tau_{1}(r)\rangle\\
&+\eta\mathbb{E}\|\nabla^{2}\sigma(\hat{X}_{r}^{x})\nabla_{v_{3}}\hat{X}_{r}^{x}D_{U_{2}}\nabla_{v_{1}}\hat{X}_{r}^{x}
+\nabla\sigma(\hat{X}_{r}^{x})\tau_{1}(r)+\frac{r}{t}\nabla^{2}\sigma(\hat{X}_{r}^{x})\nabla_{v_{2}}\hat{X}_{r}^{x}\nabla_{v_{3}}\nabla_{v_{1}}\hat{X}_{r}^{x}\\
&\qquad\qquad+\frac{r}{t}\nabla^{3}\sigma(\hat{X}_{r}^{x})\nabla_{v_{3}}\hat{X}_{r}^{x}\nabla_{v_{2}}\hat{X}_{r}^{x}\nabla_{v_{1}}\hat{X}_{r}^{x}
+\frac{r}{t}\nabla^{2}\sigma(\hat{X}_{r}^{x})\nabla_{v_{3}}\nabla_{v_{2}}\hat{X}_{r}^{x}\nabla_{v_{1}}\hat{X}_{r}^{x}\|_{{\rm HS}}^{2}.
\end{align*}
It follows from {\bf Assumption A1} and the Cauchy-Schwarz inequality that
\begin{align*}
&\frac{\dif}{\dif r}\mathbb{E}|\tau_{1}(r)|^{2}\\
\leq&-2{\theta_{0}}\mathbb{E}|\tau_{1}(r)|^{2}+2\theta_{1}\mathbb{E}\big[|\nabla_{v_{3}}\hat{X}_{r}^{x}||D_{U_{2}}\nabla_{v_{1}}\hat{X}_{r}^{x}||\tau_{1}(r)|\big]\\
&+2\mathbb{E}\big[\theta_{2}|\nabla_{v_{3}}\hat{X}_{r}^{x}||\nabla_{v_{2}}\hat{X}_{r}^{x}||\nabla_{v_{1}}\hat{X}_{r}^{x}||\tau_{1}(r)|
+\theta_{1}|\nabla_{v_{3}}\nabla_{v_{2}}\hat{X}_{r}^{x}||\nabla_{v_{1}}\hat{X}_{r}^{x}||\tau_{1}(r)|\big]\\
&+2\mathbb{E}\big[\big(\theta_{1}|\nabla_{v_{3}}\nabla_{v_{1}}\hat{X}_{r}^{x}|
+\frac{\theta_{4}}{t}|\nabla_{v_{3}}\hat{X}_{r}^{x}||\nabla_{v_{1}}\hat{X}_{r}^{x}|\|\sigma(\hat{X}_{r}^{x})^{-1}\|_{{\rm HS}}\big)|\nabla_{v_{2}}\hat{X}_{r}^{x}||\tau_{1}(r)|\big]\\
&+2\theta_{3}\frac{1}{t}\mathbb{E}\big[\big(|\nabla_{v_{3}}\nabla_{v_{1}}\hat{X}_{r}^{x}||\nabla_{v_{2}}\hat{X}_{r}^{x}|
+|\nabla_{v_{1}}\hat{X}_{r}^{x}||\nabla_{v_{3}}\nabla_{v_{2}}\hat{X}_{r}^{x}|\big)\|\sigma(\hat{X}_{r}^{x})^{-1}\|_{{\rm HS}}|\tau_{1}(r)|\big]\\
&+2\theta_{3}^{2}\frac{1}{t}\mathbb{E}\big[|\nabla_{v_{1}}\hat{X}_{r}^{x}|\|\sigma(\hat{X}_{r}^{x})^{-1}\|_{{\rm HS}}^{2}|\nabla_{v_{3}}\hat{X}_{r}^{x}||\nabla_{v_{2}}\hat{X}_{r}^{x}||\tau_{1}(r)|\big]\\
&+5\theta_{4}^{2}\eta\mathbb{E}\big[|\nabla_{v_{3}}\hat{X}_{r}^{x}|^{2}|D_{U_{2}}\nabla_{v_{1}}\hat{X}_{r}^{x}|^{2}
+|\nabla_{v_{2}}\hat{X}_{r}^{x}|^{2}|\nabla_{v_{3}}\nabla_{v_{1}}\hat{X}_{r}^{x}|^{2}+|\nabla_{v_{3}}\nabla_{v_{2}}\hat{X}_{r}^{x}|^{2}|\nabla_{v_{1}}\hat{X}_{r}^{x}|^{2}\big]\\
&+5\eta\mathbb{E}\big[\theta_{3}^{2}|\tau_{1}(r)|^{2}+\theta_{5}^{2}|\nabla_{v_{3}}\hat{X}_{r}^{x}|^{2}|\nabla_{v_{2}}\hat{X}_{r}^{x}|^{2}|\nabla_{v_{1}}\hat{X}_{r}^{x}|^{2}\big].
\end{align*}
Moreover, by (\ref{definite}) and Young's inequality, we have
\begin{align*}
&\frac{\dif}{\dif r}\mathbb{E}|\tau_{1}(r)|^{2}\\
\leq&-(2{\theta_{0}}-5\theta_{3}^{2}\eta)\mathbb{E}|\tau_{1}(r)|^{2}+\mathbb{E}\big[\frac{{\theta_{0}}}{24}|\tau_{1}(r)|^{2}
+\frac{24\theta_{1}^{2}}{{\theta_{0}}}|\nabla_{v_{3}}\hat{X}_{r}^{x}|^{2}|D_{U_{2}}\nabla_{v_{1}}\hat{X}_{r}^{x}|^{2}\big]\\
&+\mathbb{E}\big[\frac{{\theta_{0}}}{12}|\tau_{1}(r)|^{2}+\big(\frac{24\theta_{2}^{2}}{{\theta_{0}}}|\nabla_{v_{3}}\hat{X}_{r}^{x}|^{2}|\nabla_{v_{2}}\hat{X}_{r}^{x}|^{2}
+\frac{24\theta_{1}^{2}}{{\theta_{0}}}|\nabla_{v_{3}}\nabla_{v_{2}}\hat{X}_{r}^{x}|^{2}\big)|\nabla_{v_{1}}\hat{X}_{r}^{x}|^{2}\big]\\
&+\mathbb{E}\big[\frac{{\theta_{0}}}{12}|\tau_{1}(r)|^{2}+\big(\frac{24\theta_{1}^{2}}{{\theta_{0}}}|\nabla_{v_{3}}\nabla_{v_{1}}\hat{X}_{r}^{x}|^{2}
+\frac{24\theta_{4}^{2}}{{\theta_{0}}t^{2}}|\nabla_{v_{3}}\hat{X}_{r}^{x}|^{2}|\nabla_{v_{1}}\hat{X}_{r}^{x}|^{2}\frac{d}{\delta^{2}}\big)|\nabla_{v_{2}}\hat{X}_{r}^{x}|^{2}\big]\\
&+\mathbb{E}\big[\frac{{\theta_{0}}}{12}|\tau_{1}(r)|^{2}+\frac{24\theta_{3}^{2}}{{\theta_{0}}t^{2}}\big(|\nabla_{v_{3}}\nabla_{v_{1}}\hat{X}_{r}^{x}|^{2}|\nabla_{v_{2}}\hat{X}_{r}^{x}|^{2}
+|\nabla_{v_{1}}\hat{X}_{r}^{x}|^{2}|\nabla_{v_{3}}\nabla_{v_{2}}\hat{X}_{r}^{x}|^{2}\big)\frac{d}{\delta^{2}}\big]\\
&+\mathbb{E}\big[\frac{{\theta_{0}}}{24}|\tau_{1}(r)|^{2}+\frac{24\theta_{3}^{4}}{{\theta_{0}}t^{2}}|\nabla_{v_{1}}\hat{X}_{r}^{x}|^{2}\frac{d^{2}}{\delta^{4}}|\nabla_{v_{3}}\hat{X}_{r}^{x}|^{2}|\nabla_{v_{2}}\hat{X}_{r}^{x}|^{2}\big]\\
&+5\theta_{4}^{2}\eta\mathbb{E}\big[|\nabla_{v_{3}}\hat{X}_{r}^{x}|^{2}|D_{U_{2}}\nabla_{v_{1}}\hat{X}_{r}^{x}|^{2}
+|\nabla_{v_{2}}\hat{X}_{r}^{x}|^{2}|\nabla_{v_{3}}\nabla_{v_{1}}\hat{X}_{r}^{x}|^{2}+|\nabla_{v_{3}}\nabla_{v_{2}}\hat{X}_{r}^{x}|^{2}|\nabla_{v_{1}}\hat{X}_{r}^{x}|^{2}\big]\\
&+5\theta_{5}^{2}\eta\mathbb{E}\big[|\nabla_{v_{3}}\hat{X}_{r}^{x}|^{2}|\nabla_{v_{2}}\hat{X}_{r}^{x}|^{2}|\nabla_{v_{1}}\hat{X}_{r}^{x}|^{2}\big],
\end{align*}
By the fact $\eta\leq\min\{1,\frac{17{\theta_{0}}}{60\theta_{3}^{2}}\}$, the Cauchy-Schwarz inequality, (\ref{grad3}), (\ref{malgrad2}) and (\ref{gradsec}), we further have
\begin{align*}
\frac{\dif}{\dif r}\mathbb{E}|\tau_{1}(r)|^{2}\leq&-\frac{{\theta_{0}}}{4}\mathbb{E}|\tau_{1}(r)|^{2}+C_{\theta}\big(1+\frac{d^{2}}{\delta^{4}}\big)\big(1+\frac{1}{t^{2}}\big)e^{-\frac{\theta_{0}}{2}r}|v_{1}|^{2}|v_{2}|^{2}|v_{3}|^{2}.
\end{align*}
This inequality, together with $\tau_{1}(0)=0,$ implies
\begin{align*}
\mathbb{E}|\tau_{1}(s)|^{2}\leq&C_{\theta}\big(1+\frac{d^{2}}{\delta^{4}}\big)\big(1+\frac{1}{t^{2}}\big)|v_{1}|^{2}|v_{2}|^{2}|v_{3}|^{2}\int_{0}^{s}e^{-\frac{\theta_{0}}{2}r}e^{-\frac{\theta_{0}(s-r)}{4}}\dif r\\
\leq&C_{\theta}\big(1+\frac{d^{2}}{\delta^{4}}\big)\big(1+\frac{1}{t}\big)e^{-\frac{\theta_{0}}{4}s}|v_{1}|^{2}|v_{2}|^{2}|v_{3}|^{2}.
\end{align*}
\end{proof}

With the above results, we have the following estimates:

\begin{lemma}
Let $v_{1},v_{2},v_{3}\in\mathbb{R}^{d}$ and $x\in\mathbb{R}^{d}.$ Then, as $\eta\leq\min\{1,\frac{{\theta_{0}}}{48\theta_{3}^{2}}\}$, we have
\begin{align}\label{third1}
\mathbb{E}|\mathcal{R}_{v_{1},v_{2}}^{x}(t)|^{2}\leq C_{\theta}(1+\frac{d}{\delta^{2}t})e^{-\frac{{\theta_{0}}}{4}t}|v_{1}|^{2}|v_{2}|^{2},
\end{align}
\begin{align}\label{third2}
\mathbb{E}|\nabla_{v_{3}}\mathcal{R}_{v_{1},v_{2}}^{x}(t)|^{2}
\leq C_{\theta}\big(1+\frac{d^{2}}{\delta^{4}}\big)\big(1+\frac{1}{t}\big)e^{-\frac{\theta_{0}}{4}t}|v_{1}|^{2}|v_{2}|^{2}|v_{3}|^{2}
\end{align}
and
\begin{align}\label{third3}
\mathbb{E}|D_{U_{3}}\mathcal{R}_{v_{1},v_{2}}^{x}(t)|^{2}
\leq C_{\theta}\big(1+\frac{d^{2}}{\delta^{4}}\big)\big(1+\frac{1}{t^{\frac{5}{2}}}\big)e^{-\frac{\theta_{0}}{4}t}|v_{1}|^{2}|v_{2}|^{2}|v_{3}|^{2}.
\end{align}
\end{lemma}
\begin{proof}
By the Cauchy-Schwarz inequality, (\ref{gradsec}) and (\ref{malgrad1}), we have
\begin{align*}
\mathbb{E}|\mathcal{R}_{v_{1},v_{2}}^{x}(t)|^{2}\leq&2\mathbb{E}|\nabla_{v_{2}}\nabla_{v_{1}}\hat{X}_{t}^{x}|^{2}+2\mathbb{E}|D_{U_{2}}\nabla_{v_{1}}\hat{X}_{t}^{x}|^{2}
\leq C_{\theta}(1+\frac{d}{\delta^{2}t})e^{-\frac{{\theta_{0}}}{4}t}|v_{1}|^{2}|v_{2}|^{2}.
\end{align*}

By (\ref{gradthird}) and (\ref{malthird}), we have
\begin{align*}
\mathbb{E}|\nabla_{v_{3}}\mathcal{R}_{v_{1},v_{2}}^{x}(t)|^{2}\leq&2\mathbb{E}|\nabla_{v_{3}}\nabla_{v_{2}}\nabla_{v_{1}}\hat{X}_{t}^{x}|^{2}+2\mathbb{E}|\nabla_{v_{3}}D_{U_{2}}\nabla_{v_{1}}\hat{X}_{t}^{x}|^{2}\\
\leq&C_{\theta}\big(1+\frac{d^{2}}{\delta^{4}}\big)\big(1+\frac{1}{t}\big)e^{-\frac{\theta_{0}}{4}t}|v_{1}|^{2}|v_{2}|^{2}|v_{3}|^{2}.
\end{align*}

By (\ref{malthird2}) and (\ref{malthird3}), we have
\begin{align*}
\mathbb{E}|D_{U_{3}}\mathcal{R}_{v_{1},v_{2}}^{x}(t)|^{2}\leq&2\mathbb{E}|D_{U_{3}}\nabla_{v_{2}}\nabla_{v_{1}}\hat{X}_{t}^{x}|^{2}+2\mathbb{E}|D_{U_{3}}D_{U_{2}}\nabla_{v_{1}}\hat{X}_{t}^{x}|^{2}\\
\leq&C_{\theta}\big(1+\frac{d^{2}}{\delta^{4}}\big)\big(1+\frac{1}{t^{\frac{5}{2}}}\big)e^{-\frac{\theta_{0}}{4}t}|v_{1}|^{2}|v_{2}|^{2}|v_{3}|^{2}.
\end{align*}
\end{proof}

\subsection{Proof of Lemma \ref{mainlem1}}

Recall $P_{t}h(x)=\mathbb{E}[h(\hat{X}_{t}^{x})]$ for $h\in{\rm Lip}(1),$ by Lebesgue's dominated convergence theorem, the Cauchy-Schwarz inequality and (\ref{grad3}), we have
\begin{align*}
|\nabla_{v}\mathbb{E}[h(\hat{X}_{t}^{x})]|=|\mathbb{E}[\nabla h(\hat{X}_{t}^{x})\nabla_{v}\hat{X}_{t}^{x}]|
\leq\|\nabla h\|\mathbb{E}|\nabla_{v}\hat{X}_{t}^{x}|\leq e^{-\frac{\theta_{0}}{8} t}|v|,
\end{align*}
(\ref{gradient2}) is proved.

Denote
\begin{align}\label{ma2}
h_{\epsilon}(x)=\int_{\mathbb{R}^{d}}f_{\epsilon}(y)h(x-y)\dif y,
\end{align}
with $\epsilon>0$ and $f_{\epsilon}$ is the density of the normal distribution $N(0,\epsilon^{2}I_{d}).$ It is easy to see that $h_{\epsilon}$ is smooth, $\lim_{\epsilon\rightarrow0}h_{\epsilon}(x)=h(x),$ $\lim_{\epsilon\rightarrow0}\nabla h_{\epsilon}(x)=\nabla h(x)$ and $|h_{\epsilon}(x)|\leq C(1+|x|)$ for all $x\in\mathbb{R}^{d}$ and some $C>0.$ Moreover, $\|\nabla h_{\epsilon}\|\leq\|\nabla h\|\leq1.$ Then, by Lebesgue's dominated convergence theorem, we have
\begin{align*}
\nabla_{v_{2}}\nabla_{v_{1}}\mathbb{E}\big[h_{\epsilon}(\hat{X}_{t}^{x})\big]=\mathbb{E}\big[\nabla^{2}h_{\epsilon}(\hat{X}_{t}^{x})\nabla_{v_{2}}\hat{X}_{t}^{x}
\nabla_{v_{1}}\hat{X}_{t}^{x}\big]+\mathbb{E}\big[\nabla h_{\epsilon}(\hat{X}_{t}^{x})\nabla_{v_{2}}\nabla_{v_{1}}\hat{X}_{t}^{x}\big],
\end{align*}
by (\ref{malliavin1}) and (\ref{malliavin2}), we further have
\begin{align*}
\mathbb{E}\big[\nabla^{2}h_{\epsilon}(\hat{X}_{t}^{x})\nabla_{v_{2}}\hat{X}_{t}^{x}
\nabla_{v_{1}}\hat{X}_{t}^{x}\big]=&\mathbb{E}\big[\nabla^{2}h_{\epsilon}(\hat{X}_{t}^{x})D_{U_{2}}\hat{X}_{t}^{x}
\nabla_{v_{1}}\hat{X}_{t}^{x}\big]\\
=&\mathbb{E}\big[D_{U_{2}}\big(\nabla h_{\epsilon}(\hat{X}_{t}^{x})\big)\nabla_{v_{1}}\hat{X}_{t}^{x}\big]\\
=&\mathbb{E}\big[D_{U_{2}}\big(\nabla h_{\epsilon}(\hat{X}_{t}^{x})\nabla_{v_{1}}\hat{X}_{t}^{x}\big)\big]-\mathbb{E}\big[\nabla h_{\epsilon}(\hat{X}_{t}^{x})D_{U_{2}}\nabla_{v_{1}}\hat{X}_{t}^{x}\big]\\
=&\mathbb{E}\big[\nabla h_{\epsilon}(\hat{X}_{t}^{x})\nabla_{v_{1}}\hat{X}_{t}^{x}\mathcal{I}_{v_{2}}^{x}(t)\big]-\mathbb{E}\big[\nabla h_{\epsilon}(\hat{X}_{t}^{x})D_{U_{2}}\nabla_{v_{1}}\hat{X}_{t}^{x}\big],
\end{align*}
where the last equality is by Bismut's formula (\ref{bismut}). These imply
\begin{align}\label{seccrucial}
\nabla_{v_{2}}\nabla_{v_{1}}\mathbb{E}\big[h_{\epsilon}(\hat{X}_{t}^{x})\big]=\mathbb{E}\big[\nabla h_{\epsilon}(\hat{X}_{t}^{x})\nabla_{v_{1}}\hat{X}_{t}^{x}\mathcal{I}_{v_{2}}^{x}(t)\big]+\mathbb{E}\big[\nabla h_{\epsilon}(\hat{X}_{t}^{x})\mathcal{R}_{v_{1},v_{2}}^{x}(t)\big].
\end{align}
Therefore, by Lebesgue's dominated convergence theorem, the Cauchy-Schwarz inequality, (\ref{grad3}), (\ref{1}) and (\ref{third1}), we have
\begin{align*}
|\nabla_{v_{2}}\nabla_{v_{1}}\mathbb{E}[h(\hat{X}_{t}^{x})]|=|\lim_{\epsilon\rightarrow0}\nabla_{v_{2}}\nabla_{v_{1}}\mathbb{E}[h_{\epsilon}(\hat{X}_{t}^{x})]|\leq&\mathbb{E}|\nabla_{v_{1}}\hat{X}_{t}^{x}\mathcal{I}_{v_{2}}^{x}(t)|+\mathbb{E}|\mathcal{R}_{v_{1},v_{2}}^{x}(t)|\\
\leq&C_{\theta}(1+\frac{\sqrt{d}}{\delta})(1+\frac{1}{\sqrt{\eta t}})e^{-\frac{{\theta_{0}}}{8}t}|v_{1}||v_{2}|,
\end{align*}
(\ref{sec}) is proved.

By (\ref{seccrucial}) and Lebesgue's dominated convergence theorem, we have
\begin{align*}
\nabla_{v_{3}}\nabla_{v_{2}}\nabla_{v_{1}}\mathbb{E}\big[h_{\epsilon}(\hat{X}_{t}^{x})\big]=&\nabla_{v_{3}}\mathbb{E}\big[\nabla h_{\epsilon}(\hat{X}_{t}^{x})\nabla_{v_{1}}\hat{X}_{t}^{x}\mathcal{I}_{v_{2}}^{x}(t)\big]+\nabla_{v_{3}}\mathbb{E}\big[{\nabla h_{\epsilon}(\hat{X}_{t}^{x})}\mathcal{R}_{v_{1},v_{2}}^{x}(t)\big]\\
=&\mathbb{E}\big[\nabla^{2} h_{\epsilon}(\hat{X}_{t}^{x})\nabla_{v_{3}}\hat{X}_{t}^{x}\nabla_{v_{1}}\hat{X}_{t}^{x}\mathcal{I}_{v_{2}}^{x}(t)\big]+\mathbb{E}\big[\nabla h_{\epsilon}(\hat{X}_{t}^{x})\nabla_{v_{3}}\nabla_{v_{1}}\hat{X}_{t}^{x}\mathcal{I}_{v_{2}}^{x}(t)\big]\\
&+\mathbb{E}\big[\nabla h_{\epsilon}(\hat{X}_{t}^{x})\nabla_{v_{1}}\hat{X}_{t}^{x}\nabla_{v_{3}}\mathcal{I}_{v_{2}}^{x}(t)\big]+\mathbb{E}\big[\nabla^{2} h_{\epsilon}(\hat{X}_{t}^{x})\nabla_{v_{3}}\hat{X}_{t}^{x}\mathcal{R}_{v_{1},v_{2}}^{x}(t)\big]\\
&+\mathbb{E}\big[\nabla h_{\epsilon}(\hat{X}_{t}^{x})\nabla_{v_{3}}\mathcal{R}_{v_{1},v_{2}}^{x}(t)\big],
\end{align*}
by (\ref{malliavin1}), (\ref{malliavin2}) and (\ref{bismut}), we further have
\begin{align*}
&\mathbb{E}\big[\nabla^{2} h_{\epsilon}(\hat{X}_{t}^{x})\nabla_{v_{3}}\hat{X}_{t}^{x}\nabla_{v_{1}}\hat{X}_{t}^{x}\mathcal{I}_{v_{2}}^{x}(t)\big]=\mathbb{E}\big[\nabla^{2} h_{\epsilon}(\hat{X}_{t}^{x})D_{U_{3}}\hat{X}_{t}^{x}\nabla_{v_{1}}\hat{X}_{t}^{x}\mathcal{I}_{v_{2}}^{x}(t)\big]\\
=&\mathbb{E}\big[D_{U_{3}}\big(\nabla h_{\epsilon}(\hat{X}_{t}^{x})\big)\nabla_{v_{1}}\hat{X}_{t}^{x}\mathcal{I}_{v_{2}}^{x}(t)\big]\\
=&\mathbb{E}\big[D_{U_{3}}\big(\nabla h_{\epsilon}(\hat{X}_{t}^{x})\nabla_{v_{1}}\hat{X}_{t}^{x}\mathcal{I}_{v_{2}}^{x}(t)\big)\big]-\mathbb{E}\big[\nabla h_{\epsilon}(\hat{X}_{t}^{x})D_{U_{3}}\nabla_{v_{1}}\hat{X}_{t}^{x}\mathcal{I}_{v_{2}}^{x}(t)\big]\\
&-\mathbb{E}\big[\nabla h_{\epsilon}(\hat{X}_{t}^{x})\nabla_{v_{1}}\hat{X}_{t}^{x}D_{U_{3}}\mathcal{I}_{v_{2}}^{x}(t)\big]\\
=&\mathbb{E}\big[\nabla h_{\epsilon}(\hat{X}_{t}^{x})\nabla_{v_{1}}\hat{X}_{t}^{x}\mathcal{I}_{v_{2}}^{x}(t)\mathcal{I}_{v_{3}}^{x}(t)\big]-\mathbb{E}\big[\nabla h_{\epsilon}(\hat{X}_{t}^{x})D_{U_{3}}\nabla_{v_{1}}\hat{X}_{t}^{x}\mathcal{I}_{v_{2}}^{x}(t)\big]\\
&-\mathbb{E}\big[\nabla h_{\epsilon}(\hat{X}_{t}^{x})\nabla_{v_{1}}\hat{X}_{t}^{x}D_{U_{3}}\mathcal{I}_{v_{2}}^{x}(t)\big],
\end{align*}
and similarly
\begin{align*}
\mathbb{E}\big[\nabla^{2} h_{\epsilon}(\hat{X}_{t}^{x})\nabla_{v_{3}}\hat{X}_{t}^{x}\mathcal{R}_{v_{1},v_{2}}^{x}(t)\big]
=&\mathbb{E}\big[\nabla h_{\epsilon}(\hat{X}_{t}^{x})\mathcal{R}_{v_{1},v_{2}}^{x}(t)\mathcal{I}_{v_{3}}^{x}(t)\big]-\mathbb{E}\big[\nabla h_{\epsilon}(\hat{X}_{t}^{x})D_{U_{3}}\mathcal{R}_{v_{1},v_{2}}^{x}(t)\big].
\end{align*}
Hence, by Lebesgue's dominated convergence theorem, we have
\begin{align*}
|\nabla_{v_{3}}\nabla_{v_{2}}\nabla_{v_{1}}\mathbb{E}[h(\hat{X}_{t}^{x})]|=&|\lim_{\epsilon\rightarrow0}\nabla_{v_{3}}\nabla_{v_{2}}\nabla_{v_{1}}\mathbb{E}[h_{\epsilon}(\hat{X}_{t}^{x})]|\\
\leq&\mathbb{E}|\nabla_{v_{1}}\hat{X}_{t}^{x}\mathcal{I}_{v_{2}}^{x}(t)\mathcal{I}_{v_{3}}^{x}(t)|+\mathbb{E}|D_{U_{3}}\nabla_{v_{1}}\hat{X}_{t}^{x}\mathcal{I}_{v_{2}}^{x}(t)|\\
&+\mathbb{E}|\nabla_{v_{1}}\hat{X}_{t}^{x}D_{U_{3}}\mathcal{I}_{v_{2}}^{x}(t)|+\mathbb{E}|\nabla_{v_{3}}\nabla_{v_{1}}\hat{X}_{t}^{x}\mathcal{I}_{v_{2}}^{x}(t)|\\
&+\mathbb{E}|\nabla_{v_{1}}\hat{X}_{t}^{x}\nabla_{v_{3}}\mathcal{I}_{v_{2}}^{x}(t)|+\mathbb{E}|\mathcal{R}_{v_{1},v_{2}}^{x}(t)\mathcal{I}_{v_{3}}^{x}(t)|\\
&+\mathbb{E}|D_{U_{3}}\mathcal{R}_{v_{1},v_{2}}^{x}(t)|+\mathbb{E}|\nabla_{v_{3}}\mathcal{R}_{v_{1},v_{2}}^{x}(t)|.
\end{align*}
Then, the desired result follows from the Cauchy-Schwarz inequality, (\ref{grad3}), (\ref{gradsec}), (\ref{malgrad1}), (\ref{1})-(\ref{3}) and (\ref{third1})-(\ref{third3}).
\qed

\section{Proofs of Lemmas in Subsection \ref{EMPROOF}}\label{stable application}

\subsection{Proof of Lemma \ref{EMmoment}}

Define $V(y):=(1+|y|^{2})^{1/2}$, then
\begin{gather*}
    \nabla V(y)=\frac{y}{V(y)},
    \quad \nabla^{2}V(y)=\frac{V(y)^2I_{d}-yy^{T}}
    {V(y)^3},
\end{gather*}
where $I_d$ denotes the $d \times d$ identity matrix. Hence, for any $y\in\mathbb{R}^{d}$
\begin{gather*}
    |\nabla V(y)|\leq1,\quad
    \|\nabla^{2}V(y)\|_{\mathrm{HS}}\leq 2\sqrt{d}.
\end{gather*}
Notice that $|y|\leq V(y)$ and by (\ref{density}), for any $k\geq 0$, we have
\begin{align*}
V(\tilde{Y}_{k+1})=&V\left(\tilde{Y}_{k}-\frac{\eta}{\alpha}\tilde{Y}_{k}+\frac{\eta^{\frac{1}{\alpha}}}{\sigma}\tilde{Z}_{k+1}\right)\\
=&V\left(\tilde{Y}_{k}-\frac{\eta}{\alpha}\tilde{Y}_{k}\right)+V\left(\tilde{Y}_{k}-\frac{\eta}{\alpha}\tilde{Y}_{k}+\frac{\eta^{\frac{1}{\alpha}}}{\sigma}\tilde{Z}_{k+1}\right)
-V\left(\tilde{Y}_{k}-\frac{\eta}{\alpha}\tilde{Y}_{k}\right)\\
=&V\left(\tilde{Y}_{k}\right)-\int_{0}^{\frac{\eta}{\alpha}}\langle\nabla V(\tilde{Y}_{k}-r\tilde{Y}_{k}),\tilde{Y}_{k}\rangle\dif r\\
&+\int_{0}^{\frac{\eta^{\frac{1}{\alpha}}}{\sigma}}\langle\nabla V\left(\tilde{Y}_{k}-\frac{\eta}{\alpha}\tilde{Y}_{k}+r\tilde{Z}_{k+1}\right),\tilde{Z}_{k+1}\rangle\dif r.
\end{align*}
Recall $\nabla V(y)=\frac{y}{V(y)}$, we have
\begin{align*}
-\int_{0}^{\frac{\eta}{\alpha}}\langle\nabla V(\tilde{Y}_{k}-r\tilde{Y}_{k}),\tilde{Y}_{k}\rangle\dif r
=&-\int_{0}^{\frac{\eta}{\alpha}}\frac{\langle\tilde{Y}_{k}-r\tilde{Y}_{k},\tilde{Y}_{k}\rangle}{V(\tilde{Y}_{k}-r\tilde{Y}_{k})}\dif r\\
\leq&-\int_{0}^{\frac{\eta}{\alpha}}\frac{(1-r)|\tilde{Y}_{k}|^{2}}{(|\tilde{Y}_{k}|^{2}+1)^{\frac{1}{2}}}\dif r\\
=&-\int_{0}^{\frac{\eta}{\alpha}}(1-r)V\left(\tilde{Y}_{k}\right)\dif r+\int_{0}^{\frac{\eta}{\alpha}}\frac{1-r}{(|\tilde{Y}_{k}^{2}+1)^{\frac{1}{2}}}\dif r\\
\leq&-\left(\frac{\eta}{\alpha}-\frac{\eta^{2}}{2\alpha^{2}}\right)V\left(\tilde{Y}_{k}\right)+\frac{\eta}{\alpha}\leq-\frac{\eta}{2\alpha}V\left(\tilde{Y}_{k}\right)+\frac{\eta}{\alpha},
\end{align*}
where the last inequality is by the fact $\eta\leq1<\alpha$. In addition, notice that $\tilde{Z}_{k+1}$ is independent of $\tilde{Y}_{k}$, conditioned on $\tilde{Y}_{k}=y$, we have
\begin{align*}
&\mathbb{E}\left[\langle\nabla V\left(\left(1-\frac{\eta}{\alpha}\right)y+r\tilde{Z}_{k+1}\right),\tilde{Z}_{k+1}\rangle\right]\\
=&\frac{\alpha}{V(\mathbb{S}^{d-1})}\int_{|z|\geq1}\frac{\langle\nabla V\left(\left(1-\frac{\eta}{\alpha}\right)y+rz\right),z\rangle-\langle\nabla V\left(\left(1-\frac{\eta}{\alpha}\right)y\right),z\rangle{\bf 1}_{(0,\eta^{-\frac{1}{\alpha}})}(|z|)}{|z|^{\alpha+d}}\dif z\\
=&\frac{\alpha}{V(\mathbb{S}^{d-1})}\int_{1\leq|z|\leq\eta^{-\frac{1}{\alpha}}}\int_{0}^{r}\frac{\langle\nabla^{2} V\left(\left(1-\frac{\eta}{\alpha}\right)y+sz\right),zz^{T}\rangle_{{\rm HS}}}{|z|^{\alpha+d}}\dif s\dif z\\
&+\frac{\alpha}{V(\mathbb{S}^{d-1})}\int_{|z|\geq\eta^{-\frac{1}{\alpha}}}\frac{\langle\nabla V\left(\left(1-\frac{\eta}{\alpha}\right)y+rz\right),z\rangle}{|z|^{\alpha+d}}\dif z,
\end{align*}
which implies
\begin{align*}
&\left|\mathbb{E}\left[\langle\nabla V\left(\left(1-\frac{\eta}{\alpha}\right)y+r\tilde{Z}_{k+1}\right),\tilde{Z}_{k+1}\rangle\right]\right|\\
\leq&\frac{\alpha}{V(\mathbb{S}^{d-1})}\mathbb{E}\left[\int_{|z|\leq\eta^{-\frac{1}{\alpha}}}\int_{0}^{r}\frac{2\sqrt{d}|z|^{2}}{|z|^{\alpha+d}}\dif z
+\int_{|z|\geq\eta^{-\frac{1}{\alpha}}}\frac{|z|}{|z|^{\alpha+d}}\dif z\right]
\leq\frac{2\alpha\sqrt{d}}{2-\alpha}r\eta^{\frac{\alpha-2}{\alpha}}+\frac{\alpha}{\alpha-1}\eta^{\frac{\alpha-1}{\alpha}}.
\end{align*}
Hence, we have
\begin{align*}
&\left|\mathbb{E}\left[\int_{0}^{\frac{\eta^{1/\alpha}}{\sigma}}\langle\nabla V\left(\tilde{Y}_{k}-\frac{\eta}{\alpha}\tilde{Y}_{k}+r\tilde{Z}_{k+1}\right),\tilde{Z}_{k+1}\rangle\dif r\right]\right|\\
\leq&\int_{0}^{\frac{\eta^{1/\alpha}}{\sigma}}\left(\frac{2\alpha\sqrt{d}}{2-\alpha}r\eta^{\frac{\alpha-2}{\alpha}}+\frac{\alpha}{\alpha-1}\eta^{\frac{\alpha-1}{\alpha}}\right)\dif r=\left(\frac{\alpha\sqrt{d}}{(2-\alpha)\sigma^{2}}+\frac{\alpha}{(\alpha-1)\sigma}\right)\eta.
\end{align*}

Therefore, we have
\begin{align*}
\mathbb{E}[V(\tilde{Y}_{k+1})]\leq\left(1-\frac{\eta}{2\alpha}\right)\mathbb{E}[V(\tilde{Y}_{k})]
+\left(\frac{\alpha\sqrt{d}}{(2-\alpha)\sigma^{2}}+\frac{\alpha}{(\alpha-1)\sigma}+\frac{1}{\alpha}\right)\eta,
\end{align*}
which immediately implies
\begin{align*}
\mathbb{E}[V(\tilde{Y}_{k+1})]\leq&\left(1-\frac{\eta}{2\alpha}\right)^{k+1}|x|
+\left(\frac{\alpha\sqrt{d}}{(2-\alpha)\sigma^{2}}+\frac{\alpha}{(\alpha-1)\sigma}+\frac{1}{\alpha}\right)\eta\sum_{j=0}^{k}\left(1-\frac{\eta}{2\alpha}\right)^{j}\\
\leq&|x|+\frac{2\alpha^{2}\sqrt{d}}{(2-\alpha)\sigma^{2}}+\frac{2\alpha^{2}}{(\alpha-1)\sigma}+2.
\end{align*}
Hence, we have
\begin{align*}
\mathbb{E}|\tilde{Y}_{k}^{x}|\leq\mathbb{E}[V(\tilde{Y}_{k}^{x})]\leq C_{\alpha,d}(1+|x|).
\end{align*}
The proof is complete.
\qed

\subsection{Proof of Lemma \ref{SDEM}}
Let $\mathcal{L}^{\alpha}$ be the generator that corresponds to the process $\tilde{X}_{t}.$ Then, it is easy to check that for any $f\in\mathcal{C}_{b}^{2}(\mathbb{R}^{d},\mathbb{R}^{d}),$
\begin{align*}
\mathcal{L}^{\alpha}f(x)=-\frac{1}{\alpha}\langle x,\nabla f(x)\rangle+\Delta^{\alpha/2}f(x).
\end{align*}
Following \cite[Section 2]{Wan13}, we know that the extended domain of the operator $\mathcal{L}^{\alpha}$ is given as follows:
\begin{align*}
\mathcal{D}(\mathcal{L}^{\alpha}):\big\{f\in\mathcal{C}^{2}(\mathbb{R}^{d},\mathbb{R}^{d}):\int_{|z|\geq1}\frac{f(x+z)-f(x)}{|z|^{\alpha+d}}dz<\infty\big\}.
\end{align*}
Taking a function $V\in\mathcal{C}^{2}(\mathbb{R}^{d},\mathbb{R}^{d})$ such that $V\geq1$, for $|x|\leq1,$ $V(x)$ is bounded and for $|x|>1,$ $V(x)=|x|.$ It is easy to check that $V\in\mathcal{D}(\mathcal{L}^{\alpha}),$ and $\mathcal{L}^{\alpha}V$ is a well defined locally measurable function.
Then, there exist $c_{2},c_{3}>0$ and a compact set $A$ such that for all $x\in\mathbb{R}^{d},$
\begin{align*}
\mathcal{L}^{\alpha}V(x)\leq-c_{2}V(x)+c_{3}{\bf 1}_{A}(x).
\end{align*}
This along with \cite[Theorem 6.1]{MT93} yields that the process $(\tilde{X}_{t}^{x})_{t\geq0}$ is exponential ergodic, i.e., there exists a unique invariant probability $\mu$ such that for all $x\in\mathbb{R}^{d}$ and $t>0,$
\begin{align*}
\sup_{|f|\leq V+1}\big|\mathbb{E}[f(X_{t}^{x})]-\mu(f)\big|\leq C_{\alpha,d}(1+V(x))e^{-c_{4}t}
\end{align*}
for some constant $c_{4}>0$ and $\mu(V)<\infty$ (here $\mu(V)=\mathbb{E}[V(Z_{1})]$, see, e.g., \cite{CNXY19}). These further imply
\begin{align*}
\mathbb{E}|\tilde{X}_{t}^{x}|\leq C_{\alpha,d}(1+|x|).
\end{align*}

Recall $\tilde{X}_{t}^{x}=x-\frac{1}{\alpha}\int_{0}^{t}\tilde{X}_{r}^{x}dr+Z_{t},$ by (\ref{SDEmoment1}), we further have
\begin{align*}
\mathbb{E}|X_{t}^{x}-x|\leq&\frac{1}{\alpha}\int_{0}^{t}\mathbb{E}|\tilde{X}_{r}^{x}|dr+\mathbb{E}|Z_{t}|\leq C_{\alpha,d}(1+|x|)(t+t^{\frac{1}{\alpha}}).
\end{align*}
\qed

\subsection{Proof of Lemma \ref{compare}}

Recall (\ref{SDE}) and (\ref{E-M}), we have
\begin{align*}
\mathbb{E}[f(\tilde{X}_{\eta}^{x})-f(\tilde{Y}_{1})]=\mathbb{E}\big[f\big(x-\frac{1}{\alpha}\int_{0}^{\eta}\tilde{X}_{r}^{x}dr+Z_{\eta}\big)-f\big(x-\frac{\eta}{\alpha}x+\frac{\eta^{\frac{1}{\alpha}}}{\sigma}\tilde{Z}\big)\big]:=\mathcal{J}_{1}+\mathcal{J}_{2},
\end{align*}
where
\begin{align*}
\mathcal{J}_{1}:=\mathbb{E}\big[f\big(x-\frac{1}{\alpha}\int_{0}^{\eta}\tilde{X}_{r}^{x}dr+Z_{\eta}\big)-f\big(x-\frac{\eta}{\alpha}x+Z_{\eta}\big)\big],
\end{align*}
\begin{align*}
\mathcal{J}_{2}:=\mathbb{E}\big[f\big(x-\frac{\eta}{\alpha}x+Z_{\eta}\big)-f\big(x-\frac{\eta}{\alpha}x\big)\big]-\mathbb{E}\big[f\big(x-\frac{\eta}{\alpha}x+\frac{\eta^{\frac{1}{\alpha}}}{\sigma}\tilde{Z}\big)-f\big(x-\frac{\eta}{\alpha}x\big)\big].
\end{align*}
For $\mathcal{J}_{1},$ by Lemma \ref{SDEM}, we have
\begin{align*}
|\mathcal{J}_{1}|\leq\frac{\|\nabla f\|}{\alpha}\int_{0}^{\eta}\mathbb{E}|\tilde{X}_{r}^{x}-x|\dif r\leq C_{\alpha,d}(1+|x|)\|\nabla f\|\int_{0}^{\eta}r^{\frac{1}{\alpha}}dr\leq C_{\alpha,d}(1+|x|)\|\nabla f\|\eta^{\frac{1}{\alpha}}.
\end{align*}
For $\mathcal{J}_{2},$ by It$\hat{o}$'s formula, we have
\begin{align*}
\mathbb{E}\big[f\big(x+\eta b(x)+Z_{\eta}\big)-f\big(x+\eta b(x)\big)\big]=\int_{0}^{\eta}\mathbb{E}\big[\Delta^{\alpha/2}f\big(x+\eta b(x)+Z_{r}\big)\big]dr,
\end{align*}
and noting that $d_{\alpha}=\frac{\alpha}{V(\mathbb{S}^{d-1})\sigma^{\alpha}},$ by Taylor expansion, we have
\begin{align*}
&\mathbb{E}\big[f\big(x-\frac{\eta}{\alpha}x+\frac{\eta^{\frac{1}{\alpha}}}{\sigma}\tilde{Z}\big)-f\big(x-\frac{\eta}{\alpha}x\big)\big]\\
=&\frac{\eta^{\frac{1}{\alpha}}}{\sigma}\mathbb{E}\big[\int_{0}^{1}\langle\nabla f\big(x-\frac{\eta}{\alpha}x+\frac{\eta^{\frac{1}{\alpha}}}{\sigma}t\tilde{Z}\big),\tilde{Z}\rangle dt\big]\\
=&\frac{\eta^{\frac{1}{\alpha}}}{\sigma}\int_{|z|\geq1}\int_{0}^{1}\frac{\alpha \langle\nabla f\big(x-\frac{\eta}{\alpha}x+\frac{\eta^{\frac{1}{\alpha}}}{\sigma}tz\big),z\rangle}{V(\mathbb{S}^{d-1})|z|^{\alpha+d}}dtdz\\
=&\frac{\alpha\eta}{V(\mathbb{S}^{d-1})\sigma^{\alpha}}\int_{|z|\geq\frac{\eta^{\frac{1}{\alpha}}}{\sigma}}\int_{0}^{1}\frac{\langle f\big(x-\frac{\eta}{\alpha}x+tz\big),z\rangle}{|z|^{\alpha+d}}dtdz
=\eta\Delta^{\alpha/2}f(x-\frac{\eta}{\alpha}x)-\mathcal{R}
\end{align*}
with
\begin{align*}
\mathcal{R}=\eta d_{\alpha}\int_{|z|<\frac{\eta^{\frac{1}{\alpha}}}{\sigma}}\int_{0}^{1}\frac{\langle\nabla f\big(x-\frac{\eta}{\alpha}x+tz\big),z\rangle}{|z|^{\alpha+d}}dtdz.
\end{align*}
These imply
\begin{align*}
|\mathcal{J}_{2}|\leq|\mathcal{R}|+\big|\int_{0}^{\eta}\mathbb{E}\big[\Delta^{\alpha/2}f\big(x-\frac{\eta}{\alpha}x+Z_{r}\big)\big]dr-\eta\Delta^{\alpha/2}f\big(x-\frac{\eta}{\alpha}x\big)\big|.
\end{align*}
It is easy to check that
\begin{align*}
|\mathcal{R}|=&\eta d_{\alpha}\big|\int_{|z|<\frac{\eta^{\frac{1}{\alpha}}}{\sigma}}\int_{0}^{1}\frac{\langle\nabla f\big(x-\frac{\eta}{\alpha}x+tz\big)-\nabla f\big(x-\frac{\eta}{\alpha}x\big),z\rangle}{|z|^{\alpha+d}}dtdz\big|\\
\leq&\eta d_{\alpha}\int_{|z|<\frac{\eta^{\frac{1}{\alpha}}}{\sigma}}\int_{0}^{1}\frac{\big|\nabla f\big(x-\frac{\eta}{\alpha}x+tz\big)-\nabla f\big(x-\frac{\eta}{\alpha}x\big)\big|}{|z|^{\alpha+d-1}}dtdz\\
\leq&C_{\alpha,d}\eta\|\nabla^{2}f\|_{{\rm HS}}\int_{|z|<\frac{\eta^{\frac{1}{\alpha}}}{\sigma}}\frac{1}{|z|^{\alpha+d-2}}dz\leq C_{\alpha,d}\|\nabla^{2}f\|_{{\rm HS}}\eta^{\frac{2}{\alpha}}.
\end{align*}
By \cite[(2.23)]{CNXY19}, that is,
\begin{equation*}
\left|
(\Delta^{\frac{\alpha}{2}}f)(x) - (\Delta^{\frac{\alpha}{2}}f)(y)
\right|\leq
\frac{2d_{\alpha}V(\mathbb{S}^{d-1})\|\nabla^{2}f\|_{{\rm
HS}}}{\alpha(2-\alpha)(\alpha-1)}|x-y|^{2-\alpha},
\end{equation*}
we further have
\begin{align*}
&\big|\int_{0}^{\eta}\mathbb{E}\big[\Delta^{\alpha/2}f\big(x-\frac{\eta}{\alpha}x+Z_{r}\big)\big]dr-\eta\Delta^{\alpha/2}f\big(x-\frac{\eta}{\alpha}x\big)\big|\\
\leq&\int_{0}^{\eta}\mathbb{E}\big|\Delta^{\alpha/2}f\big(x-\frac{\eta}{\alpha}x+Z_{r}\big)\big]-\eta\Delta^{\alpha/2}f\big(x-\frac{\eta}{\alpha}x\big)\big|dr\\
\leq&C_{\alpha,d}\|\nabla^{2}f\|_{{\rm HS}}\int_{0}^{\eta}\mathbb{E}|Z_{r}|^{2-\alpha}dr
=C_{\alpha,d}\|\nabla^{2}f\|_{{\rm HS}}\int_{0}^{\eta}\mathbb{E}|Z_{1}|^{2-\alpha}r^{\frac{2-\alpha}{\alpha}}dr\leq C_{\alpha,d}\|\nabla^{2}f\|_{{\rm HS}}\eta^{\frac{2}{\alpha}}.
\end{align*}
Therefore, we have
\begin{align*}
\big|\mathbb{E}[f(\tilde{X}_{\eta}^{x})-f(\tilde{Y}_{1})]\big|\leq C_{\alpha,d}(1+|x|)(\|\nabla f\|+\|\nabla^{2}f\|_{{\rm HS}})\eta^{\frac{2}{\alpha}}.
\end{align*}
\qed

\subsection{Proof of Lemma \ref{Qregular}}

Let $p(t,x)$ be the transition probability density of rotationally symmetric $\alpha$-stable process
$(Z_{t})_{t\geq0}$ and the following heat kernel estimates is well known (see, e.g., \cite[Lemma 5]{BJ07}), that is,
\begin{align}\label{bounds}
|\nabla p(t,x)|\leq C_{\alpha,d}t^{-\frac{1}{\alpha}}\frac{t}{(t^{\frac{1}{\alpha}}+|x|)^{\alpha+d}}.
\end{align}

Recall the SDE (\ref{SDE}),
\begin{align*}
d\tilde{X}_{t}=-\frac{1}{\alpha}\tilde{X}_{t}dt+dZ_{t},\quad X_{0}=x\in\mathbb{R}^{d}.
\end{align*}
Such an equation can be solved explicitly
\begin{align}\label{solutionstable}
\tilde{X}_{t}^{x}=xe^{-\frac{t}{\alpha}}+\int_{0}^{t}e^{-\frac{t-s}{\alpha}}\dif Z_{s},
\end{align}
see \cite[p.105]{Sat99}. It follows from (\ref{solutionstable}) that the density of $(Q_{t})_{t\geq0}$ is given by
\begin{align*}
q(t,x,y)=p(1-e^{-t},y-e^{-\frac{t}{\alpha}}x),
\end{align*}
which further implies that for any $h\in{\rm Lip}(1)$ and $x\in\mathbb{R}^{d}$, we have
\begin{align*}
Q_{t}h(x)=\int_{\mathbb{R}^{d}}p(1-e^{-t},y-e^{-\frac{t}{\alpha}}x)h(y)\dif y.
\end{align*}

Now, we are at the position to prove the Lemma \ref{Qregular}.\\
{\bf {\it Proof of Lemma \ref{Qregular}.}}
For any $x\in\mathbb{R}^{d}$ and $t>0$, by integration by parts, we have
\begin{align*}
\nabla(Q_{t}h)(x)=&\int_{\mathbb{R}^{d}}\nabla^{x}p(1-e^{-t},y-e^{-\frac{t}{\alpha}}x)h(y)\dif y\\
=&-e^{-\frac{t}{\alpha}}\int_{\mathbb{R}^{d}}\nabla^{y}p(1-e^{-t},y-e^{-\frac{t}{\alpha}}x)h(y)\dif y\\
=&e^{-\frac{t}{\alpha}}\int_{\mathbb{R}^{d}}p(1-e^{-t},y-e^{-\frac{t}{\alpha}}x)\nabla h(y)\dif y,
\end{align*}
where $\nabla^{x}$ means that the gradient operator acts on $x$. Then, we have
\begin{align*}
|\nabla(Q_{t}h)(x)|\leq\|\nabla h\|e^{-\frac{t}{\alpha}}\int_{\mathbb{R}^{d}}p(1-e^{-t},y-e^{-\frac{t}{\alpha}}x)\dif y= \|\nabla h\|e^{-\frac{t}{\alpha}}.
\end{align*}
Furthermore, by (\ref{bounds}), we have
\begin{align*}
\|\nabla(Q_{t}h)^{2}(x)\|_{{\rm{HS}}}\leq&\int_{\mathbb{R}^{d}}e^{-\frac{2t}{\alpha}}|\nabla^{y}p(1-e^{-t},y-e^{-\frac{t}{\alpha}}x)||\nabla h(y)|\dif y\\
\leq&C_{\alpha,d}\|\nabla h\|(1-e^{-t})^{-\frac{1}{\alpha}}e^{-\frac{2t}{\alpha}}\int_{\mathbb{R}^{d}}\frac{1-e^{-t}}{\big((1-e^{-t})^{\frac{1}{\alpha}}+|y-e^{-\frac{t}{\alpha}}x|\big)^{\alpha+d}}\dif y\\
\leq&C_{\alpha,d}\|\nabla h\|(1-e^{-t})^{-\frac{1}{\alpha}}e^{-\frac{2t}{\alpha}}\\
=&C_{\alpha,d}\|\nabla h\|(e^{t}-1)^{-\frac{1}{\alpha}}e^{-\frac{t}{\alpha}}\leq C_{\alpha,d}\|\nabla h\| t^{-\frac{1}{\alpha}}e^{-\frac{t}{\alpha}}.
\end{align*}
\qed

\section{Proof of Lemma \ref{normal}}\label{normal spplication}

In this section, we use the semigroup of $(B_{t}^{x})_{t\geq0}$ and the formula of integration by parts to prove Lemma
\ref{normal}.

{\bf Proof of Lemma \ref{normal}.}
Recall
\begin{align*}
P_{t}h(x)=\mathbb{E}h(B_{t}^{x})=\frac{1}{(2\pi
t)^{\frac{d}{2}}}\int_{\mathbb{R}^{d}}e^{-\frac{|y-x|^{2}}{2t}}h(y)dy.
\end{align*}
For any $v,x_{1},x_{2}\in\mathbb{R}^{d}$ and $f\in\mathcal{C}^{1}(\mathbb{R}^{d}\times\mathbb{R}^{d},\mathbb{R}),$
denote the directional derivative of $f(x_{1},x_{2})$ with respect to $x_{i}$ by
$\nabla_{v}^{x_{i}}f(x_{1},x_{2}),$ $i=1,2,$ respectively. Then, we have
\begin{align}\label{condition1}
\nabla_{v_{1}}P_{t}h(x)=&\frac{1}{(2\pi
t)^{\frac{d}{2}}}\int_{\mathbb{R}^{d}}\nabla_{v_{1}}^{x}e^{-\frac{|y-x|^{2}}{2t}}h(y)dy
=\frac{1}{(2\pi t)^{\frac{d}{2}}}\int_{\mathbb{R}^{d}}e^{-\frac{|y-x|^{2}}{2t}}\langle\frac{y-x}{t},v_{1}\rangle
h(y)dy,
\end{align}
which implies
\begin{align}\label{condition2}
\nabla_{v_{2}}\nabla_{v_{1}}P_{t}h(x)=&\frac{1}{(2\pi
t)^{\frac{d}{2}}}\int_{\mathbb{R}^{d}}\nabla_{v_{2}}^{x}(e^{-\frac{|y-x|^{2}}{2t}}\langle\frac{y-x}{t},v_{1}\rangle)h(y)dy\nonumber\\
=&\frac{1}{(2\pi
t)^{\frac{d}{2}}}\int_{\mathbb{R}^{d}}e^{-\frac{|y-x|^{2}}{2t}}\big(\langle\frac{y-x}{t},v_{2}\rangle\langle\frac{y-x}{t},v_{1}\rangle-\frac{1}{t}\langle
v_{2},v_{1}\rangle\big)h(y)dy.
\end{align}
Hence, by integration by parts, we have
\begin{align*}
\nabla_{v_{3}}\nabla_{v_{2}}\nabla_{v_{1}}P_{t}h(x)=&\frac{1}{(2\pi
t)^{\frac{d}{2}}}\int_{\mathbb{R}^{d}}\nabla_{v_{3}}^{x}\big[e^{-\frac{|y-x|^{2}}{2t}}\big(\langle\frac{y-x}{t},v_{2}\rangle\langle\frac{y-x}{t},v_{1}\rangle-\frac{1}{t}\langle
v_{2},v_{1}\rangle\big)\big]h(y)dy\\
=&\frac{-1}{(2\pi
t)^{\frac{d}{2}}}\int_{\mathbb{R}^{d}}\nabla_{v_{3}}^{y}\big[e^{-\frac{|y-x|^{2}}{2t}}\big(\langle\frac{y-x}{t},v_{2}\rangle\langle\frac{y-x}{t},v_{1}\rangle-\frac{1}{t}\langle
v_{2},v_{1}\rangle\big)\big]h(y)dy\\
=&\frac{1}{(2\pi
t)^{\frac{d}{2}}}\int_{\mathbb{R}^{d}}e^{-\frac{|y-x|^{2}}{2t}}\big(\langle\frac{y-x}{t},v_{2}\rangle\langle\frac{y-x}{t},v_{1}\rangle-\frac{1}{t}\langle
v_{2},v_{1}\rangle\big)\nabla_{v_{3}}h(y)dy\\
=&\mathbb{E}\Big[\big(\langle\frac{B_{t}}{t},v_{2}\rangle\langle\frac{B_{t}}{t},v_{1}\rangle-\frac{1}{t}\langle
v_{2},v_{1}\rangle\big)\langle\nabla h(B_{t}^{x}),v_{3}\rangle\Big],
\end{align*}
then, by Cauchy-Schwarz inequality, we have
\begin{align*}
|\nabla_{v_{3}}\nabla_{v_{2}}\nabla_{v_{1}}P_{t}h(x)|\leq&\|\nabla h\||v_{3}|\Big[\frac{1}{t^{2}}\mathbb{E}|\langle
B_{t},v_{2}\rangle\langle B_{t},v_{1}\rangle|+\frac{1}{t}|\langle v_{2},v_{1}\rangle|\Big]\\
\leq&|v_{3}|\big[\frac{1}{t}|v_{1}||v_{2}|+\frac{1}{t}|v_{2}||v_{1}|\big]=\frac{2}{t}|v_{1}||v_{2}||v_{3}|,
\end{align*}
(\ref{Bp3}) is proved.

Next, noticing that
	\begin{align*}
	\Delta P_{t}h(x)=\langle\nabla^{2}P_{t}h(x),I_{d}\rangle
	\end{align*}
and $I_{d}=\mathbb{E}[WW^{T}]$ with $W\sim N(0,I_{d})$, it follows from (\ref{Bp3}) that
	\begin{align*}	
\big|\Delta(P_{t}h)(x+v)-\Delta(P_{t}h)(x)\big|=&\big|\langle\nabla^{2}(P_{t}h)(x+v)-\nabla^{2}(P_{t}h)(x),I_{d}\rangle_{\rm{HS}}\big|\leq\frac{2d}{t}|v|.
	\end{align*}
\qed
\end{appendix}

\begin{acks}[Acknowledgments]
The authors would like to gratefully thank Jim Dai for very helpful discussions on probability approximations. We are grateful to the referees whose constructive comments and suggestions have helped to greatly improve the quality of this paper.
\end{acks}

\begin{funding}
Shao Q.M. is partially supported by National Nature Science Foundation
of China NSFC 12031005, Shenzhen Outstanding Talents Training Fund. Xu L. is partially supported by NSFC No. 12071499, Macao S.A.R grant FDCT  0090/2019/A2 and University of Macau grant  MYRG2018-00133-FST.
\end{funding}




\end{document}